\documentclass[a4paper,11pt,fleqn,leqno]{article}

\usepackage{latexsym,amssymb,amsmath,amsthm,epsfig,color}
\usepackage[utf8]{inputenc}
\usepackage[english]{babel}
\usepackage{hyperref}

\makeatletter
\@addtoreset{equation}{section}
\makeatother

\makeatletter
\@addtoreset{enunciato}{section}

\makeatother

\newcounter{enunciato}[section]

\newtheorem{ittheorem}{Theorem}
\newtheorem{itlemma}{Lemma}
\newtheorem{itproposition}{Proposition}
\newtheorem{itdefinition}{Definition}
\newtheorem{itremark}{Remark}
\newtheorem{itconjecture}{Conjecture}
\newtheorem{itcorollary}{Corollary}

\newtheorem{xx}{\bf xxx}

\newtheorem{zz}{\bf zzz}

\newenvironment{theorem}{\addtocounter{enunciato}{1}
 
\begin{ittheorem}}{\end{ittheorem}}

\newenvironment{lemma}{\addtocounter{enunciato}{1}
\begin{itlemma}}{\end{itlemma}}

\newenvironment{proposition}{\addtocounter{enunciato}{1}
\begin{itproposition}}{\end{itproposition}}

\newenvironment{definition}{\addtocounter{enunciato}{1}
\begin{itdefinition}}{\end{itdefinition}}

\newenvironment{remark}{\addtocounter{enunciato}{1}
\begin{itremark}}{\end{itremark}}

\newenvironment{conjecture}{\addtocounter{enunciato}{1}
\begin{itconjecture}}{\end{itconjecture}}

\newenvironment{corollary}{\addtocounter{enunciato}{1}
\begin{itcorollary}}{\end{itcorollary}}

\newcommand{\be}[1]{\begin{equation}\label{#1}}
\newcommand{\ee}{\end{equation}}

\newcommand{\bea}[1]{\begin{eqnarray}\label{#1}}
\newcommand{\eea}{\end{eqnarray}}

\newcommand{\bt}[1]{\begin{theorem}\label{#1}}
\newcommand{\et}{\end{theorem}}

\newcommand{\bl}[1]{\begin{lemma}\label{#1}}
\newcommand{\el}{\end{lemma}}

\newcommand{\bp}[1]{\begin{proposition}\label{#1}}
\newcommand{\ep}{\end{proposition}}

\newcommand{\bd}[1]{\begin{definition}\label{#1}}
\newcommand{\ed}{\end{definition}}

\newcommand{\br}[1]{\begin{remark}\label{#1}}
\newcommand{\er}{\end{remark}}

\newcommand{\bcj}[1]{\begin{conjecture}\label{#1}}
\newcommand{\ecj}{\end{conjecture}}

\newcommand{\bcor}[1]{\begin{corollary}\label{#1}}
\newcommand{\ecor}{\end{corollary}}

\newcommand{\bpr}{\begin{proof}}
\newcommand{\epr}{\end{proof}}

\newcommand{\T}{\mathbb{T}}
\newcommand{\wt}{\widetilde}

\newcommand{\dd}{\mathrm{d}}
\newcommand{\eee}{\mathrm{e}}
\newcommand{\bP}{\mathbb{P}}
\newcommand{\bE}{\mathbb{E}}
\DeclareMathOperator{\var}{Var}

\newcommand{\suml}{\sum\limits}

\def\CB{\mathcal{B}}
\def\CC{\mathcal{C}}

\def\CM{\mathcal{M}}

\def\CE{\mathcal{E}}
\def\CF{\mathcal{F}}

\def\CL{\mathcal{L}}

\def\CP{\mathcal{P}}

\def\B{\mathbb{B}}

\def\E{\mathbb{E}}

\def\N{\mathbb{N}}

\def\R{\mathbb{R}}

\def\Z{\mathbb{Z}}

\newcommand{\capa}{\mathrm{cap}}

\newcommand{\Ntwo}{\N \backslash \{1\}}

\newcommand{\uc}{\underline{c}}
\newcommand{\ux}{\underline{x}}
\newcommand{\ud}{\underline{d}}

\newcommand{\uL}{\underline{\Lambda}}
\newcommand{\ul}{\underline{\lambda}}

\allowdisplaybreaks

\begin{document}

\title{The hierarchical Cannings process in random environment}

\author{A. Greven$^1$, F. den Hollander$^2$,  A. Klimovsky$^3$}

\date{{\today}}

\maketitle

\begin{abstract}
In an earlier paper, we introduced and studied a system of hierarchically interacting measure-valued 
random processes that arises as the continuum limit of a large population of individuals carrying 
different types. Individuals live in colonies labelled by the hierarchical group of order $N$, and are 
subject to \emph{migration} and \emph{resampling} on all hierarchical scales simultaneously. The 
resampling mechanism is such that a random positive fraction of the population in a block of colonies 
inherits the type of a random single individual in that block, which is why we refer to our system as 
the hierarchical Cannings process. Before resampling in a block takes place, all individuals in that 
block are relocated uniformly, which we call \emph{reshuffling}.

In the present paper, we study a version of the hierarchical Cannings process in \emph{random 
environment}, namely, the resampling measures controlling the change of type of individuals in 
different blocks are chosen randomly with a given mean and are kept fixed in time, i.e., we work 
in the \emph{quenched} setting. We give a necessary and sufficient condition under which a 
multi-type equilibrium is approached (= coexistence) as opposed to a mono-type equilibrium 
(= clustering). Moreover, in the hierarchical mean-field limit $N \to \infty$, with the help of a 
\emph{renormalization analysis} we obtain a full picture of the space-time scaling behaviour of 
block averages on all hierarchical scales simultaneously. We show that the $k$-block averages 
are distributed as the superposition of a Fleming-Viot diffusion with a deterministic volatility 
constant $d_k$ and a Cannings process with a random jump rate, both depending on $k$. In 
the random environment $d_k$ turns out to be smaller than in the homogeneous environment 
of the same mean. We investigate how $d_k$ scales with $k$. This leads to five \emph{universality 
classes of cluster formation} in the mono-type regime. We find that if clustering occurs, then the 
random environment slows down the growth of the clusters, i.e., enhances the diversity of types. 
In some universality classes the growth of the clusters depends on the realisation of the random 
environment.



\medskip\noindent 
\emph{Acknowledgements:} 
AG was supported by the Deutsche Forschungsgemeinschaft  (grant DFG-GR 876/15-2), 
FdH was supported by the European Research Council (Advanced Grant VARIS-267356) 
and by the Netherlands Organization for Scientific Research (Gravitation Grant 
NETWORKS-024.002.003), AK was supported by the Netherlands Organization for 
Scientific Research (grant 613.000.913). The authors are grateful to Evgeny Verbitskiy 
for help with the renormalization analysis. 
\end{abstract}

\vspace{0.3cm}

\footnoterule
\noindent
\hspace*{0.3cm} {\footnotesize$^{1)}$
Department Mathematik, Universit\"at Erlangen-N\"urnberg, Cauerstrasse 11, 
D-91058 Erlangen, Germany\\
greven@mi.uni-erlangen.de}\\
\hspace*{0.3cm} {\footnotesize$^{2)}$
Mathematisch Instituut, Universiteit Leiden, P.O.\ Box 9512, NL-2300RA  Leiden, 
The Netherlands\\
denholla@math.leidenuniv.nl}\\
\hspace*{0.3cm} {\footnotesize$^{3)}$
Fakult\"at f\"ur Mathematik, Universit\"at Duisburg-Essen, Thea-Leymann-Strasse 9, 
D-45127 Essen, Germany\\
ak@aklimovsky.net}

\tableofcontents

\section{Introduction}
\label{s.intro}

\subsection{Motivation and goal}
\label{ss.motivation}

Two models play a central role in the world of stochastic multi-type population dynamics:
\begin{itemize} 
\item[(1)] 
The Moran model and its limit for large populations, the Fleming-Viot measure-valued 
\emph{diffusion}.
\item[(2)] 
The Cannings model and its limit for large populations, the Cannings measure-valued 
\emph{jump process} (also called the generalized Fleming-Viot process).
\end{itemize} 
The Cannings model accounts for situations in which \emph{resampling} is such that a 
random positive fraction of the population in the next generation inherits the type of a 
random single individual in the current generation, even in the infinite population limit 
(see Cannings~\cite{C74}, \cite{C75}). In order to describe a setting where this effect 
has a geographical structure, i.e., where \emph{migration} of individuals is allowed as 
well, different models have been proposed in Limic and Sturm~\cite{LS06}, Blath, 
Etheridge and Meredith~\cite{BEM07}, Barton, Etheridge and V\'eber~\cite{BEV10}, 
Berestycki, Etheridge and V\'eber~\cite{BEV13}, and Greven, den Hollander, Kliem and 
Klimovsky~\cite{GHKK14}. The behaviour of these models has been studied in detail 
and its dependence on the geographic space is fairly well understood. 

The type space is typically chosen to be a compact Polish space $E$. In~\cite{GHKK14}, we focused 
on the case where the geographic space is the hierarchical group $\Omega_N$  of order $N$, since 
this allowed us to carry out a full \emph{renormalization analysis}. In the \emph{hierarchical mean-field 
limit} $N \to \infty$, the migration can be chosen in such a way that it approximates migration on the 
geographic space $\Z^2$, a possibility that was exploited by Sawyer and Felsenstein~\cite{SF83} 
(see also Dawson, Gorostiza and Wakolbinger~\cite{DGW04}).

We analyze the model introduced in \cite{GHKK14}, but add the effect that the Cannings resampling 
mechanism is controlled by \emph{catastrophic events} on a small time scale, for which it is appropriate 
to assume that the rate of occurrence has a spatially inhomogeneous structure. This leads us to 
consider spatial Cannings models with block resampling in \emph{random environment}, i.e., both 
the form and the overall rate of the block resampling mechanism depend on the geographic location. 

\begin{remark}\label{r.325}
{\rm In a catastrophic event, a part of the population is killed in a large spatial area and is subsequently 
replenished via a rapid recolonization, resulting in a bottleneck effect consisting of compression and 
subsequent expansion of the descendants of a single ancestor. The mechanisms behind such events 
are functions of the background environment, which is inhomogeneous in space but constant in time. 
It would be interesting to derive our continuum model (defined in Section~\ref{ss.random}) from an 
individual-based model with \textit{two time scales}: the catastrophic events happen on a fast time 
scale, while the migration and resampling happen on a slow time scale. Moreover, in our individual-based 
model we do \emph{reshuffling} before resampling, which must be motivated likewise. Carrying out 
the details of such a derivation would merit a paper in its own right.} \hfill $\square$
\end{remark}

The {\em goal} of the present paper is three-fold: 
\begin{itemize}
\item[(1)]
\emph{Construction} of the hierarchical Cannings process in random environment via a well-posed 
martingale problem and derivation of a \emph{duality relation} with a hierarchical spatial coalescent 
in random environment.
\item[(2)]
Analysis of the longtime behaviour, in particular, the \emph{dichotomy} between a multi-type equilibrium 
and a mono-type equilibrium.
\item[(3)]
Scaling analysis of a collection of r\emph{enormalized processes} obtained by looking at the evolution 
of blocks averages on successive space-time scales in the hierarchical mean-field limit and the 
consequences for universality classes of the mono-type cluster formation.
\end{itemize}
We are particularly interested in \emph{new effects caused by the random environment}.

The mathematical tools we will exploit are the duality of the hierarchical Canning process in random 
environment with a hierarchical spatial coalescent in random environment, and the scaling of the block 
averages towards a mean-field process in random environment called the McKean-Vlasov process. 
This in turn will lead us to study two independent hierarchical random walks in the same random 
environment, and to analyze the orbit of iterations of non-linear transformations arising from random 
M\"obius transformations that link the behaviour on successive hierarchical scales.

\subsection{Summary of the main results}
\label{ss.summary}

In an earlier paper, we introduced and studied a system of hierarchically interacting measure-valued 
random processes that arises as the continuum limit of a large population of individuals subject to 
migration, reshuffling and resampling \cite{GHKK14}. More precisely, individuals live in colonies labelled 
by $\Omega_N$, the hierarchical group of order $N$, and are subject to \emph{migration} based on a 
sequence of migration coefficients $\uc=(c_k)_{k\in\N_0}$ and to \emph{resampling} based on a 
sequence of resampling measures $\uL= (\Lambda_k)_{k\in\N_0}$, both acting on blocks of colonies 
(= macro-colonies) on all hierarchical scales $k\in\N_0$ simultaneously. The resampling mechanism 
is such that a random positive fraction of the population in a block of colonies inherits the type of a 
random single individual in that block, even in the infinite population limit, which is why we refer to 
our system as the hierarchical Cannings process. Before resampling in a block takes place, all individuals 
in that block are relocated uniformly. This relocation is called \emph{reshuffling} and means that 
resampling is done in a locally ``panmictic'' manner.

In the present paper, we study a version of the hierarchical Cannings process in \emph{random 
environment}, namely, the resampling measures in different blocks are chosen randomly with 
mean $\uL$ and are kept fixed in time, i.e., we consider the \emph{quenched} version of the 
system. We construct the hierarchical Cannings process in random environment via a well-posed 
martingale problem, and establish duality with a system of coalescing hierarchical random walks 
with block coalescence in random environment. We study the long-time behaviour of the process, 
in particular, we give a necessary and sufficient condition on $\uc$ and $\uL$ under which 
almost sure convergence to a \emph{multi-type equilibrium} occurs (= coexistence), as opposed 
to a \emph{mono-type equilibrium} (= clustering). The equilibrium depends on the environment, 
but the condition on $\uc$ and $\uL$ for its occurrence does not. 

To obtain more detailed information on the evolution of the system, we consider the \emph{hierarchical 
mean-field limit} $N \to \infty$. In this limit, with the help of a \emph{renormalization analysis}, we 
obtain a full picture of the space-time scaling behaviour on all hierarchical scales simultaneously. 
Our main result is that, on each hierarchical scale $k \in \N_0$, the $k$-block averages on time 
scale $N^k$ converge to a random process that is a superposition of a Cannings process with a 
resampling measure equal to the associated $k$-block resampling measure (which depends on 
the environment) and an additional Fleming-Viot process with volatility $d_k$, reflecting the 
macroscopic impact of the lower-order resampling and of the drift of strength $c_k$ towards 
the limiting $(k+1)$-block average (which is constant on the limiting time scale). It turns out that 
$d_k$ is a function of $c_l$ and $\Lambda_l$ for all $0 \leq l < k$, and of the law of the random 
environment. Thus, \emph{it is through the volatility that the renormalization manifests itself}. 

We show that \emph{the random environment makes the system less volatile}, i.e., $d_k$ is strictly 
smaller than its corresponding value for the homogenous system where the resampling measures 
are replaced by their mean. We investigate how $d_k$ scales as $k \to \infty$, which leads to 
\emph{various different cases} depending on the choice of $\uc$ and $\uL$. We find that if migration 
and resampling occur with comparable rates on all hierarchical scales, then the lower volatility 
persists in the limit as $k\to\infty$. The renormalization transformation connecting $d_{k+1}$ to 
$d_k$ turns out to be a non-linear transformation arising from a \emph{random M\"obius transformation}. 
The scaling behaviour of the iterates of these transformations is studied in detail. We find that if 
clustering occurs, then the random environment slows down the growth of the clusters, i.e., 
enhances the diversity of types. We find five \emph{universality classes of cluster formation} in 
the regime of clustering. These are linked to the different cases of scaling behaviour of $d_k$.
We find that if the growth of the clusters is \emph{rapid}, then the rate of growth depends on 
the realisation of the environment, while if the growth is \emph{slow}, then the effect of the 
environment averages out. The latter happens e.g.\ in the \emph{critical regime} where the 
system is barely clustering. 

\subsection{Outline}
\label{ss.outline}

Sections~\ref{s.model}--\ref{s.randomwalk} deal with the model for finite $N$, while 
Sections~\ref{s.mkvrand}--\ref{s.clustering} deal with the hierarchical mean-field limit 
$N\to\infty$. In Section~\ref{s.model} we define the hierarchical Cannings process and 
its dual. In Section~\ref{s.results} we state our main theorems and summarize the effects 
of the random environment. Section~\ref{s.dual} contains the proof of existence and 
uniqueness of the hierarchical Cannings process and its dual, and establishes convergence 
to an equilibrium. Section~\ref{s.randomwalk} proves the dichotomy between coexistence 
(multi-type equilibrium) versus clustering (mono-type equilibrium), and provides the parameter 
range for both. Section~\ref{s.mkvrand} contains the multi-scale analysis for the evolution 
of block averages on successive space-time scales in the hierarchical mean-field limit, proves 
the dichotomy in that limit, and identifies the renormalization transformations connecting the 
successive scales. Section~\ref{s.completeproof} analyzes the orbit of the iterations of 
these transformations and identifies various different cases for the scaling of the volatility of 
the block averages. Section~\ref{s.clustering} links these cases to the universality classes 
of cluster formation.

\section{The model}
\label{s.model}

In this section, we define the hierarchical Cannings process in random environment and construct 
its dual: a spatial coalescent in random environment. We begin in Section~\ref{ss.reviewC} by 
recalling the process without random environment introduced in~\cite{GHKK14}. In 
Section~\ref{ss.random} we explain how the random environment is added. 

\subsection{The hierarchical Cannings process}
\label{ss.reviewC}

In Sections~\ref{sss.hg}--\ref{sss.resh}, we recall the definition of the hierarchical Canning 
process given in ~\cite{GHKK14}. In Section~\ref{sss.hierarCan} we add the random environment 
and indicate how the definition needs to be modified.

\subsubsection{The hierarchical group of order $N$}
\label{sss.hg}

The {\em hierarchical group $\Omega_N$} of order $N$ is the set
\be{ag30_a}
\Omega_N = \Big\{\eta=(\eta^l)_{l\in\N_0} \in\{0,1,\ldots, N-1\}^{\N_0}
\colon\, \sum_{l\in\N_0} \eta^l < \infty\Big\}, \qquad N\in\N\backslash\{1\},
\ee
endowed with the addition operation $+$ defined by $(\eta+\zeta)^l=\eta^l+\zeta^l \textrm{ (mod $N$)}$, 
$l\in\N_0$. In other words, $\Omega_N$ is the direct sum of the cyclical group of order $N$ (a fact 
that is important for the application of Fourier analysis). The group $\Omega_N$ is equipped with 
the ultrametric distance $d_{\Omega_N}(\cdot,\cdot)$ defined by
\be{ag31}
d_{\Omega_N}(\eta,\zeta)=d_{\Omega_N}(0,\eta-\zeta)
= \min\{k\in\N_0 \colon\, \eta^l=\zeta^l \,\, \forall\, l \geq k\},
\qquad \eta,\zeta\in\Omega_N.
\ee
Let
\be{block-definition} 
B_k(\eta) = \{\zeta\in\Omega_N\colon\, d_{\Omega_N}(\eta,\zeta) \leq k\},
\qquad \eta\in\Omega_N,\,k\in\N_0
\ee
denote the $k$-block around $\eta$ (i.e., the ball of hierarchical radius $k$ around $\eta$), which 
we think of as a {\em macro-colony}. The geometry of $\Omega_N$ is explained in Fig.~\ref{fig-hierargr}.

\begin{figure}[htbp]
\centering
\includegraphics[width=\textwidth]{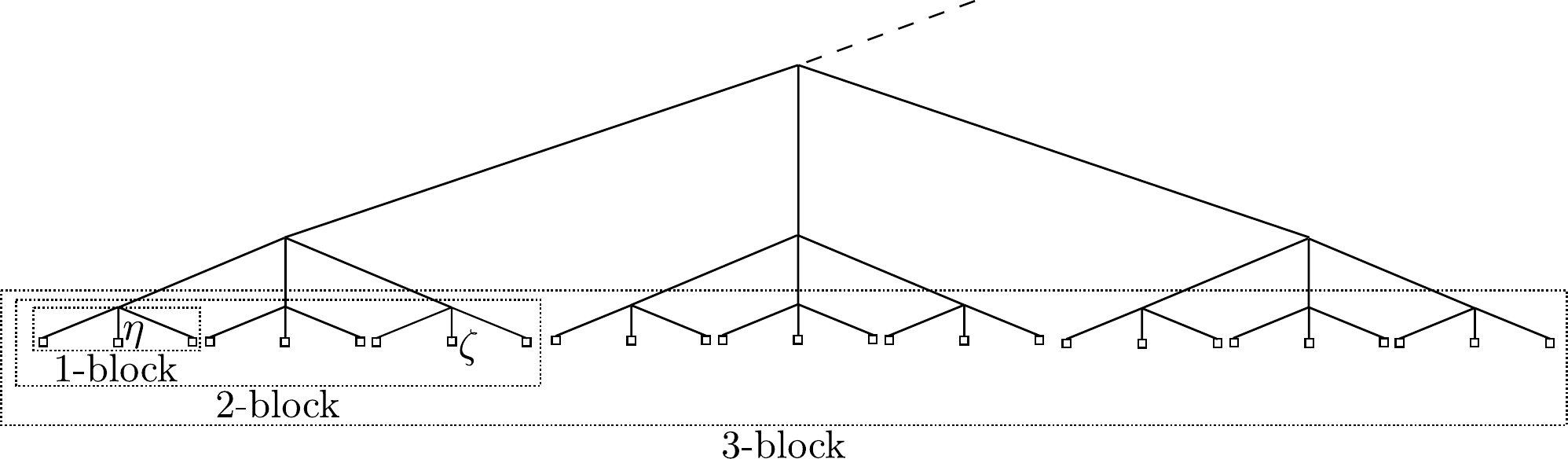}
\caption{\small Close-ups of a 1-block, a 2-block and a 3-block in the hierarchical group of order $N=3$. 
The elements of the group are the leaves of the tree (indicated by $\Box$'s). The hierarchical distance 
between two elements is the graph distance to the most recent common ancestor: $d_{\Omega_3}
(\eta,\zeta) = 2$ for $\eta$ and $\zeta$ in the picture.}
\label{fig-hierargr}
\end{figure}

In what follows, we consider a system of individuals organized in colonies labelled by $\Omega_N$. 
Initially each colony has $M$ individuals, each carrying a type drawn from a Polish type space $E$ 
that is compact. Subsequently, individuals are subject to block migration (Section~\ref{sss.mg}) 
and block reshuffling-resampling (Section~\ref{sss.resh}). In the continuum-mass limit $M\to\infty$, 
the evolution converges to the hierarchical Cannings process (Section~\ref{sss.hierarCan}).

\subsubsection{Block migration}
\label{sss.mg}

We introduce migration on $\Omega_N$ through a random walk kernel. 
For that purpose, we introduce a sequence of migration rates
\be{ckdef}
\uc = (c_k)_{k\in\N_0} \in (0,\infty)^{\N_0},
\ee
and we let the individuals \emph{migrate} as follows:
\begin{itemize}
\item
Each individual, for every $k\in\N$, chooses at rate $c_{k-1}/N^{k-1}$ the block of radius $k$ around 
its present location and jumps to a location chosen uniformly at random in that block.
\end{itemize}
The transition kernel of the random walk thus performed by the individuals is
\be{32b}
a^{(N)}(\eta,\zeta) = \sum_{k \geq d_{\Omega_N}(\eta,\zeta)} \frac{c_{k-1}}{N^{2k-1}},
\qquad \eta,\zeta\in\Omega_N,\,\eta\neq\zeta, \qquad a^{(N)}(\eta,\eta)=0.
\ee

\begin{remark}
\label{degreeretr}
{\rm The behaviour of the random walk in \eqref{32b} is known in great detail. Dawson, Gorostiza 
and Wakolbinger~\cite{DGW05} showed that it is recurrent if and only if $\sum_{k\in\N_0} (1/c_k) 
= \infty$. They introduced the concept of \emph{degree of recurrence/transience} $\gamma_N$
\cite[Definition 2.1.1]{DGW05}, which in the special case where $c_k=c^k$ equals $\gamma(N)
= \log c/\log(N/c)$. Note that
\begin{equation}
\gamma(N) \left\{\begin{array}{ll}
<0, &\quad c<1 \text{ (strongly recurrent)},\\ 
=0, &\quad c=1 \text{ (critically recurrent)},\\ 
>0, &\quad c>1 \text{ (transient)}.
\end{array}
\right.
\end{equation}
This is the same as for simple random walk on $\Z^d$ with Hausdorff dimension $d=d(N)=(2\log N)/
\log(N/c)$ (when we allow for a continuum of dimensions). In particular, $d=d(N)=2$ for $c=1$.}
\end{remark}

Throughout the paper, we assume that
\be{ak:recurrence-cond}
\limsup_{k\to\infty} \tfrac{1}{k} \log c_k < \log N.
\ee
This guarantees that the total migration rate per individual is finite.

\subsubsection{Block reshuffling-resampling}
\label{sss.resh}

The idea of the Cannings resampling mechanism is to allow reproduction with an offspring that 
is of a size comparable to the whole population. Since we have introduced a spatial structure, 
we now allow, on all hierarchical levels $k$ simultaneously, a reproduction event where each 
individual treats the $k$-block around its present location as a \emph{macro-colony} and uses 
it for its resampling. More precisely, we choose a sequence of resampling measures 
\be{ag41}
\uL = \big(\Lambda_k)_{k\in\N_0} \in \CM_f([0,1])^{\N_0},
\ee
where $\CM_f([0,1])$ denotes the set of finite non-negative measures on $[0,1]$, satisfying
\be{ag}
\Lambda_0(\{0\})=0, \qquad \int_{(0,1]} \frac{\Lambda_0(\dd r)}{r} = \infty,
\ee
and
\begin{equation}
\label{ak1000}
\Lambda_k(\{0\})=0, \qquad \int_{(0,1]} \frac{\Lambda_k(\dd r)}{r^2}  < \infty. \qquad k \in \N,
\end{equation}

Let $\Lambda^*_k(\dd r) = \Lambda_k(\dd r)/r^2$, $r \in (0,1]$. Set
\be{lambda-total-masses}
\lambda_k = \Lambda_k((0,1]), \qquad \lambda^*_k = \Lambda^*_k((0,1]), \qquad k\in\N_0,
\ee
and assume that
\be{lambdakdef}
\ul = (\lambda_k)_{k\in\N_0} \in (0,\infty)^{\N_0}.
\ee
We let individuals \emph{reshuffle-resample} by carrying out the following two steps at once:
\begin{itemize}
\item
For every $\eta\in\Omega_N$ and $k\in\N_0$, choose the block $B_k(\eta)$ at rate $1/N^{2k}$.
\item
First, each individual in $B_k(\eta)$ independently is moved to a uniformly random location in 
$B_k(\eta)$, i.e., a reshuffling takes place (see Fig.~\ref{fig-reshuffle}). After that, $r$ is drawn 
according to the intensity measure $\Lambda^*_k$ and $a$ is drawn according to the current 
type distribution in $B_k(\eta)$, and each of the individuals in $B_k(\eta)$ independently is 
replaced by an individual of type $a$ with probability $r$.
\end{itemize}
Note that the reshuffling-resampling affects all the individuals in a macro-colony simultaneously 
and in the same manner. The reshuffling-resampling occurs at all levels $k\in\N_0$, at a rate that 
is fastest in single colonies and gets slower as the level $k$ of the macro-colony increases.
\footnote{Because the reshuffling is done first, the resampling always acts on a uniformly 
distributed state (``panmictic resampling''). Reshuffling is a parallel update affecting all individuals 
in a macro-colony simultaneously. Therefore it cannot be seen as a migration of individuals equipped 
with independent clocks.}

\begin{figure}[htbp]
\label{fig:reshuffling} 
\centering
\vspace{0.2cm}
\includegraphics[width=\textwidth]{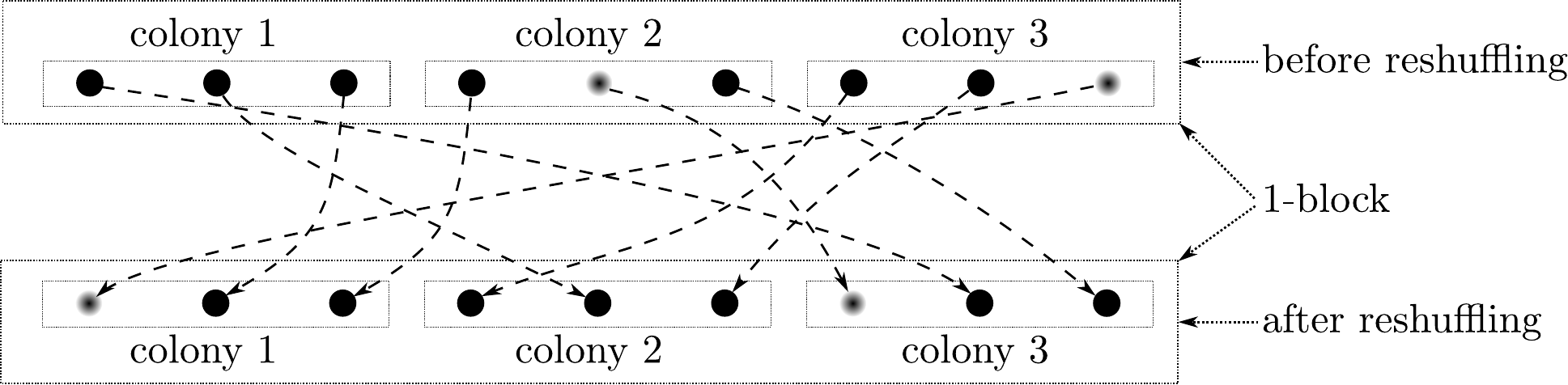} 
\caption{\small
Random reshuffling in a 1-block on the hierarchical lattice of order $N=3$, with $M=3$ individuals 
of two types (full circles and fuzzy circles) per colony. \emph{Note}: Typically a random reshuffling 
does not preserve the number of individuals per colony, but in the example drawn here it does.}
\label{fig-reshuffle}
\end{figure}

The first conditions in (\ref{ag}) and (\ref{ak1000}) make the resampling a \emph{jump process}. 
Later we will add in diffusion by hand. The second condition in (\ref{ag}) guarantees that the population 
has a well-defined genealogy and that after a positive finite time most of the population at a site 
descends from a finite number of ancestors (see Pitman~\cite{P99}). The second condition in 
(\ref{ak1000}) is needed to guarantee that in finite time a macro-colony is affected by finitely many 
reshuffling-resampling events, otherwise the resampling cannot be properly defined. 

Throughout the paper, we assume that
\be{ak:lambda-growth-condition}
\limsup_{k\to\infty} \tfrac{1}{k}\,\log \lambda^*_k < \log N.
\ee
Note that each of the $N^k$ colonies in a $k$-block can trigger reshuffling-resampling in that block, 
and for each colony the block is chosen at rate $N^{-2k}$. Therefore, \eqref{ak:lambda-growth-condition} 
guarantees that the total resampling rate per individual is bounded.

\subsubsection{The generator and the martingale problem}
\label{sss.hierarCan}

We are now ready to formally define the hierarchical Cannings process in terms of a martingale 
problem. The process arises as the continuum-mass limit of the individual-based model described 
in Sections~\ref{sss.hg}--\ref{sss.resh}. Namely, in each colony of size $M$, instead of recording 
the \emph{numbers} of individuals of a given type we record the \emph{empirical distribution} of the 
types and pass to the limit $M\to\infty$.

Let $\CP(E)$ denote the set of probability measure on $E$ equipped with the topology of weak 
convergence. We equip the set $\CP(E)^{\Omega_N}$ with the product topology to get a state 
space that is Polish. Let $\mathcal{F} \subset C_\mathrm{b} \big(\CP(E)^{\Omega_N},\R\big)$ be 
the algebra of functions of the form
\be{ak:multi-level-test-functions}
\begin{aligned}
&F(x) = \int_{E^n} \left(\bigotimes_{m=1}^n x_{\eta_m}\big(\dd u^m\big)\right)
f\big(u^1,\ldots,u^n\big),
\quad x = (x_{\eta})_{\eta\in\Omega_N}\in\CP(E)^{\Omega_N},\\
&n \in \N, \quad f \in C_{\mathrm{b}}(E^n,\R), 
\quad \eta_1,\ldots,\eta_n \in \Omega_N.
\end{aligned}
\ee
The linear operator for the martingale problem
\be{ak:multi-level-generator}
L^{(\Omega_N)}\colon\,\mathcal{F} \to C_\mathrm{b}\big(\CP(E)^{\Omega_N},\R\big)
\ee
has two parts,
\be{ak:multi-level-generator-decomposition}
L^{(\Omega_N)}=L^{(\Omega_N)}_{\mathrm{mig}}+L^{(\Omega_N)}_{\mathrm{res}}.
\ee
The \emph{migration operator} is given by
\be{ak:multi-level-migration}
(L^{(\Omega_N)}_{\mathrm{mig}} F)(x)
= \sum_{\eta,\zeta \in \Omega_N} a^{(N)}(\eta,\zeta)
\int_E (x_{\zeta} - x_{\eta})(\dd a)\,\frac{\partial F(x)}{\partial x_\eta}[\delta_a]
\ee
and the \emph{reshuffling-resampling operator} by
\be{ak:multi-level-resampling-global}
\begin{aligned}
(L^{(\Omega_N)}_{\mathrm{res}} F)(x)
&= \sum_{\eta \in \Omega_N} \sum_{k\in\N_0} N^{-2k} \int_{(0,1]} \Lambda^*_k (\dd r) 
\int_E y_{\eta,k}(\dd a) \left[F\left(\Phi_{r,a,B_k(\eta)}(x)\right)-F(x)\right]\\
&\qquad + \sum_{\eta \in \Omega_N} (L^{d_0}_{\eta} F)(x),
\end{aligned}
\ee
where 
\be{ak:blockaverage}
y_{\eta,k} = N^{-k} \sum_{\zeta \in B_k(\eta)} x_\zeta 
\ee
is the $k$-block average of the components of $x$ in $B_k(\eta)$, $\Phi_{r,a,B_k(\eta)} \colon\,
\CP(E)^{\Omega_N}\to\CP(E)^{\Omega_N}$ is the \emph{reshuffling-resampling map} acting as
\be{ak:resampling-mapping-global}
\Big[\big(\Phi_{r,a,B_k(\eta)}\big)(x)\Big]_\zeta =
\begin{cases}
(1-r) y_{\eta,k} + r \delta_a, &\zeta \in B_k(\eta),\\
x_{\zeta}, &\zeta \in \Omega_N \backslash B_k(\eta),
\end{cases}
\quad r \in [0,1],\, a \in E,\, k\in\N_0,\,\eta\in\Omega_N,
\ee
and $L^{d_0}_{\eta}$ is the \emph{Fleming-Viot diffusion operator} with volatility $d_0 \geq 0$, 
acting on the colony $x_\eta$, given by
\be{ak:FlemingViot}
(L^{d_0}_{\eta}F)(x) = d_0 \int_E \int_E Q_{x_\eta}(\dd u,\dd v)\,
\frac{\partial^2 F(x)}{\partial x_\eta^2}[\delta_u,\delta_v]
\ee
with
\be{ak:flemming-viot-kernel}
Q_y(\dd u,\dd v) = y(\dd u)\,\delta_u(\dd v) - y(\dd u)\,y(\dd v), \qquad y \in \CP(E),
\ee
the Fleming-Viot diffusion coefficient, and
\be{ak-second-variation}
\frac{\partial^2 F(x)}{\partial x_\eta^2 }[\delta_u,\delta_v]
= \frac{\partial}{\partial x_\eta} \left( \frac{\partial F(x)}{\partial x_\eta} [\delta_u] \right)[\delta_v],
\quad u,v \in E.
\ee

\begin{remark}\label{r.711}
{\rm Note that the right-hand side of \eqref{ak:multi-level-resampling-global} is well-defined 
because of assumption \eqref{ak1000}. Indeed, by Taylor-expanding the inner integral in 
\eqref{ak:multi-level-resampling-global} in powers of $r$, we get 
\be{ak:9999}
\int_E y_{\eta,k}(\dd a)
\left[F\left(\Phi_{r,a,B_k(\eta)}(x)\right)-F(x)\right]
=  F(y_{\eta,k})-F(x)+O(r^2),
\qquad r \downarrow 0.
\ee
To have a well-defined resampling operator \eqref{ak:multi-level-resampling-global}, the expression 
in \eqref{ak:9999} must be integrable with respect to $\Lambda^*_k (\dd r)$, which is equivalent to 
assumption~\eqref{ak1000}.} \hfill $\square$
\end{remark}

\medskip
The following proposition was proved in \cite{GHKK14}.

\begin{proposition}[{\bf Hierarchical martingale problem}]
\label{P.vecLambda}
\mbox{}\\
For every $x \in \CP(E)^{\Omega_N}$, the martingale problem for $(L^{(\Omega_N)}, \CF,\delta_x)$ 
is well-posed. \footnote{As part of the definition of the martingale problem, we always require that 
the solution has c\`adl\`ag paths and is adapted to the natural filtration.} The unique solution is a 
strong Markov process with the Feller property. \hfill $\square$
\end{proposition}

\noindent
The Markov process arising as the solution of this martingale problem is denoted by
\be{XOmegaNdef}
X^{(\Omega_N)}=(X^{(\Omega_N)}(t))_{t\geq 0},
\ee 
and is referred to as the $C_N^{\uc,\uL}$-process on $\Omega_N$. Proposition~\ref{P.vecLambda} 
does not actually need the second condition in \eqref{ag}.

This condition will be needed only later.
 
\subsection{The hierarchical Cannings process in random environment}
\label{ss.random}

Our task in this section is to modify the first term in the right-hand side of 
\eqref{ak:multi-level-resampling-global} so as to include the effect of a random environment on 
the Cannings resampling mechanism. Section~\ref{sss.RE} defines the random environment, 
Section~\ref{sss.mp} the modified generator.

\subsubsection{The random environment on the full tree}
\label{sss.RE}

Recall that $\Omega_N$ is the set of \textit{leaves} of the tree in Fig.~\ref{fig-hierargr}. 
To introduce the random environment, we need to consider the \emph{full tree}, i.e.,
\be{fulltreedef}
\Omega^{\mathbb{T}}_N = \bigcup_{k\in\N_0} \Omega^{(k)}_N
\quad \text{ with } \quad \Omega^{(k)}_N = \Omega_N / B_k(0),
\ee
where $\Omega_N/B_k(0)$ denotes the quotient group of $\Omega_N$ modulo $B_k(0)$, which 
can be identified with the layer of the tree situated at height $k$ above the leaves. Indeed, because 
$d_{\Omega_N}$ is an ultrametric distance (recall \eqref{ag31}), for each $k\in\N_0$ the set 
$\Omega_N$ decomposes into disjoint balls of radius $k$, which can be labelled by the set 
$\Omega^{(k)}_N$. For $\xi \in \Omega^{\mathbb{T}}_N$, we write 
\be{heightdef}
|\xi| = \text{ the height of $\xi$ (counting from the leaves)},
\ee 
i.e., $|\xi| = k$ when $\xi \in \Omega^{(k)}_N$ for $k\in\N_0$, and we define
\be{Bkequivdef}
B_{|\xi|}(\xi)
\ee
to be the set of sites in $\Omega_N$ that lie below $\xi$ (see Fig.~\ref{fig-hierartree}). We can 
define the distance on the layer $\Omega^{(k)}_N$ as the graph distance to the most recent common 
ancestor, and the distance on the full tree $\Omega^{\mathbb{T}}_N$ as the \emph{largest} of the 
two graph distances to the most recent common ancestor (recall Fig.~\ref{fig-hierargr}).
The latter will be denoted by $d_{\Omega^{\mathbb{T}}_N}$. We write
\begin{equation}\label{e786}
\mathrm{MC}_k(\eta)
\end{equation}
to denote the {\em vertex in $\Omega^{\mathbb{T}}_N$ at height $k \in \N_0$ above $\eta \in \Omega_N$}. 
This site carries the rate for the random walk on $\Omega_N$ to become uniformly distributed on the 
$k-$ball around $\eta$. Since $\Omega^{\mathbb{T}}_N$ is isomorphic to $\Omega_N \times \N_0$, 
we sometimes write $\mathrm{MC}_k(\eta) = \xi = (\eta,k)$.

\begin{figure}[htbp]
\vspace{0.3cm}
\centering \includegraphics[width=\textwidth]{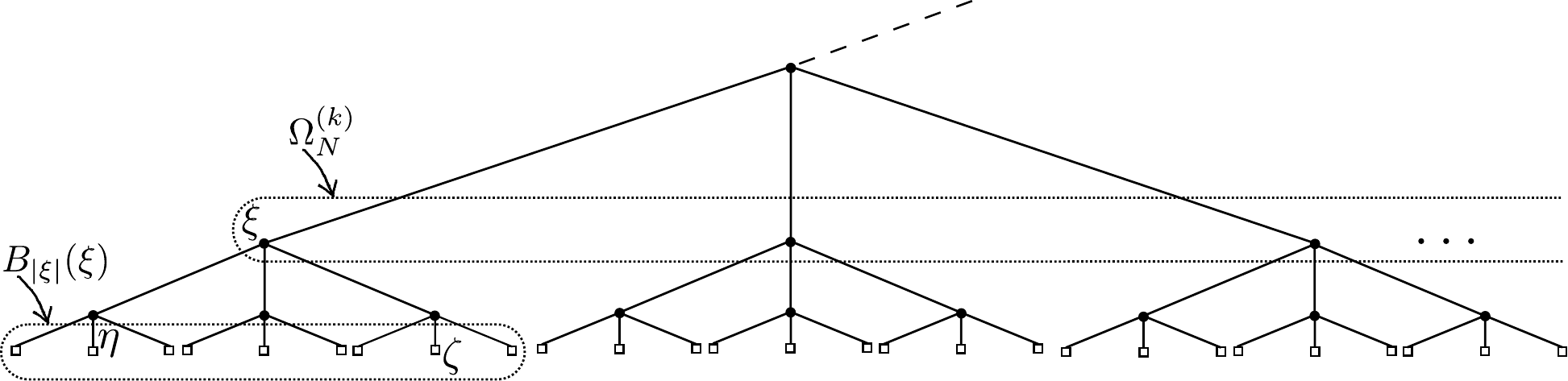} 
\caption{\small $\Omega^{\mathbb{T}}_N $ with $N=3$, $\xi \in \Omega^{\mathbb{T}}_N$
with $|\xi| = k = 2$, and $\eta, \zeta \in B_{|\xi|}(\xi)$. The elements of $\Omega^{\mathbb{T}}_N$ are 
the vertices of the tree (indicated by $\bullet$'s and $\Box$'s). The elements of $\Omega_N$ are the 
leaves of the tree (indicated by $\Box$'s).} 
\label{fig-hierartree}
\end{figure}
\medskip

We want to make the reshuffling-resampling \emph{spatially random}. To that end, we let
\be{ap1}
\underline{\Lambda}(\omega) = \big\{\Lambda^\xi(\omega)  \colon\, 
\xi \in \Omega^{\mathbb{T}}_N\big\}  
\ee 
be a random field of $\mathcal{M}_f([0,1])$-valued resampling measures indexed by the tree.
\begin{itemize}
\item
Throughout the paper, we \emph{use the symbol $\omega$ to denote the random environment and 
the symbol $\bP$ to denote the law of $\omega$}. 
\end{itemize}
In what follows, we assume that $\Lambda^\xi(\omega)$ is of the form
\be{ap2}
\Lambda^\xi(\omega) = \lambda_{|\xi|} \chi^\xi(\omega), 
\ee
where $\ul = (\lambda_k)_{k\in\N_0}$ is a deterministic sequence in $(0,\infty)$ (playing the role of 
modulation coefficients) and
\be{chiprop}
\{\chi^\xi(\omega)\colon\,\xi\in\Omega_N^{\mathbb{T}}\}
\ee
is a random field of $\mathcal{M}_f([0,1])$-valued resampling measures that is \emph{stationary under 
translations} in $\Omega^{\mathbb{T}}_N$ (i.e., translations sideways and up in Fig.~\ref{fig-hierartree}), 
and satisfies the conditions in \eqref{ag} when $\xi=0$ and the conditions in \eqref{ak1000} when 
$\xi\neq 0$. 

Abbreviate
\be{fdh:masschi}
\rho^\xi(\omega) = \chi^\xi(\omega)((0,1]),
\ee
which is the total mass of  $\chi^\xi(\omega)$. Clearly,
\be{rhoseq}
\{\rho^\xi(\omega)\colon\,\xi\in\Omega_N^{\mathbb{T}}\}
\ee
is a random field of $(0,\infty)$-valued total masses that is also stationary under translations in 
$\Omega^{\mathbb{T}}_N$. Throughout the paper, we assume that
\be{Aprop}
\bE[\rho^\xi(\omega)]=1, \qquad \bE[(\rho^\xi(\omega))^2] = C \in (0,\infty),
\ee
and that the sigma-algebra at infinity associated with \eqref{fdh:masschi}, defined by
\be{fdh:tail}
\mathcal{T} = \bigcap_{L \in \N_0} \mathcal{F}_L, \qquad
\mathcal{F}_L = \sigma\Big(\rho^\xi(\cdot) ~\Big|~
\xi \in \Omega^{\mathbb{T}}_N\colon\,d_{\Omega^{\mathbb{T}}_N}(0,\xi) \geq L\Big),
\ee
is trivial, i.e., all its events have probability 0 or 1 under the law $\bP$. For one of the theorems 
below we need to strengthen \eqref{Aprop} to
\be{Apropalt}
\bE[\rho^\xi(\omega)]=1, \qquad 
\exists\,\delta>0\colon\,\,\delta \leq \rho^\xi(\omega) \leq \delta^{-1}
\,\,\forall\,\xi\in\Omega_N \text{ for } \bP\text{-a.e. } \omega.
\ee

\subsubsection{The generator in random environment}
\label{sss.mp}

Throughout the sequel, we use the symbols $\eta,\zeta$ to denote elements of $\Omega_N$ and 
the symbol $\xi$ to denote elements of $\Omega^{\mathbb{T}}_N$. 

In random environment, we keep the definitions in Section~\ref{sss.hierarCan}, but we replace 
the \emph{reshuffling-resampling operator} $L^{(\Omega_N)}_{\mathrm{res}}$ in 
\eqref{ak:multi-level-resampling-global} by
\be{ak:multi-level-resampling-global-RE}
\begin{aligned}
(L^{(\Omega_N)}_{\mathrm{res}}(\omega) F)(x)
&= \sum_{\xi \in \Omega^\mathbb{T}_N} N^{-2|\xi|}
\int_{(0,1]} (\Lambda^\xi(\omega))^*(\dd r) \int_E y_{\xi}(\dd a)
\left[F\left(\Phi_{r,a,B_{|\xi|}(\xi)}(x)\right)-F(x)\right]\\
&\qquad + \sum_{\eta \in \Omega_N} (L_\eta^{d_0}F)(x)
\end{aligned}
\ee
with $(\Lambda^\xi(\omega))^*(\dd r) = \Lambda^\xi(\omega)(\dd r)/r^2$, $r \in (0,1]$, where 
$y_{\xi} \in \CP(E)$ is given by
\be{def-blocks-RE}
y_{\xi}=N^{-|\xi|} \sum_{\zeta\in B_{|\xi|}(\xi)} x_\zeta, \quad \xi \in \Omega^\mathbb{T}_N,
\ee
and $\Phi_{r,a,B_{|\xi|}(\xi)}\colon\,\CP(E)^{\Omega_N}\to\CP(E)^{\Omega_N}$ is the 
\emph{reshuffling-resampling map} acting as
\be{ak:resampling-mapping-global-RE}
\Big[\big(\Phi_{r,a,B_{|\xi|}(\xi)}\big)(x)\Big]_\zeta =
\begin{cases}
(1-r) y_{\xi} + r \delta_a, &\zeta \in B_{|\xi|}(\xi),\\
x_{\zeta}, &\zeta \in \Omega_N \backslash B_{|\xi|}(\xi),
\end{cases}
\qquad r \in [0,1],\, a \in E,\,\xi\in\Omega^\mathbb{T}_N.
\ee
The difference between \eqref{ak:multi-level-resampling-global} and 
\eqref{ak:multi-level-resampling-global-RE} is that the resampling in blocks occurs according to 
the resampling measure associated with the center of the block, labelled by $\Omega^\mathbb{T}_N$. 
The full generator is
\be{ak:resmig}
L^{(\Omega_N)}(\omega) = L^{(\Omega_N)}_{\mathrm{mig}} 
+ L^{(\Omega_N)}_{\mathrm{res}}(\omega)
\ee
with $L^{(\Omega_N)}_{\mathrm{mig}}$ the migration operator in \eqref{ak:multi-level-migration}.

\section{Main theorems}
\label{s.results}

In Section~\ref{ss.resN} we present results for fixed $N$, in Section~\ref{ss.hiermfl} for 
$N \to \infty$, the hierarchical mean-field limit. In Section~\ref{ss.effectre} we summarize the 
effects of the random environment. Throughout the paper, the environment $\omega$ is fixed 
and we use the symbol $\CL[W]$ to denote the law of a random variable $W$.

\subsection{Results for fixed $N$}
\label{ss.resN}

Section~\ref{sss.wposmart} establishes the well-posedness of the martingale problem, 
Section~\ref{sss.coexist} the convergence to an equilibrium that depends on $\omega$.

\subsubsection{Well-posedness of the martingale problem}
\label{sss.wposmart}

We begin by establishing that the martingale problem characterizes the process uniquely and 
specifies a strong Markov process.

\begin{theorem}[{\bf Well-posedness of the martingale problem}]
\label{T.wpbasic}
Fix $N\in\Ntwo$. For $\bP$-a.s.\ $\omega$ and every $x \in \CP(E)^{\Omega_N}$, the 
$(L^{(\Omega_N)}(\omega),\mathcal{F},\delta_x)$-martingale problem is well-posed. 
The unique solution is a strong Markov process with the Feller property. \hfill $\square$
\end{theorem}

The Markov process arising as the solution of the martingale problem is denoted by
\be{XOmegaNomegadef}
X^{(\Omega_N)}(\omega)=(X^{(\Omega_N)}(\omega;t))_{t\geq 0} 
= ((X_\eta(\omega,t))_{\eta \in \Omega_N})_{t \geq 0},
\ee 
and is referred to as the hierarchical Cannings process on $\Omega_N$ in the environment 
$\omega$. Theorem~\ref{T.wpbasic} does not actually need the second 
condition in \eqref{ag}. This condition will be needed only later.

\subsubsection{Dichotomy: coexistence versus clustering}
\label{sss.coexist}

We next show that the law of our process converges to a limit law that depends on $\omega$.

\begin{theorem}[{\bf Equilibrium}]
\label{T.ltbeq}
Fix $N \in \N\backslash\{1\}$. Suppose that, under the law $\bP$, the law of the initial state 
$X^{(\Omega_N)}(\omega;0)$ is stationary and ergodic under translations in $\Omega_N$, 
with mean single-coordinate measure 
\be{ext}
\theta = \bE\big[X_0^{(\Omega_N)} (\omega;0)\big] \in \CP(E).
\ee 
Then, for $\bP$-a.e.\ $\omega$, there exists an equilibrium measure $\nu^N_\theta(\omega) \in 
\CP(\CP(E)^{\Omega_N})$, arising as
\be{ag4.9}
\lim_{t\to\infty} \CL\big[X^{(\Omega_N)}(\omega;t)\big] = \nu^N_\theta(\omega),
\ee
satisfying
\be{ag4.2}
\int_{\CP(E)^{\Omega_N}} x_0\,\nu^N_\theta(\omega)(\dd x) = \theta.
\ee
Moreover, under the law $\bP$, $\nu^N_\theta(\omega)$ is stationary and ergodic under translations 
in $\Omega_N$. \hfill $\square$
\end{theorem}

\noindent
Note that $\nu^N_\theta(\omega)$ depends on $\omega$ even though its mean single-coordinate 
measure $\theta$ (which is determined by the initial state) does not. The proof of Theorem~\ref{T.ltbeq}
is based on a computation with the dual hierarchical Cannings process, which allows us to control 
second moments. As we will see in Section~\ref{s.randomwalk}, in random environment this computation 
is delicate because it involves two random walks in the same environment, and the \emph{difference} 
of these two random walks is \emph{not} a random walk itself, like in the average environment.

Using the stationarity and ergodicity of $\nu^N_\theta$, we next identify the parameter regime for 
which $\nu^N_\theta(\omega)$ is a \emph{multi-type equilibrium} (= coexistence given $\omega$), 
i.e.,
\be{ag4.8}
\sup_{f \in C_{\mathrm{b}}(E)} \int_{\CP(E)^{\Omega_N}} 
\nu^N_\theta(\omega)(\dd x) \int_E [f(u)-\theta]^2\,x_0(\dd u) > 0,
\ee
respectively, a \emph{mono-type equilibrium} (= clustering given $\omega$), i.e.,
\be{ag4.6}
\nu^N_\theta(\omega) = \int_E \delta_{(\delta_u)^{\Omega_N}} \theta(\dd u).
\ee
The two regimes are complementary. In the latter regime the system grows mono-type clusters that 
eventually cover all finite subsets of $\Omega_N$ (where types may or may not change infinitely 
often). 

\begin{theorem}[{\bf Dichotomy for finite $N$}]
\label{T.coexcritNfin}
Fix $N \in \Ntwo$ and assume \eqref{Apropalt}.
\begin{itemize}
\item[\textup{(a)}]
Let $\CC_N=\{\omega\colon\,\text{in $\omega$ coexistence occurs}\}$. Then 
$\bP(\CC_N) \in \{0,1\}$.
\item[\textup{(b)}] 
$\bP(\CC_N)=1$ if and only if
\be{Ndich}
\sum_{k\in\N_0} \frac{1}{c_k+N^{-1}\lambda_{k+1}} \sum_{l=0}^k \lambda_l<\infty. 
\ee
\end{itemize} 
\hfill $\square$
\et

\noindent
Cox and Klenke~\cite{CK00} give a criterion in the clustering regime for when the type at a given
site changes infinitely often. For interacting Fleming-Viot processes they show that this happens 
as soon as $\theta$ is not a $\delta$-measure. Because of the reasoning in Section~\ref{s.randomwalk}, 
we therefore get the following.

\begin{corollary}[{\bf Change of types}]
In the clustering regime, if $\theta \neq \delta_u$ for some $u \in E$, then at every site the type 
changes infinitely often.  \hfill $\square$ 
\end{corollary}

\subsection{Results for $N\to\infty$}
\label{ss.hiermfl}

Our remaining theorems capture the space-time scaling behaviour of our process in the 
hierarchical mean-field limit $N \to \infty$. In this limit, the degree of recurrence/transience 
$\gamma(N)$ tends to $0$ while the Hausdorff dimension $d(N)$ in the metric $\eta\mapsto 
e^{|\eta|}$ tends to $2$ (recall Remark~\ref{degreeretr}), so that our process becomes 
\emph{near-critical}.

In Section~\ref{sss.mfl} we introduce a key process, called the \emph{McKean-Vlasov process}, 
which naturally arises in this limit. In Section~\ref{sss.REinfty} we define the random environment 
for $N=\infty$. In Section~\ref{sss.blav} we look at the block averages on successive space-time 
scales and show that as $N\to\infty$ these converge to a sequence of McKean-Vlasov processes 
with \emph{renormalized volatilities}. In Section~\ref{sss.vola} we identify the scaling behaviour 
of the volatility on hierarchical scale $k$ in the limit as $k\to\infty$, which leads to various different 
cases as a function of $\uc$ and $\uL$. In Section~\ref{sss.coex} we identify the parameter regimes 
that correspond to coexistence, respectively, clustering. In Section~\ref{sss.clust} we link the different 
cases of scaling to five universality classes of cluster formation.
 
\subsubsection{McKean-Vlasov process}
\label{sss.mfl}

We need some definitions and basic facts about the McKean-Vlasov process from \cite{GHKK14}.

Let $\mathcal{F} \subseteq C_\mathrm{b}(\CP(E),\R)$ be the algebra of functions $F$ of the form
\be{ak:multi-level-test-functions*}
F(y) = \int_{E^n} y^{\otimes n}(\dd u)\,f(u),
\qquad y \in \CP(E),\, n \in \N,\,f \in C_{\mathrm{b}}(E^n,\R).
\ee
For $c,d \in [0,\infty)$, $\Lambda\in\CM_f([0,1])$ subject to (\ref{ag}) and $\theta\in\CP(E)$, let 
$L_\theta^{c,d,\Lambda}\colon\,\mathcal{F} \to C_\mathrm{b}(\CP(E),\R)$ be the linear operator
\be{generic-interaction-operator}
L_\theta^{c,d,\Lambda} = L^c_\theta + L^d + L^\Lambda
\ee
acting on $F\in \mathcal{F}$ as (recall \eqref{ak:flemming-viot-kernel})
\be{Ldefs}
\begin{aligned}
(L^c_\theta F)(y) &= c \int_E \left(\theta-y\right)(\dd a)\,
\frac{\partial F(y)}{\partial y}[\delta_a],\\
(L^dF)(y) &= d \int_E \int_E Q_y(\dd u,\dd v)\,
\frac{\partial^2 F(y)}{\partial y^2}[\delta_u,\delta_v],\\
(L^\Lambda F)(y) &= \int_{(0,1]} \Lambda^*(\dd r) \int_E y(\dd a)\,
\big[F\big((1-r)y+ r\delta_a\big)-F(y)\big].
\end{aligned}
\ee
The three parts of $L_\theta^{c,d,\Lambda}$ correspond to: (1) a \emph{drift} towards $\theta$ of 
strength $c$ (``immigration-emigration''); (2) a \emph{Fleming-Viot diffusion} with \emph{volatility} 
$d$ (``Moran resampling''); (3) a \emph{Cannings process} with \emph{resampling measure} 
$\Lambda$ (``Cannings resampling''). This model arises as the $M \to\infty$ limit of an individual-based 
model with $M$ individuals at a single site, with {\em immigration} at rate $c$ from a constant source 
with type distribution $\theta \in \CP(E)$, {\em emigration} at rate $c$ to a cemetery state, diffusive 
resampling at rate $d$, and $\Lambda$-resampling.
 
The following proposition was proved in \cite{GHKK14}.
  
\begin{proposition}[{\bf McKean-Vlasov martingale problem}]
\label{prop:McKean-Vlasov-well-posedness} $\mbox{}$
\begin{itemize}
\item[\textup{(a)}]
For every $y \in \CP(E)$, the martingale problem for $(L_\theta^{c,d,\Lambda}, \mathcal{F},\delta_y)$ 
is well-posed. The unique solution is a strong Markov process with the Feller property.
\item[\textup{(b)}]
For every $c \in (0,\infty)$, the solution from \textup{(a)} is ergodic in time with unique equilibrium 
measure $\nu_\theta^{c,d,\Lambda}$. For $c=0$, the solution from \textup{(a)} is not ergodic in 
time, and $\nu_\theta^{0,d,\Lambda}$ is the unique equilibrium measure obtained as the $t\to\infty$ 
limit with initial state $y=\theta$.
\item[\textup{(c)}] 
For $c > 0$,
\begin{equation}
\label{eq:equi-ident-c}
\nu_\theta^{c,d,\Lambda} = \nu_\theta^{1,d/c,\Lambda/c}.
\end{equation}
\end{itemize}
\hfill $\square$
\ep

Denote by
\be{Zdef}
Z_\theta^{c,d,\Lambda} = \big(Z_\theta^{c,d,\Lambda}(t)\big)_{t\geq 0},
\quad Z^{c,d,\Lambda}_\theta (0) = \theta,
\ee
the solution of the martingale problem in Proposition~\ref{prop:McKean-Vlasov-well-posedness} 
for the special choice $y=\theta$. This is called the \emph{McKean-Vlasov process} with  parameters 
$c,d,\Lambda$ and initial state $\theta$.

\subsubsection{Random environment for $N = \infty$}
\label{sss.REinfty}

In order to be able to pass to the limit $N\to\infty$, we need to define a random environment for 
$N=\infty$ in which all the random environments for finite $N$ are embedded. To that end, define 
$\Omega_\infty = \oplus_\N \N$, and let (recall \eqref{fulltreedef})
\be{e1096}
\Omega^{\mathbb{T}}_\infty = \bigcup_{k\in\N_0} \Omega^{(k)}_\infty
\quad \text{ with } \quad \Omega^{(k)}_\infty = \Omega_\infty / B_k(0).
\ee 
Note that for any $N \in \N$ there is a natural embedding of $\Omega^{\mathbb{T}}_N$ into 
$\Omega^{\mathbb{T}}_\infty$. Similarly as in Section~\ref{sss.RE}, we let
\be{reNinfinite}
\underline{\Lambda}(\omega) = \big\{\Lambda^\xi(\omega)  \colon\, 
\xi \in \Omega^{\mathbb{T}}_\infty \big\}  
\ee 
be a random field of $\mathcal{M}_f([0,1])$-valued resampling measures index by the full tree, where 
$\omega$ again denotes the random environment. We retain the symbol $\bP$ for the law of of 
$\omega$. As in \eqref{ap2}--\eqref{fdh:tail}, we assume that $\Lambda^\xi(\omega) = \lambda_{|\xi|} 
\chi^\xi(\omega)$ where, under the law $\bP$,  $\{\chi^\xi(\omega)\colon\,\xi \in\Omega^{\mathbb{T}}_\infty\}$ 
is stationary under translations in $\Omega^{\mathbb{T}}_\infty$, and is such that the total masses 
$\rho^\xi(\omega) = \chi^\xi(\omega)((0,1])$ have first moment equal to $1$, second moment finite, 
and a trivial sigma-algebra at infinity. For any $N \in \N$, the natural restriction of the random field in 
\eqref{reNinfinite} equals the random field in \eqref{ap1}.

\subsubsection{Renormalization via block averages}
\label{sss.blav}

For each $k \in \N_0$, we look at the \emph{$k$-block averages} defined by (recall Fig.~\ref{fig-hierargr}) 
\be{k-block-average}
Y_{\eta,k}^{(\Omega_N)}(\omega;t) = \frac{1}{N^k}
\sum_{\zeta\in B_k(\eta)} X_{\zeta}^{(\Omega_N)}(\omega;t), 
\qquad \eta \in \Omega_N,
\ee
which constitute a {\em renormalization of space} where the component $\eta$ is replaced by the average 
of the components in $B_k(\eta)$. After a corresponding {\em renormalization of time} where $t$ is 
replaced by $tN^k$, i.e., $t$ is the associated macroscopic time variable, we obtain a {\em renormalized} 
interacting system
\be{a3}
\left(\left(Y^{(\Omega_N)}_{\eta,k}(\omega;t N^k)
\right)_{\eta \in \Omega_N}\right)_{t \geq 0},
\qquad k \in \N_0,\,\eta \in \Omega_N,
\ee
which is constant in $B_k(\eta)$ and can be viewed as an interacting system indexed by the set 
$\Omega^{(k)}_N$ (see Fig.~\ref{fig-hierargr}). This provides us with a \emph{sequence of 
renormalized interacting systems}, which for fixed $N$ are {\em not} Markov.

The key ingredient to study the $ N \to \infty $ limit of \eqref{a3} is the following. 
Let $\ud=(d_k)_{k\in\N_0}$ be the sequence of \emph{volatility constants} defined recursively as
\be{diffusion-constants}
d_{k+1} = \bE_{\CL_\rho}\left[\frac{c_k(\mu_k\rho +d_k)}{c_k + (\mu_k\rho + d_k)}\right],
\quad k \in \N_0,
\ee
where 
\be{mulambdarel}
\mu_k=\tfrac12\lambda_k, \quad k\in\N_0,
\ee 
$\rho$ is the $(0,\infty)$-valued random variable whose law $\CL_\rho$ is the same as that of 
$\rho^0(\omega)$ under $\bP$ (recall (\ref{Aprop}--\ref{fdh:tail})), and $\bE_{\CL_\rho}$ is expectation 
w.r.t.\ $\CL_\rho$. For fixed $\uc$, $\uL$ and $d_0$, the recursion in \eqref{diffusion-constants} 
determines $\ud$. The right-hand side is the \emph{average of a random M\"obius transformation 
that depends on $\rho$}. Recall that $\rho$ has mean 1.
 
\paragraph{Heuristics behind the recursion formula for the volatilities.} 
In order to understand the recursion formula in \eqref{diffusion-constants}, we consider the 1-block 
around the origin $0$ on time scale $Nt$ and let $N\to\infty$. Note that, in this limit, the time scales 
for the jumps to different levels separate (recall \eqref{32b}), so that we can focus on each of the time 
scales separately.

If we randomly draw two lineages from the 1-block and ask whether they have a common ancestor 
some time back (so that they are of the same type), then we get exactly the event that generates 
the variance of the 1-block average (otherwise the lineages and their types would be independent 
and would have an asymptotically vanishing contribution to the variance). The fact that the lineages 
behave like a spatial coalescent follows from the duality introduced in Section~\ref{s.model}. The 
lineages have to meet in order to have a common ancestor, which takes them a time of order $Nt$. 
Note that triples of lineages have a negligible probability to meet at times of order $Nt$ in the limit of 
$N \to \infty$.

If the lineages meet, then they may coalesce. This happens at rate 
\begin{equation}\label{e1175}
\lambda^{(\eta,0)}(\omega)=\Lambda^{(\eta,0)}((0,1])(\omega)
\end{equation} 
when they both sit at $\eta\in\Omega$. However, they may also move before they coalesce, i.e., make 
a migration jump away, which happens with probability $2c_0 /(2c_0+\lambda^{(\eta,0)}(\omega))$. 
Hence the effective coalescence rate is $\lambda^{(\eta,0)}(\omega)[2c_0/(2c_0+\lambda^{(\eta,0)}
(\omega))]$. Since the vertex where the lineages meet is uniformly distributed over the 1-block, the 
average rate is given by 
\begin{equation}\label{e1184} 
\E\left[\frac{2c_0\lambda_0\rho(\omega)}{2c_0+\lambda_0\rho(\omega)}\right],
\end{equation} 
where we use that $\lambda^{(\eta,0)}(\omega)$ has the same distribution as $\lambda_0\rho(\omega)$ 
(recall \eqref{ap2}--\eqref{fdh:masschi}). If we would have a diffusive part as well, at constant rate 
$2d_0 $, then the lineages would coalesce at the same rate but with $\lambda(\omega)$ replaced 
by $2d_0+\lambda(\omega)$. Since the volatility turns out to be equal to this rate, we get the recursion 
formula 
\begin{equation}\label{e1192}
2d_1=\E\left[\frac{2c_0(2d_0+\lambda_0\varrho(\omega))}
{2c_0+(2d_0+\lambda_0\varrho(\omega))}\right].
\end{equation}
By the same reasoning for $k$-blocks on time scale $tN^k$, we get a heuristic explanation for the 
recursion formula in \eqref{diffusion-constants}.

\medskip
Our next theorem states that for each $k\in\N_0$ the $k$-block averages in the limit as $N\to\infty$ 
evolve according to the McKean-Vlasov process defined in Section~\ref{sss.mfl} with certain 
$k$-dependent parameters.  

\begin{theorem}[{\bf Hierarchical mean-field limit and renormalization}]
\label{mainth}
Suppose that for each $N$ the random field $X^{(\Omega_N)}(\omega;0)$ is the restriction to 
$\Omega_N$ of a random field $X(\omega)$ indexed by $\Omega_\infty = \bigoplus_\N  \N$ 
that is i.i.d.\ with single-component mean $\theta\in\CP(E)$. Then, for $\bP$-a.e.\ $\omega$ and 
every $k\in\N$ and $\eta\in\Omega_\infty$,
\be{macroscopic-behaviour}
\lim_{N\to\infty} \mathcal{L}
\left[\left(Y_{\eta,k}^{(\Omega_N)}(\omega;t N^k)\right)_{t\geq 0}\right]
= \mathcal{L}\left[\left(Z_\theta^{c_k,d_k,\Lambda^{\mathrm{MC}_k(\eta)}
(\omega)}(t)\right)_{t\geq 0}\right],
\ee
where 
\be{MCdef}
\mathrm{MC}_k(\eta) = \text{ unique site in $\Omega^\T_\infty$ at height $k$ above 
$\eta\in\Omega_\infty$}, 
\ee
i.e., the label of the block (= macro-colony) of radius $k$ in $\Omega_\infty$ around $\eta\in
\Omega_\infty$ (see \textup{Fig.}~{\rm \ref{fig-hierargr}}). The same is true for $k=0$ 
when the initial condition for the McKean-Vlasov process in the right-hand side of 
\eqref{macroscopic-behaviour} is $Z_\theta^{c_0,d_0,\Lambda^{\eta}(\omega)}(0) 
=X^{(\Omega_N)}(\omega;0)$ instead of $Z_\theta^{c_0,d_0,\Lambda^{\eta}(\omega)}(0)
=\theta$. \hfill $\square$
\end{theorem}

Note that among the parameters $c_k,d_k,\Lambda^{\mathrm{MC}_k(\eta)}(\omega)$ of the 
limiting McKean-Vlasov process, the volatility $d_k$ is the result of a \emph{self-averaging} 
with respect to the random environment up to and including level $k$, as exemplified by 
\eqref{diffusion-constants}. \emph{It is through this recursion relation that the renormalization 
manifests itself.} 
 
Our next theorem looks at successive block averages simultaneously.
 
\begin{theorem}[{\bf Multi-scale analysis and the interaction chain}]
\label{T.InteractionChain}
Let $(t_N)_{N\in\N}$ be such that 
\begin{equation}\label{e1223}
\begin{aligned}
\lim_{N\to\infty} t_N=\infty \text{ and } \lim_{N\to\infty}t_N/N=0.
\end{aligned}
\end{equation}
Then, for $\bP$-a.e.\ $\omega$, every $j\in\N$ and every $\eta\in\Omega_\infty$,
\be{ag:interaction-chain}
\lim_{N\to\infty} \mathcal{L}\left[ 
\left( Y^{(\Omega_N)}_{\eta, k}(\omega; t_N N^k) \right)_{k = -(j+1), -j, \ldots, 0}\right]
= \mathcal{L}\left[\left(M^{(j)}_{\eta, k}(\omega)\right)_{k = -(j+1), -j, \ldots, 0}\right],
\ee
where $M^{(j)}_\eta(\omega)=(M^{(j)}_{\eta, k}(\omega))_{k = -(j+1), -j, \ldots, 0}$ is the 
time-inhomogeneous Markov chain with initial state 
\be{ag:interaction-chain-initial}
M^{(j)}_{\eta, -(j+1)}(\omega) = \theta,
\ee
and transition kernel from time $-(k+1)$ to $-k$ given by
\be{ag:interaction-chain-transition-kernel}
K_{\eta, k}(\omega; \theta, \cdot) 
= \nu^{c_k, d_k, \Lambda^{\mathrm{MC}_k(\eta)}(\omega)}_\theta(\cdot).
\ee 
\hfill $\square$
\end{theorem} 

\noindent
The right-hand side of \eqref{e1223} describes the large space-time scaling behaviour of our 
hierarchical Cannings process. 

\begin{definition}[{\bf Interaction chain}]
\label{def:interaction-chain}
$M^{(j)}_{\eta}(\omega)$ is called the interaction chain at level $j$ at location $\eta \in 
\Omega_\infty$ given $\omega$. \hfill $\square$
\end{definition}

\begin{remark}
\label{r.1261}
{\rm Theorem \ref{T.InteractionChain} only specifies the limiting distribution of the one-dimensional 
spatial marginals, i.e., the single interaction chains. Similarly as in Dawson, Greven and 
Vaillancourt~{\rm \cite[Section~0e]{DGV95}}, it is possible to also specify the joint distribution 
of the interaction chains, which can be viewed as a \emph{field of Markov chains indexed by} 
$\Omega^{\mathbb{T}}_\infty$.} \hfill $\square$
\end{remark}

An important characteristic of $ M^{(j)}_\eta$ is the variance of $ M^{(j)}_{\eta,0}$, calculated 
as 
\begin{equation}\label{e1268}
\var \langle M_{\eta,0}^{(j)}, f \rangle 
= \prod^{j}_{k=0} \frac{2c_k}{2c_k + \lambda_k\rho_k + 2d_k} \var_\theta(f).
\end{equation}
This shows that a key ingredient for $M_\eta^{(j)}$ is the sequence of volatilities $\ud=(d_k)_{k 
\in \N_0}$ and the way this sequence grows or decays. How is this affected by the randomness of 
the environment?

Our next theorem shows that the volatility $d_k$ in the random environment can be sandwiched 
between the volatility $d^0_k$ in the zero environment ($\CL_\rho=\delta_0$, i.e., the system 
without resampling) and the volatility $d^1_k$ in the average environment ($\CL_\rho=\delta_1$, 
i.e., the system with average resampling).

\begin{theorem}[{\bf Randomness lowers volatility}]
\label{T.order}
If $d^0_0=d_0 = d^1_0$, then $d^0_k<d_k < d^1_k$ for all $k \in \N$.  \hfill $\square$
\end{theorem}

\subsubsection{Dichotomy for the interaction chain}
\label{sss.coex} 

How are the qualitative properties of the Cannings process for large $N$ reflected in the interaction 
chain? What about the dichotomy clustering versus coexistence? Before answering these questions 
we need to first establish the existence of the entrance law of the interaction chain from level $\infty$, 
which we will obtain from the level $j$ interaction chain as limit $j\to\infty$. With this object, we can 
address the question of coexistence versus clustering.

\begin{proposition}[{\bf Entrance law of interaction chain exists}]\label{p.1267}\mbox{}\\
The limit as $j\to\infty$ of $M^{(j)}_\eta$ exists. \hfill $\square$
\end{proposition}

The object corresponding to the equilibrium of the stochastic system for finite $N$ in the 
hierarchical mean-field limit $N \to \infty$ is the {\em field of entrance laws} of the interaction 
chain from level $\infty$ (recall Remark~\ref{r.1261}), in particular, its marginal law $\Pi_\eta
\nu_\theta(\omega)$ at level $0$ in $\eta$, which is element of  $\CP(\CP(E)$. 

\begin{definition}[{\bf Entrance law of interaction chain}]
\label{p:interaction-chain-equilibrium}
For $\bP$-a.e.\ $\omega$ and all $\eta \in \Omega_\infty$,
\be{hierarch-mean-field-equil}
\lim_{j\to\infty} \mathcal{L}\left[M^{(j)}_{\eta, 0}(\omega)\right] 
= \Pi_\eta \nu_\theta(\omega),
\ee
where $\nu_\theta (\omega) \in \mathcal{P}(\mathcal{P}(E))^{\Omega_\infty}$ is the entrance law from 
level $\infty$ of the (tree-indexed) interaction chain at level $0$, and $\Pi_\eta\nu_\theta(\omega)$ 
denotes the projection of $\nu_\theta(\omega)$ on $\eta$. 
\hfill $\square$
\end{definition}

Our next theorem is indeed the analogue of Theorem~\ref{T.coexcritNfin} for $N \to \infty$. In this 
limit, coexistence and clustering in $\omega$ are defined for $(M^{(\infty)}_{\eta,0})_{\eta\in\Omega_N}$ 
in the same way as in \eqref{ag4.8}--\eqref{ag4.6}. 

\begin{theorem}[{\bf Dichotomy for $N=\infty$}]
\label{T.dicho}
$\mbox{}$
\begin{itemize}
\item[\textup{(a)}]
Let $\CC=\{\omega\colon\, \text{in $\omega$ coexistence occurs}\}$. 
Then $\bP(\CC) \in \{0,1\}$.
\item[\textup{(b)}] 
$\bP(\CC)=1$ if and only if
\be{inftydich}
\sum_{k\in\N_0} \frac{1}{c_k} \sum_{l=0}^k \lambda_l<\infty. 
\ee
\end{itemize} 
\hfill $\square$
\end{theorem}

\noindent
Note that condition \eqref{inftydich} is the limit of condition \eqref{Ndich} as $N\to\infty$. 
In fact, the two conditions are equivalent when the following \emph{weak regularity condition} 
holds:
\be{dichreg}
\text{ either } \quad \limsup_{k\to\infty} \frac{\lambda_{k+1}}{c_k} < \infty \quad \text{ or } \quad 
\liminf_{k\to\infty} \left(\frac{\lambda_{k+1}}{c_k} \wedge \frac{\lambda_k}{\lambda_{k+1}}\right)>0.
\ee

An important question is whether the equilibrium measure $\nu_\theta (\omega)$ is the limit 
as $N\to\infty$ of the equilibrium measure  $\nu_\theta^N (\omega)$ (recall \eqref{ag4.9}). The 
answer is yes. We only prove the following.

\begin{corollary}[{\bf Hierarchical mean field limit of equilibrium}]
\label{c.1282}
$\mbox{}$\\
For $ \bP$-a.s.\ all $\omega$ and all $\eta \in \Omega_\infty$,
\begin{equation}\label{e1309}
\lim_{N\to\infty} \Pi_\eta \nu^N_\theta = \Pi_\eta \nu_\theta.
\end{equation} 
\hfill $\square$
\end{corollary}

\subsubsection{Scaling of the volatility}
\label{sss.vola}

We are interested in the behaviour of the variance of the interaction chain $M^{(j)}$ as $j\to\infty$, 
since this allows us to identify {\em universality classes} for the scaling behaviour of our stochastic 
system. From the variance formula, we see that $(d_k)_{k \in \N_0}$ is the key input, and so we 
study this sequence first. Note that the variance formula really depends on the ratios $d_k/c_k, 
\mu_k/d_k$, which we encounter below.

Our next two theorems identify the scaling behaviour of $d_k$ as $k\to\infty$ \emph{in the regime 
of clustering}. The first theorem considers the case of \emph{polynomial coefficients}, i.e.,
\be{fdh:regcond}
c_k \sim L_c(k)\,k^a, \qquad \mu_k \sim L_\mu(k)\,k^b, \qquad k \to \infty, 
\ee
with $a,b \in \R$ and $L_c,L_\mu$ slowly varying at infinity. 
In what follows, we assume that
\be{Kdef}
K = \lim_{k\to\infty} \frac{\mu_k}{c_k} \in [0,\infty], \qquad 
L = \lim_{k\to\infty} \frac{k^2\mu_k}{c_k} \in [0,\infty], 
\ee 
exist and write $ K_k $ and $ L_k $ for the respective sequences. There are \emph{five cases} according 
to the values of $K$ and $L$. Four of these, labelled (a)--(d), we can analyze in detail. For the remaining 
case, labelled (e), see Remark~\eqref{r:notexhaus}. For cases (c)--(d), we need extra regularity conditions 
on $L_c,L_\mu$ in \eqref{fdh:regcond}, for which we refer the reader to \cite[Eqs.~(1.79)--(1.81)]{GHKK14}.

\begin{theorem}[{\bf Scaling of the Fleming-Viot volatility: polynomial coefficients}]
\label{T.scaleFVpol}
$\mbox{}$\\
Under the polynomial scaling assumptions \eqref{fdh:regcond}--\eqref{Kdef}, the following cases apply:
\begin{itemize}
\item[\textup{(a)}]
If $K = \infty$, then $\lim_{k\to\infty} d_k/c_k = 1$.
\item[\textup{(b)}] 
If $K \in (0,\infty)$, then $\lim_{k\to\infty} d_k/c_k = M$ with $M \in (0,1)$ the unique solution of the 
equation
\be{MK}
M = \bE_{\CL_\rho}\left[\frac{(K\rho+M)}{1+(K\rho+M)}\right].
\ee
\item[\textup{(c)}]
If $K = 0$ and $L = \infty$, then $\lim_{k\to\infty} d_k/\sqrt{c_k\mu_k} = 1$.
\item[\textup{(d)}] 
If $K=0$, $L \in [0,\infty)$ and $a \in (-\infty,1)$, then $\lim_{k\to\infty} \sigma_kd_k = M$ with 
$\sigma_k = \sum_{l=0}^{k-1} (1/c_l)$ and $M \in [1,\infty)$ given by
\be{M*K}
M = \tfrac12\left[1+\sqrt{1+4L/(1-a)^2}\right].
\ee 
\end{itemize}
\hfill $\square$
\end{theorem}

\begin{remark}\label{r:notexhaus}
{\rm It is straightforward to check with the help of \eqref{fdh:regcond}--\eqref{Kdef} that all four 
cases {\rm (a)-(d)} correspond to choices of $\uc$ and $\ul$ for which clustering holds, i.e., the 
sum in \eqref{inftydich} diverges (note that $\lim_{k\to\infty} \sigma_k=\infty$ in case {\rm (d)}).
However, they are not exhaustive: there is a  fifth case (e), corresponding to $K=0$, $L \in [0,\infty)$, 
$a=1$ and $\lim_{k\to\infty} \sigma_k=\infty$, for which we have no scaling result. This case lies 
at the border of the clustering regime. An example is $c_k \sim k (\log k)^\gamma$, $\gamma 
\in (-\infty,1]$, and $\mu_k = k^{-2} c_k$, which we were able to handle in the deterministic model 
in \cite{GHKK14}, but cannot handle in the random model treated here.} \hfill $\square$
\end{remark}

The second theorem considers the case of \emph{exponential coefficients}, i.e., 
\be{expregvar}
c_k = c^k\bar{c}_k, \qquad \mu_k= \mu^k\bar{\mu}_k
\ee
with $c,\mu \in (0,\infty)$ and $\bar{c}_k,\bar{\mu}_k$ satisfying \eqref{fdh:regcond} with 
exponents $a,b$. We further assume that
\be{barK}
\bar{K} = \lim_{k \to \infty} \frac{\bar{\mu}_k}{\bar{c}_k} \in [0, \infty], \mbox{  we write  } 
\bar K_k = \frac{\bar \mu_k}{\bar c_k}.
\ee 
exists.

\begin{theorem}[{\bf Scaling of the Fleming-Viot volatility: exponential coefficients}] 
\label{T.scaleFVexp}
$\mbox{}$\\
Under the exponential scaling assumptions in \eqref{expregvar}--\eqref{barK}, the following cases 
apply (cf.\ \textup{Theorem~\ref{T.scaleFVpol}}):
\begin{itemize}
\item[\textup{(A)}]
\textup{[Like Case~(a)]}
$c<\mu$, or $c=\mu$ and $\bar{K}=\infty$: $\lim_{k\to\infty} d_k/c_k=1/c$.
\item[\textup{(B)}]
\textup{[Like Case~(b)]}
$c=\mu$ and $\bar{K} \in (0,\infty)$: $\lim_{k\to\infty} d_k/c_k = \bar{M}/c$ with $\bar{M} \in (0,1)$ 
the unique solution of the equation
\be{Mexp}
\bar{M} = \bE_{\CL_\rho}\left[\frac{(cK\rho+\bar{M})}{c+(cK\rho+\bar{M})}\right].
\ee
\item[\textup{(C)}]
The case $\bar{K}=0$, with $c=\mu$ or $c>\mu$, splits into three sub-cases:
\begin{itemize}
\item[\textup{(C1)}]
\textup{[Like Case~(b)]}
$c=\mu<1$, $\bar{K}=0$: $\lim_{k\to\infty} d_k/c_k=(1-c)/c$.
\item[\textup{(C2)}]
\textup{[Like Case \textup{(c)}]}
$c=\mu>1$, $\bar{K}=0$, $\sum_{k\in\N_0} \bar{K}_k = \infty$:
\footnote{In \cite{GHKK14} the condition $\sum_{k\in\N_0} \bar{K}_k = \infty$ was mistakenly omitted.}
$\lim_{k\to\infty} d_k/\mu_k 
= 1/(\mu-1)$. 
\item[\textup{(C3)}]
\textup{[Like Case~(d)]}
$1>c>\mu$ and $\bar{K}=0$, or $1=c>\mu$, $\bar{K}=0$ and $a \in (-\infty,1)$: 
$\lim_{k\to\infty} \sigma_k d_k = 1$. 
\end{itemize}
\end{itemize} 
\hfill $\square$
\end{theorem}

\noindent
The same observation as in Remark~\ref{r:notexhaus} applies. Again, the critical case $a=1$ is 
missing in (C3).

\subsubsection{Cluster formation}
\label{sss.clust}

Within the clustering regime it is of interest to study the size of the mono-type regions as a function of 
time, i.e., how fast the clusters where one type prevails grow.

This question has been addressed for other population models. For the voter model on $\Z^2$, Cox 
and Griffith~\cite{CG86} showed that the radii of the clusters with opinion ``all 1'' or ``all 0'' scale as 
$t^{\alpha/2}$ with $\alpha \in [0,1)$, i.e., clusters occur on all scales $\alpha \in [0,1)$. For the model 
of hierarchically interacting Fleming-Viot diffusions with $c_k \equiv 1$ (= critically recurrent migration), 
Fleischmann and Greven~\cite{FG94} showed that, for all $N \in \N \setminus \{1\}$ and all $\eta \in 
\Omega_N$, 
\be{allim}
\lim_{t\to\infty}
\CL\left[\left(Y^{(\Omega_N)}_{\eta,\lfloor(1-\alpha)t\rfloor}(N^t)\right)_{\alpha \in [0,1)}\right] 
= \CL\left[\left(Y\left(\log\left(\frac{1}{1-\alpha}\right)\right)\right)_{\alpha \in [0,1)}
\right]
\ee 
in the sense of finite-dimensional distributions, where $(Y(t))_{t \in [0,\infty)}$ is the standard 
Fleming-Viot diffusion on $\CP(E)$. A similar behaviour occurs for other models, e.g.\ branching 
models as shown in Dawson and Greven~\cite{DG96}. 

The advantage of the hierarchical group is that we can analyze the cluster formation as a function 
of $N$ and let $N \to \infty$ to approach the critically recurrent case (recall Remark~\ref{degreeretr}). 
We can do this by using the {\em interaction chain}. In \cite{GHKK14} we analysed the Cannings 
model in the limit as $N \to \infty$, namely, we proved that for some \textit{level scaling function} 
$k\colon\,\N_0\to\N_0$, satisfying $0 \leq k(j) \leq j+1$ and $\lim_{j\to\infty} k(j)=\infty$, we obtained 
a \textit{non-trivial clustering limiting law} (henceforth we pick $\eta=0$ and drop it from the notation)
\be{ak:interaction-chain-scaling-limit-homogeneous}
\lim_{j \to \infty} \mathcal{L} \left[ M^{(j)}_{-k(j)} \right]
= \mathcal{L} \big[M^\infty\big]
\ee
for some $M^\infty\in\mathcal{P}(E)$ satisfying $E[M^\infty] = E[X_0^{(\Omega_N)}(0)] = \theta 
\in \CP(E)$ that is not of the form $M^\infty=\delta_{U}$ for some possibly random $U \in E$. We 
will do the same in the random environment $\omega$, namely, our aim is to show that for 
$\mathbb{P}$-a.e.\ $\omega$
\be{ak:interaction-chain-scaling-limit}
\lim_{j \to \infty} \mathcal{L} \left[ M^{(j)}_{-k(j)}(\omega) \right]
= \mathcal{L} \big[M^\infty(\omega) \big]
\ee
for some $M^\infty(\omega) \in \mathcal{P}(E)$ satisfying $\E[M^\infty(\omega)] = \theta$ that is not of the 
form $M^\infty(\omega) = \delta_{U(\omega)}$ for some possibly random $U(\omega) \in E$. 

As in Dawson and Greven~\cite{DG93b,DGV95,DG96}, and similarly as in \eqref{allim}, in order to 
obtain the \emph{profile of cluster formation} it is necessary to consider a whole \emph{family of scalings} 
$k_\alpha\colon\,\N_0 \to \N_0$, $\alpha \in I$, with $I=\N_0$, $I=[0,\infty)$ or $I=[0,1)$, and with
$j \mapsto k_\alpha(j)$ non-decreasing, $0 \leq k_\alpha(j) \leq j+1$ and $\lim_{j\to\infty} k_\alpha(j) 
= \infty$, such that 
\be{eq:ak:scaling-interaction-chain}
\lim_{j\to\infty} \mathcal{L} \left[M^j_{-k_\alpha(j)}(\omega)\right]
= \mathcal{L} \left[M^*_{\alpha}(\omega)\right]
\quad \text{ for } \mathbb{P} \text{-a.e. } \omega \text{ and all } \alpha \in I,
\ee
for some non-constant Markov process $M^*=(M^*_\alpha(\omega))_{\alpha \in I}$ on $\mathcal{P}(E)$ 
that preserves the mean $\theta$. The convergence in \eqref{eq:ak:scaling-interaction-chain} is in the weak 
topology on the product space of $\mathcal{P}(E)$ and the space of the environment. 

There are five {\em universality classes} of clustering behaviour (see \cite{DG96}): 
\begin{definition}[{\bf Clustering classes}]
\label{def:clusuniv} 
$\mbox{}$
\begin{itemize}
\item[\textup{(I)}] 
{\bf Concentrated clustering} ($M^\ast$ is a Markov chain):
\begin{itemize}
\item[\textup{(I1)}]  
$k_\alpha(j) = 0 \vee (j+1 - \alpha)$, $\alpha \in \N_0$, $M^*$ is trapped after one step.
\item[\textup{(I2)}]
$k_\alpha(j) = 0 \vee (j+1 - \alpha)$, $\alpha \in \N_0$, $M^*$ is not trapped.
\end{itemize}
\item[\textup{(II)}] 
{\bf Diffusive clustering} ($M^\ast$ is a diffusion process):
\begin{itemize}
\item[\textup{(II1)}]  
\emph{Fast clustering}: 
$k_\alpha(j) = 0 \vee \lfloor j+1 - \alpha h(j) \rfloor $, $\alpha \in [0,\infty)$, where $h\colon\,\N_0
 \to [0,\infty)$ is such that $\lim_{j\to\infty} h(j) = \infty$ and $\lim_{j\to\infty} h(j)/j = 0$.
\item[\textup{(II2)}] 
\emph{Moderate clustering}: 
$k_\alpha(j) = \lfloor (1-\alpha)(j+1)\rfloor$, $\alpha \in [0,1)$. 
\item[\textup{(II3)}] 
\emph{Slow clustering}: 
$\lim_{j\to\infty} k_\alpha(j)/j=0$, $\alpha \in [0,1)$.
\end{itemize}
\end{itemize}
\end{definition}

\noindent
(The terminology is slightly different from \cite{GHKK14}.) The volume of the clusters at time $t$ in 
these five universality classes (arranged in decreasing order of magnitude) equals, respectively, 
$N^t$, $ZN^t$, $N^{t-o(t)}$, $N^{Zt}$, $N^{o(t)}$, with $Z \in (0,1)$ some random variable. Note 
that slow clustering borders with the regime of coexistence (= no clustering).

Recall (a)-(d) in Theorem~\ref{T.scaleFVpol} and (A)-(C) in Theorem~\ref{T.scaleFVexp}. Recall 
that, under the law $\bP$, the law of the initial state $X^{(\Omega_N)} (\omega;0)$ is stationary 
and ergodic under translations in $\Omega_N^{\mathbb{T}}$, with mean single-coordinate measure 
\begin{equation}\label{e1548}
\theta = \bE[X_0^{(\Omega_N)}(\omega;0)] \in \CP(E).
\end{equation} 
The interaction chain on level $j$, arising in the scaling limit $N\to\infty$, starts in $\theta$. 
Below this is also assumed for the scaling limit $j\to\infty$.

\begin{theorem}[{\bf Cluster formation}]
\label{T.cluform}
Fix $N \in \Ntwo$. The five universality classes in the regime of clustering, linked to the different 
cases of scaling behaviour of $ (d_k)_{k \in \N}$, are as follows:
\begin{itemize}
\item[$\bullet$]
{\rm \textbf{(a), (A):}}
The scaling in regime \textup{(I1)} yields \eqref{eq:ak:scaling-interaction-chain} with $I=\N_0$. 
The scaling limit $M^*$ is the time-homogeneous Markov chain on $\mathcal{P}(E)$ starting in 
$\theta$ with transition kernel $K(\theta, \cdot)$ given by 
\be{regi1}
K(\theta, \cdot) = \int_E \theta(\dd u) \delta_{\delta_u}(\cdot),
\ee
which satisfies $K_\alpha=K$, for all $\alpha \in \N$.
\item[$\bullet$]
{\rm \textbf{(b), (B), (C1), (C3)[first subcase]:}}
The scaling in regime \textup{(I2)} yields \eqref{eq:ak:scaling-interaction-chain} with $I=\N_0$. 
The scaling limit $M^*$ is the time-inhomogeneous Markov chain on $\mathcal{P}(E)$ in random 
environment $(\chi_\alpha)_{\alpha \in \N_0}$ starting in $\theta$ with transition kernels 
$\{K_\alpha(\theta,\cdot)(\omega)\}_{\alpha \in \N_0}$ given by (recall \eqref{eq:equi-ident-c}) 
\be{regi2}
K_\alpha(\theta, \cdot)(\omega) 
= \nu^{1, \widetilde{M}, 2 \widetilde{K} \chi_\alpha(\omega)}_\theta (\cdot)
\ee
with
\begin{equation}
\label{wtMwtK}
(\widetilde{M}, \widetilde{K}) =
\begin{cases}
(M, K), & \textup{(b)},\\
(\bar{M}/c, \bar{K}), & \textup{(B)},\\
((1-c)/c, 0), & \textup{(C1), (C3)[first subcase]}.
\end{cases}
\end{equation}
In the last two cases, the random environment does not affect the scaling limit, and the scaling is 
the same as for the homogeneous environment with the same mean. In the first two cases the 
measure-valued process $(\chi_\alpha(\omega))_{\alpha\in\N_0}$ in \eqref{regi2} is constructed 
by extending the one-sided stationary random environment $(\chi^{(\eta,k)} (\omega))_{\eta\in\Omega_N,
\,k\in\N_0}$ introduced in \eqref{chiprop} to a two-sided stationary random environment $(\chi^{(\eta,k)}
(\omega))_{\eta\in\Omega_N,\,k\in\Z}$, and defining
\be{regi3}
\CL[(\chi_\alpha(\omega))_{\alpha \in \N_0}] 
= \CL[(\chi^{(\eta,-\alpha)}(\omega))_{\alpha \in \N_0}],
\quad \eta \in \Omega_N,
\ee
which by stationarity does not depend on $\eta$. Furthermore, $\chi_\alpha(\omega)$ is an 
$\mathcal{M}_f([0,1])$-valued resampling measure with $\bE[\chi_\alpha(\omega)] = \bar \chi_\alpha$ 
satisfying $\bar \chi_\alpha((0,1])=1$.
\item[$\bullet$]
{\rm \textbf{(c), (C2)[subcase $\lim_{k\to\infty} k\bar{K}_k=\infty$]:}} 
The scaling in regime \textup{(II1)} yields \eqref{eq:ak:scaling-interaction-chain} with $I=[0,\infty)$. 
The scaling limit $M^*$ is the time-changed standard Fleming-Viot process
\be{clusteringC}
M^*_{\alpha}  = Z^{0,1,0}_{\theta} \left( \ell(\alpha) \right), \quad \alpha \in [0,\infty)
\ee
with
\begin{itemize}
\item[--]
{\rm \textbf{(c):}} 
$\ell(\alpha) = \alpha$, $h(j) = 1/\sqrt{K_j}$.
\item[--]
{\rm \textbf{(C2)[$\lim_{k\to\infty} k\bar{K}_k = \infty$]:}}
$\ell(\alpha) = \frac{\mu}{\mu-1}\alpha$, $h(j) = 1/K_j$. 
\end{itemize}
\item[$\bullet$]
{\rm \textbf{(d), (C2)[subcase $\lim_{k\to\infty} k\bar{K}_k=\bar{N}$], (C3)[second subcase]:}}
The scaling in regime \textup{(II2)} yields \eqref{eq:ak:scaling-interaction-chain} with $I = [0,1)$.
The scaling limit $M^*$ is the time-changed standard Fleming-Viot process
\footnote{\cite{GHKK14} contains a typo: there the time scaling $1/(1-\alpha)^R$ was wrongly 
written as $1/(1-\alpha^R)$.}
\be{clusteringD}
M^\ast_{\alpha}  = Z^{0,1,0}_{\theta} \left(\log\left(\frac{1}{(1-\alpha)^R}\right)\right),
\quad \alpha \in [0,1)
\ee
with
\begin{itemize}
\item[--]
{\rm \textbf{(d):}} 
$R=M(1-a)$.
\item[--]
{\rm \textbf{(C2)[subcase $\lim_{k\to\infty} k\bar{K}_k = \bar{N}$]:}}
$R=\bar{N}\frac{\mu}{\mu-1}$.
\item[--]
{\rm \textbf{(C3)[second subcase]:}}
$R=1-a$.
\end{itemize}
\end{itemize}
For reasons explained in Remark~{\rm \ref{rem:homclus}}, in cases {\rm \textbf{(c), (d), (C2), 
(C3)[second subcase]}} only convergence in $\mathbb{P}$-probability and not 
$\mathbb{P}$-a.s.\ is obtained.
\hfill $\square$
\end{theorem}

\begin{remark}\label{r.1602}
{\rm We expect that also {\em slow} clustering occurs, namely, in the case (e) that was not treated 
in Theorems~\ref{T.scaleFVpol} and \ref{T.scaleFVexp} (recall Remark~\ref{r:notexhaus}).} 
\hfill $\square$
\end{remark}

\subsection{Summary of the effects of the random environment}
\label{ss.effectre} 

\paragraph{1.}
Theorem~\ref{T.wpbasic} says that the hierarchical Cannings process in random environment is 
well-defined for $\bP$-a.e.\ $\omega$, while Theorem~\ref{T.ltbeq} shows that it converges to 
an $\omega$-dependent equilibrium that preserves the single-component mean. 

\paragraph{2.}
Theorem~\ref{mainth} (mono-scale) and Theorem~\ref{T.InteractionChain} (multi-scale) identify 
the behaviour of the $k$-block averages in the limit as $N\to\infty$ in terms of the McKean-Vlasov 
process with parameters that depend on $\omega$ and $k$. The volatility $d_k$ depends on the 
parameters $c_l,\lambda_l$, $0 \leq l < k$, and on the law $\CL_\rho$ via the recursion relation 
in \eqref{diffusion-constants}, which is a \emph{randomized} version of the recursion relation in 
\cite{GHKK14}.  

\paragraph{3.}
Theorems~\ref{T.coexcritNfin} and \ref{T.dicho} show that the dichotomy ``coexistence versus 
clustering'' is \emph{not} affected by the random environment: the \emph{same} conditions apply to 
the homogeneous hierarchical Cannings process studied in \cite{GHKK14}. Apparently, for the nature 
of the equilibrium only the large-scale properties of the random environment matter. Since the 
resampling measures are stationary under translations with total masses whose sigma-algebra at 
infinity is trivial, only the average medium behaviour is relevant. The proof of the dichotomy in 
Theorem~\ref{T.coexcritNfin} requires assumption \eqref{Apropalt} rather than assumption 
\eqref{Aprop}. We believe this strengthening to be redundant, but a proof would require considerable 
extra work. 

\paragraph{4.}
Theorem~\ref{T.order} shows that \emph{the effect of the random environment is to lower the 
volatility parameter $ d_k $ on every hierarchical scale $k$} compared to the average environment. The intuition behind 
this is that the random environment causes fluctuations in the resampling, which in turn reduce the 
clustering. The sandwich between the volatilities for the zero environment and the average environment 
is useful to control the scaling. 

\paragraph{5.}
Theorem~\ref{T.scaleFVpol} (polynomial coefficients) and Theorem~\ref{T.scaleFVexp} 
(exponential coefficients) show that for Cases (b) and (B), where \emph{migration and resampling 
occur at comparable rates}, the phenomenon of lower volatility $ d_k $ in random environment persists in 
the limit as $k\to\infty$: even though the scaling of $d_k$ as $k\to\infty$ is the same as for the 
average environment, it has a \emph{different prefactor} (e.g.\ $M$ solving \eqref{MK} is strictly 
smaller than $M^*$ solving \eqref{MK} with $\CL_\rho$ replaced by $\delta_1$, as is easily shown 
by applying Jensen's inequality). For all other cases both the scaling and the prefactor are the same 
as for the average environment.

\paragraph{6.} 
Theorem~\ref{T.cluform} shows that for Cases (b), (B), (C1), (C3)[first subcase] the scaling of the 
clusters in the random environment changes compared to that in the average environment: the 
random environment is visible even in the scaling limit. \emph{The effect of the random environment 
is to slow down the growth of the clusters}, i.e., to enhance the diversity of types. For all other cases 
the scaling of the clusters is the same as for the average environment.

\section{Existence, uniqueness, duality and equilibrium}
\label{s.dual}

In this section, we prove Theorems~\ref{T.wpbasic}--\ref{T.ltbeq}. In Section~\ref{ss.spatialrand} 
we construct the dual process with the help of a graphical representation based on Poisson random 
measures. In Section~\ref{ss.dualities} we exhibit the duality. In Section~\ref{ss.dualproofs} we 
establish the existence and uniqueness of the dual process and show the existence of its equilibrium. 
In Section~\ref{ss.dualproofsthms} we use these results to prove Theorems~\ref{T.wpbasic}--\ref{T.ltbeq}. 
Theorems~\ref{T.existence-coalescent}--\ref{T.duallongrun} below do not need a separate proof: this 
is verbatim the same as the proof for the homogeneous environment given in \cite{GHKK14}.

\subsection{The spatial coalescent in random environment}
\label{ss.spatialrand}

In this section we introduce a hierarchical coalescent process in random environment that will serve 
as a dual to the hierarchical Cannings process in random environment. 

The coalescent is a Markov process taking values in the set of partitions of  $\N$ labelled by the 
points of a geographical space. We recall the basic objects and notations, and refer to 
\cite[Section~2]{GHKK14} for details. 

Let $G$ be a discrete geographical space. Our target geographical space is $G = \Omega_N$. 
This will be approximated by a sequence of geographical spaces 
\be{Gchoices}
G_{N,K}=\{0,\ldots,N-1\}^K, \quad K\in\N,
\ee
which are to be thought of as a sequence of blocks filling $\Omega_N$. 
We will also need to consider the mean-field geographical space
\be{Gchoices*}
G = \{0,*\},
\ee
where $\{\ast\}$ is a cemetery location. 
The state space of the spatial coalescent is the set of \emph{$G$-labelled partitions} defined as 
\be{labelled-partitions} 
\Pi_{G,n} = \Big\{ \pi_{G} = \{ (\pi_1, g_1), (\pi_2, g_2), \ldots, (\pi_b, g_b) \}
\colon\,\{ \pi_1, \ldots, \pi_b \} \in \Pi_n,\, g_1, \ldots, g_b \in G, b \in [n] \Big\}, 
\ee 
where $n\in\N$ and
\be{ag19} 
\Pi_n = \text{ set of partitions }  \pi = \{ \pi_i \subset [n] \}_{i=1}^b
\text{ of $[n]$ into disjoint families $\pi_i$, $i\in[b]$.}
 \ee

We equip the set $\Pi_{G,n}$ with the {\em discrete topology}. Let $a$ be a random walk transition 
kernel on $G$. When $G = \Omega_N$ we use the hierarchical random walk kernel $a = a^{(N)}$ 
in \eqref{32b}, when $G = G_{N,K}$ we use the same hierarchical random walk kernel but with 
$c_k = 0$ for $k > K$, and when $G = \{ 0, \ast \} $ we use the random walk kernel with 
$a(0,\ast) = c$, $a(\ast,0) = 0$.

Given the random environment $\omega$ (recall Section~\ref{sss.RE}), the spatial coalescent 
in random environment is the Markov process on state space $\Pi_{G,n}$ with the following dynamics:
\begin{itemize}
\item 
\textbf{[Migration]} Each partition block performs an independent random walk on $G$ with random 
walk kernel $a^*$, where $a^*(g_1,g_2) = a(g_2,g_1)$, $g_1,g_2 \in G$, is the conjugate random 
walk kernel. 
\item
\textbf{[Local coalescence]} Independently at each location $g \in G$, the $l$-tuples of the partition 
elements at $g$ coalesce into a single partition element at $g$ at rate
\be{eq:lambda-coalescence-rates}
\lambda_{b,l}(\omega) = \int_{(0,1]} r^l (1-r)^{b-l} 
\frac{\Lambda^{[g]}(\omega)(\dd r)}{r^2},
\ee
where $b$ is the current total number of partition elements and $\Lambda^{[g]}(\omega)$ is the 
resampling measure at $g$ in environment $\omega$.
\item 
\textbf{[Non-local coalescence with reshuffling]} In the case $G = \Omega_N$, independently at each 
location $g \in B_{|\xi|}(\xi)$, $\xi \in \Omega^{\mathbb{T}}_N$, the $l$-tuples of the partition elements 
in $B_{|\xi|}(\xi)$ coalesce into a single partition element at $g$ at rate
\be{eq:lambda-coalescence-rates-l}
N^{-2k} \lambda^{(\xi)}_{b,l}(\omega),
\qquad 
\lambda^{(\xi)}_{b,l}(\omega) = \int_{(0,1]} r^l (1-r)^{b-l} 
\frac{\Lambda^\xi(\omega)(\dd r)}{r^2}.
\ee
Subsequently, all the partition elements located in $B_{|\xi|}(\xi)$ are uniformly \emph{reshuffled}, i.e., 
all the partition elements in $B_{|\xi|}(\xi)$ get a new location that is drawn uniformly from $B_{|\xi|}(\xi)$.
\end{itemize}

Note that in the case $G=\Omega_N$ the partition elements of the coalescent perform a hierarchical 
random walk on $\Omega_N$ in the environment $\omega$ with migration coefficients given by 
(recall \eqref{32b})
\be{cNomega}
c_k(\omega)(N,\eta) = c_k + N^{-1}\,\lambda^{\mathrm{MC}_{k+1}(\eta)}(\omega),
\qquad \eta \in \Omega_N, k \in \N_0,
\ee
where $\mathrm{MC}_k(\eta)$ is the unique site at height $k$ above $\eta\in\Omega_N$ and 
$\lambda^{\mathrm{MC}_{k+1}(\eta)}(\omega) = \Lambda^{\mathrm{MC}_{k+1}(\eta)} (\omega)$ 
$((0,1])$ (recall the notation introduced in Section~\ref{ss.random} and see Fig.~\ref{fig-hierartree}).
The extra term in the right-hand side of \eqref{cNomega} comes from the reshuffling that takes 
place prior to the resampling. 

The {\em coalescence} rate of two partition elements in $B_{|\xi|}(\xi)$ equals (recall \eqref{fdh:masschi})
\be{lambdaxi}
N^{-2|\xi|} \lambda^\xi(\omega), \qquad  
\lambda^\xi(\omega)= \lambda_{|\xi|} \rho^\xi(\omega), 
\qquad \xi\in\Omega_N^{\mathbb{T}}.
\ee 

We specify the spatial coalescent as a Markov process on $\Pi_G = \cup_{n\in\N} \Pi_{G,n}$ by 
providing its generator. To that end, we need a space of test functions on $\Pi_G$. Namely, let 
$\CC_G$ be the algebra of bounded continuous functions $F\colon\,\Pi_G \to \R$ such that for all 
$F \in \CC_G$ there exists an $n=n(F) \in \N$ and a bounded function
\be{ghkk1}
F_n\colon\, \Pi_{G,n} \to \R
\ee 
with the property that $F(\cdot)=F_n(\cdot\vert_n)$.
Consider the linear operator $L^{(G)*}(\omega) \colon \mathcal{C}_G \to \mathcal{C}_G$ 
defined as
\be{spatial-coalescent-generator}
L^{(G)*}(\omega) = L^{(G)*}_{\mathrm{mig}} + L^{(G)*}_{\mathrm{coal}}(\omega),
\ee
where the operators $L^{(G)*}_{\mathrm{mig}},L^{(G)*}_{\mathrm{coal}}(\omega)\colon\, 
\mathcal{C}_G \to \mathcal{C}_G$ are defined for $\pi_G \in \Pi_G$ and $F \in \mathcal{C}_G$ as
\be{no3}
(L^{(G)*}_{\mathrm{mig}} F) (\pi_G) 
= \sum_{i=1}^{b(\pi_{G}|_n)} \sum_{g,f \in G} a^*(g,f)
\big[F_n\big(\textrm{mig}_{g \to f,i}(\pi_{G,n})\big)-F(\pi_G)\big]
\ee
and
\be{add2_a}
\begin{aligned}
&(L^{(G)*}_{\mathrm{coal}}(\omega) F) (\pi_G) 
= \sum_{k\in\N_0} \, \sum_{\xi \in \Omega^{(k)}_N} \,\,
\sum_{\substack{J\subset\{i \in [n]\colon\,g_i = g\},\\|J| \geq 2}}\\
&\qquad \left(N^{-2k} \lambda^{(\xi)}_{b(\pi_{G,n},g),|J|}(\omega)
\left[F_n\left(\textrm{resh}_{B_{|\xi|}(\xi),U_{B_{|\xi|}(\xi)}} 
\circ \textrm{coal}_{J,g}(\pi_{G,n})\right)-F(\pi_G)\right]\right).
\end{aligned}
\ee
Here, the \textit{migration map} $\textrm{mig}_{g \to f,i}(\pi_{G}\vert_n)$ changes the spatial coordinate 
of the $i$-th partition block from $g$ to $f$ (if such a partition element exists), the \textit{coalescence 
map} $\textrm{coal}_{J,g}(\pi_{G,n})$ coalesces the partition blocks with indices in $J$ and location 
$g$ (if any) into one block, while the \textit{reshuffling map} $\textrm{resh}_{B_{|\xi|}(\xi),U_{B_{|\xi|}(\xi)}}$ 
independently relocates each partition element located in $B_{|\xi|}(\xi)$ to a new location in 
$B_{|\xi|}(\xi)$ that is randomly chosen. 

\begin{theorem}[{\bf Existence and uniqueness}]
\label{T.existence-coalescent}
For every $\pi \in \Pi_G$ the $(L^{(G)*}(\omega), \mathcal{C}_G, \delta_\pi)$-mar\-tingale problem is 
well-posed. \hfill $\square$
\end{theorem}

We denote the solution of the $(L^{(G)*}(\omega), \mathcal{C}^{(G)},\delta_\pi)$-martingale problem 
by 
\be{dualdef}
\mathfrak{C}^{(G)}(\omega) = \big( \mathfrak{C}^{(G)}(\omega;t) \big)_{t \geq 0},
\qquad \mathfrak{C}^{(G)}(\omega;0) = \delta_\pi.
\ee
For every $n\in\N$, when restricted to $\Pi_{G,n}$,  $\mathfrak{C}^{(G)}(\omega)$ becomes a 
strong Markov process $\mathfrak{C}_n^{(G)}(\omega)$ with the Feller property.

\subsection{Dualities}
\label{ss.dualities}

Consider the map
\be{qq:duality}
\begin{aligned}
H^{(n)}_\varphi(x , \pi_{G,n})
& = \int_{E^{b}}
\left(\bigotimes_{i=1}^{b} x_{\eta_{\pi^{-1}(i)}}\big(\dd u_i\big)\right)
\varphi\big(u_{\pi(1)},u_{\pi(2)},\ldots,u_{\pi(n)}\big),
\end{aligned}
\ee
where $n\in\N$, $\phi \in C_{\rm b}(\CP(E)^n)$, $x = (x_{\eta})_{\eta \in G}\in\CP(E)^{G}$, $\pi_{G,n} 
\in \Pi_{G,n}$, $b = b(\pi_{G,n}) = |\pi_{G,n}|$.

\begin{theorem}[{\bf Duality}]
\label{T.dual}
Fix $N\in\Ntwo$. For each of the choices $G$ in \eqref{Gchoices} and \eqref{Gchoices*},
\be{e1855} 
E[H^{(n)}_\varphi(X^G(\omega,t), \Pi_{G,n})]= E [H^{(n)}_\varphi(X^G,C^G(\omega,t))]
\ee
for all $n\in\N$ and $\phi \in C_{\rm b}(\CP(E)^n)$, where the same $\omega$ is used on both sides. \hfill $\square$
\end{theorem}

\noindent
This theorem is a consequence of the generator relation 
\be{generator-duality} 
\left(L^{(G)}(\omega) H^{(n)}_\varphi(\cdot,\pi_{G,n}) \right)(x) 
= \left( L^{(G)*}(\omega) H^{(n)}_\varphi(x,\cdot) \right)(\pi_G)
\quad \text{for $\mathbb{P}$-a.e.~$\omega$}.
\ee
This relation has been verified for the homogeneous model in \cite{GHKK14}, but here works 
the same.

\subsection{Well-posedness of the martingale problems and equilibria}
\label{ss.dualproofs}

Theorem~\ref{T.wpbasic} can be formulated for geographic spaces that are countable Abelian 
groups, in particular, the hierarchical group and the Euclidean lattice. For us the following 
generalization suffices.

\begin{theorem}[{\bf Well-posedness}]
\label{T.dualwellpos}
For each of the choices $G$ in \eqref{Gchoices} and \eqref{Gchoices*} the following holds: 
For $\bP$-a.e.\ $\omega$ and every $\pi \in \Pi_G$, the $(L^{(G)}(\omega), \mathcal{C}^{(G)}, 
\delta_{\pi})$-martingale problem is well-posed. \hfill $\square$
\end{theorem}

\begin{theorem}[{\bf Equilibrium}]
\label{T.duallongrun}
Fix $N\in\Ntwo$. Fix $n\in\N$ and start the $\mathfrak{C}^{(\Omega_N)}(\omega)$-process in 
a labelled partition $\{(\pi_i,\eta_i)\}_{i=1}^{n}$, where $\{\pi_i\}_{i=1}^n$ form a partition of $\N$ 
and $\{\eta_i\}_{i=1}^n$ represent the labels. If $x$ is a random state with mean $\theta \in \CP(E)$ 
whose law is invariant and ergodic under translations, then
\be{ag40}
\lim_{t\to\infty} \CL \left[H^{(n)}_\varphi\left(x, \mathfrak{C}^{(\Omega_N)}_n(\omega; t)\right)\right]
= \CL\left[H^{(n)}_\varphi\left(\underline{\theta},\mathfrak{C}^{(\Omega_N)}_n(\infty)\right)\right]
\quad \text{for $\mathbb{P}$-a.e.~$\omega$}
\ee
for all $\phi \in C_{\rm b}(\CP(E)^n)$. \hfill $\square$
\end{theorem}

In order to prove Theorem~\ref{T.duallongrun}, we follow the argument in \cite[Section 3]{DGV95}). 
The partition-valued process converges to a limiting partition. If the locations of the partition elements 
would follow a homogeneous random walk, then the key to the argument would be the averaging property one can prove via Fouries analysis
\be{ag14}
\lim_{t\to\infty} \sum_{\zeta \in \Omega_N} p_t(\eta,\zeta) f(\zeta) 
= \int_{\Omega_N} f(\xi)\nu(\dd\xi) \qquad \forall\,\eta \in \Omega_N,\,f \in C_b(\Omega_N),
\ee
where $p_t(\cdot,\cdot)$ is the time-$t$ transition kernel of the random walk on $\Omega_N$ and 
$\nu(\cdot)$ is the Haar measure on $\Omega_N$ (see Evans and Fleischmann~\cite{EF96}). 
We need to show that the same holds for our random walk in random environment $\omega$, which 
goes as follows.

Place a Poisson clock at every $\xi\in\Omega_N^{\mathbb{T}}\backslash\Omega_N$. 
Let the clock at $\xi$ ring at rate
\be{PC}
N^{-(|\xi|-1)} \big[c_{|\xi|-1} + N^{-1}\lambda^\xi(\omega)\big].
\ee 
At any moment of time let the random walk look at the ancestral line above its current position 
(see Fig.~\ref{fig-hierartree}) and redistribute itself uniformly over the block around its current 
position whose hierarchical label corresponds to the height of the first clock on that ancestral line 
that rings. The resulting random walk is the same as the hierarchical random walk in environment 
$\omega$ with migration coefficients given by \eqref{cNomega}. 

Next, let $K_t(\eta)$ be the highest hierarchical level at which prior to time $t$ a Poisson clock 
that lies on the ancestral line above $\eta$ has rang. Then at time $t$ the random walk starting 
from $\eta$ is uniformly distributed on the $K_t(\eta)$-block around $\eta$. Hence we have
\be{PC1}
\sum_{\zeta \in \Omega_N} p_t^\omega(\eta,\zeta) f(\zeta) 
= \sum_{k\in\N_0} P(K_t(\eta)=k) \left[ N^{-k} \sum_{\zeta \in B_k(\eta)} f(\zeta)
+ \sum_{\zeta \in \Omega_N \backslash B_k(\eta)} p_t^\omega(\eta,\zeta) f(\zeta)\right],
\ee
where $p_t^\omega(\cdot,\cdot)$ is the time-$t$ transition kernel of the random walk in $\omega$. 
Fix $\eta \in \Omega_N$ and $f \in C_b(\Omega_N)$. The first term between the square brackets 
in \eqref{PC1} tends to $\int_{\Omega_N} f(\xi)\nu(\dd\xi)$ as $k \to \infty$. The second term is 
bounded from above by  $\|f\|_\infty\,p_t^\omega (\eta,\Omega_N \backslash B_k(\eta))$, which 
tends to zero as $k\to\infty$. Finally, since all Poisson clocks ring at a strictly positive rate, we 
have 
\be{PC2}
P\left(\lim_{t\to\infty} K_t(\eta) = \infty\right) =1. 
\ee
It therefore follows that the right-hand side of \eqref{PC1} tends to the right-hand side of 
\eqref{ag14} as $t\to\infty$.

\subsection{Consequences for the Cannings process}
\label{ss.dualproofsthms}

The claims in Theorems~\ref{T.wpbasic}--\ref{T.ltbeq} follow from 
Theorems~\ref{T.dualwellpos}--\ref{T.duallongrun}. As argued in \cite{GHKK14}, the proof 
follows the strategy for the two-type case given in Evans~\cite[Theorem 4.1]{E97}, which says 
that for spatial coalescent processes well-posedness and existence carry over from the dual 
process to the original process. 

We next prove Corollary~\ref{c.1282}.

\begin{proof}
We analyze both $\nu_\theta^N$ and $\nu_\theta$ with the help of duality relations and show 
that the dual representation of the former converges to the dual representation of the latter.

\medskip\noindent
{\bf Step 1: $\nu_\theta$.} 
We have to construct a dual process for the entrance law of a Markov chain, namely, the interaction 
chain running from level $\infty$ down to level $0$. We consider the process that is dual to the 
interaction chain at level $j$. This dual process is a discrete-time Markov chain whose transition 
kernel we can determine, for fixed $j$ and in the limit as $j\to\infty$, via an explicit construction. 
This dual Markov chain is a spatial coalescent on $\{0,1,\ldots,j\}$, or on $\N$ when we consider 
all $j$ simultaneously and are interested in its limit state as $ j\to\infty$. 

We first focus on the dual transition kernel at one particular level. In the interaction chain this is 
defined via the equilibrium of the McKean-Vlasov process. How did this equilibrium arise? We 
consider a {\em mean-field} system of size $N^k$ with parameters $c_k$, $d_k$, $\Lambda_k$ 
and take the mean-field dual started in $n$ individuals at mutual distance $k$. This dual is shown 
to {\em converge}, in the limit as $N\to\infty$ and on time scale $t_N N^k$ with $t_N \to \infty$ 
and $t_N = o(N)$, to a limiting process that is a coalescent on the geographic space $k\cup
\{\bigtriangleup\}$, with $\bigtriangleup$ a cemetery state, such that the process jumps from 
$k$ to $\bigtriangleup$ at rate $c_k$ and does Kingman coalescence at rate $d_k$ and
$\Lambda$-coalescence according to $\Lambda_k(\mathrm{MC}_k(0))(\omega)$. This limiting process 
is run for infinite time to obtain the dual transition kernel at the $k$-th step. This partition at 
$\bigtriangleup$ is used as input for the next step of the dual with label $k+1$. Altogether this 
procedure defines the full Markov chain, i.e., the new site and the new partition element. We 
denote the path of the dual Markov chain by
\begin{equation}
\label{e1973}
(\Pi^\infty_k)_{k \in \N_0},\, \mbox{with $ \Pi^\infty_\infty $ its limiting state as   } k \to \infty.
\end{equation}
(Recall that partitions are ordered and hence the limiting state exists.) The line of argument is the 
same for each $k$. A detailed argument can be found in \cite[Corollary 2.12]{GHKK14}.

We need an explicit description as an $\N$-marked partition-valued process, namely, the above 
mentioned random walk, moving one step to the right on $\N$, doing Kingman coalescence at 
rate $d_k$ and $\Lambda$-coalescence according to $ \Lambda_k(\mathrm{MC}_k(\eta))$ 
in state $k$, provided the rate-$c_k$ clock does not ring first.

The dual chain after $j$ steps gives the expression $E_\theta[\langle M^{(j)}_0,f\rangle^n]
(\omega)$, which in the limit as $j\to\infty$ equals $\int^1_0 \Pi_\eta \nu_\theta (\omega)(\dd x)\,
\langle x_\eta,f \rangle^n$ by the definition of $\nu_\theta(\omega)$. The dual expectation is the 
expression $E[\langle \theta,f\rangle^{\mid \Pi_\infty^\infty \mid}](\omega)$. It therefore suffices 
to show that the latter is obtained from the dual representation of $\nu^N_\theta(\omega)$ as 
$N\to\infty$.

\begin{remark}
\label{r.1992}
{\rm What is the dual counterpart of Theorem~\ref{mainth}? The connection between the renormalized 
system and the interaction chain on the level of the dual is as follows. Consider the dual process for 
the $j$-level \emph{hierarchical} system for finite $N$, starting with $n$ partition elements at one site and 
letting $t\to\infty$ and $N\to\infty$ in the following way.  Consider time scales $(t^k_N)_{k\in\N_0}$ 
with $t^k_N/N^{k+1} \to 0$ and $ t^k_N/N^k\to\infty$ as $N\to\infty$. Then the coalescent reaches 
a partition $\Pi^{k+1}_\infty$, with the remaining partition elements in $B_k$ uniformly distributed. 
After that move to the next time scale. Finally, first take $N\to\infty$ and then take
\begin{equation}
\label{e1626}
\Pi^\infty_\infty \mbox{ as the limiting partition element for   } j \to \infty.
\end{equation}
By our scaling result in \eqref{ag:interaction-chain}, this object gives us the dual process of the 
interaction chain at level $j$.} \hfill $\square$
\end{remark}

\medskip\noindent
{\bf Step 2: $\nu^N_\theta$.}
We return to the representation of $\nu^N_\theta$, respectively, its marginal law at level $0$.
The convergence of the dual chain for the Cannings process on $\Omega_N$, and its limit as 
$t\to\infty$ followed by $N\to\infty$ to the dual chain of the interaction chain, will follow from 
the fact that the partitions become successively finer and hence converge to a limit partition, 
and the fact that the time scales for the random walk to reach distance $k$ separate as $N\to\infty$. 
Since the monomials are convergence determining, this will yield the claim.

We have to show that the coalescent on $\Omega_N$, starting with $n$ individuals at site 
$0$, converges to a limit process as $t\to\infty$, which we can investigate in the limit as 
$N\to\infty$. We need to show that this process has the property that, when we consider the 
times where the coalescent makes jumps to the next larger block, we get an embedded Markov 
chain with index in $\N_0$, describing the successive maximal jump sizes and values in partitions. 
The corresponding partition converges to $ \Pi^N_\infty $ as the index $ k $ tends to infinity. The 
claim is that, as $N\to\infty$, this Markov chain converges to a birth process in the first component, 
which moves one step to the right, with Kingman-coalescence at rate $d_k$ and 
$\Lambda$-coalescence according to $\Lambda_k(\mathrm{MC}_k(0))(\omega)$, provided the 
rate-$c_k$ clock does not ring first. This gives $ \Pi^\infty_k $. As $ k \to \infty $ we get \eqref{e1626}. 
This is done in Dawson and Greven~\cite{DG96} for the Kingman coalescent, but the necessary 
modification is straightforward.
\end{proof}

\section{Dichotomy: coexistence versus clustering}
\label{s.randomwalk}

In this section, we prove Theorem~\ref{T.coexcritNfin}. The question is whether $\mathfrak{C}^{(G)}
(\omega)$, the hierarchical coalescent in the environment $\omega$ defined in \eqref{dualdef}, 
converges to a single labelled partition element as $t\to\infty$ with probability one. To answer this 
question, we have to investigate \emph{whether two tagged partition elements coalesce with 
probability one or not}. Recall that, by the projective property of the coalescent, we may focus 
on the subsystem of just two dual individuals, because this translates into the same dichotomy 
for $\mathfrak{C}_n^{(G)}(\omega)$ for any $n\in\N$, and hence for the entrance law starting 
from $n$ partition elements. However, there is additional reshuffling at all higher levels, which is 
triggered by a corresponding block-coalescence event. Therefore, we need to consider two coalescing 
random walks with slightly {\em adapted} migration coefficients, lacking in particular the random 
walk property.

Recall the notation introduced in Sections~\ref{sss.hg}--\ref{sss.resh} and \ref{sss.RE}.
Recall that $\mathrm{MC}_k(\eta)$ is the unique site at height $k$ above $\eta\in\Omega_N$ 
(see Fig.~\ref{fig-hierartree}). Consider two independent copies 
\be{RWs}
Y(\omega) = (Y_t(\omega))_{t\geq 0}, \qquad 
Y^\prime(\omega) = (Y^\prime_t(\omega))_{t\geq 0},
\ee 
of the hierarchical random walk on $\Omega_N$ in the environment $\omega$ with migration 
coefficients given by \eqref{cNomega} and coalescence rates given by \eqref{lambdaxi}. Write 
$P^\omega, P^{\omega,\prime}$ for the marginal laws of  $Y(\omega),Y^\prime(\omega)$ and 
$\bar{P}^\omega = P^\omega \times P^{\omega,\prime}$ for the joint law of the pair $\bar{Y}
(\omega)=(Y(\omega),Y^\prime(\omega))$. Consider the time-$t$ accumulated hazard for 
coalescence:
\be{ak12}
H_N(\omega;t) = \sum_{k\in\N_0} N^{-k} 
\sum_{\substack{\eta, \eta^\prime \in \Omega_N \\ d_{\Omega_N}(\eta,\eta^\prime) \leq k}}
\lambda^{\mathrm{MC}_k(\eta)}(\omega) 
\int^t_0 1_{\{Y_s(\omega) = \eta,Y^\prime_s(\omega) = \eta^\prime\}}\, \dd s,
\ee
where we use that $\mathrm{MC}_k(\eta)=\mathrm{MC}_k(\eta^\prime)$ when $d_{\Omega_N}
(\eta,\eta^\prime) \leq k$. The rate $N^{-2k}$ to choose a $k$-block for the coalescence is 
multiplied by $N^k$ because all partition elements in that block can trigger a coalescence event, 
which explains the factor $N^{-k}$ in \eqref{ak12}. 

Let $\lim_{t\to\infty} H_N(\omega;t) = H_N(\omega;\infty)$. We have coalescence of the two random 
walks (``common ancestor'') with probability 1 when $H_N(\omega;\infty) = \infty$ $\bar{P}^\omega$-a.s., 
but separation of the two random walks (``different ancestors'') with positive  probability when 
$H_N(\omega;\infty) < \infty$ $\bar{P}^\omega$-a.s. In Section~\ref{ss.meanhazard} we identify 
the dichotomy for the mean hazard $\bar{E}^\omega[H_N(\omega;\infty)]$ combining Fouries analysis 
with potential theory of reversible Markov chains to handle the fact that our migration is no longer 
a random walk. In Section~\ref{ss.zero-one} we use a zero-one law to show that the same dichotomy 
holds for the hazard $H_N(\omega;\infty)$.

\subsection{Mean hazard}
\label{ss.meanhazard}

\bl{meanhazard}
For every $N \in \Ntwo$ and $\bP$-a.e.\ $\omega$, 
\be{dich}
\bar{E}^\omega[H_N(\omega;\infty)] = \infty 
\quad \Longleftrightarrow \quad 
\sum_{k\in\N_0} \frac{1}{c_k + N^{-1}\lambda_{k+1}} \sum_{l=0}^k \lambda_l = \infty.
\ee 
\hfill $\square$
\el

\bpr\label{p.1958}
Write
\be{trker}
p^\omega_t(\eta,\zeta) = P^\omega\{Y_t(\omega) = \zeta \mid  Y_0(\omega) = \eta\}, 
\quad \eta, \zeta \in \Omega_N,
\ee
to denote the time-$t$ transition kernel. Let 
\be{jointGreen}
G^\omega((0,0),(\eta,\eta^\prime)) = \int_0^{\infty} p^\omega_t(0,\eta) 
p^\omega_t(0,\eta^\prime)\, \dd t,
\qquad (\eta,\eta^\prime) \in \Omega_N \times \Omega_N,
\ee
denote the Green function for $\bar{Y}(\omega)$. Then \eqref{ak12} gives
\be{eq:expected-hazard}
\begin{aligned}
\bar{E}^\omega[H_N(\omega;\infty)] 
&= \sum_{k \in \N_0} N^{-k} 
\sum_{\substack{\eta, \eta^\prime \in \Omega_N \\ d_{\Omega_N}(\eta,\eta^\prime) \leq k}}
\lambda^{\mathrm{MC}_k(\eta)}(\omega)\, G^\omega((0,0),(\eta,\eta^\prime))\\
&= \sum_{k \in \N_0} N^{-k} \sum_{\xi \in \Omega_N^{(k)}} \lambda^\xi(\omega) 
\sum_{\substack{\eta \in \Omega_N \\ \mathrm{MC}_k(\eta)=\xi}}
\sum_{\eta' \in B_k(\eta)} G^\omega((0,0),(\eta,\eta^\prime))\\
&= \sum_{k \in \N_0} N^{-k} \sum_{\xi \in \Omega_N^{(k)}} \lambda^\xi(\omega) 
\sum_{\eta,\eta^\prime \in B_{|\xi|}(\xi)} G^\omega((0,0),(\eta,\eta^\prime))
\end{aligned}
\ee
(recall \eqref{Bkequivdef}). The proof comes in two steps. In Step 1, we pretend that the 
$\omega$-dependent term in the right-hand side of \eqref{cNomega} is replaced by its 
mean, i.e., the two hierarchical random walks are homogenous with migration coefficients 
$\bar{c}_k$ given by
\be{barck}
\bar{c}_k(N) = \bE[c_k(\omega)(N,\eta)] = c_k + N^{-1}\,\lambda_{k+1},
\ee
and show that the same dichotomy as in \eqref{dich} holds. In Step 2, we explain why this 
replacement does not affect the dichotomy. The Green function of the two homogeneous 
hierarchical random walks will be denoted by $G((0,0),(\eta,\eta^\prime))$.

\paragraph{Step 1.}
In what follows, we use the explicit form of the transition kernel $p_t(\eta,\zeta)$, $\eta,\zeta
\in \Omega_N$, for the \emph{homogeneous} hierarchical random walk computed in Dawson, 
Gorostiza and Wakolbinger~\cite{DGW05} with the help of Fourier analysis. Namely,
\be{ak:dgw-asympt}
p_t(0,\eta) = \sum_{j \geq k} K_{jk}(N)\,\frac{\exp[-h_j(N) t]}{N^j},
\qquad t \geq 0,\, \eta \in \Omega_N\colon\, d_{\Omega_N}(0,\eta)=k \in \N_0,
\ee
where
\be{Kjkdef}
K_{jk}(N) = \left\{\begin{array}{ll}
0, &j=k=0,\\
-1, &j=k>0,\\
N-1, &\mbox{otherwise},
\end{array}
\right.
\qquad j,k \in \N_0,
\ee
and
\be{hjrjrel}
h_j(N) = \frac{N-1}{N}\,r_j(N) + \sum_{i>j} r_i(N), \qquad j \in \N,
\ee
with
\be{ak:106}
r_j(N) = \frac{1}{D(N)} \frac{N-1}{N} \sum_{i \geq j} \frac{\bar{c}_{i-1}(N)}{N^{2i-j-1}},
\qquad j\in\N, 
\ee
where $D(N)$ is the normalizing constant such that $\sum_{j\in\N} r_j(N)=1$. Note that the 
expressions in \eqref{hjrjrel}--\eqref{ak:106} simplify considerably in the limit as $N\to\infty$, 
namely, the term with $i=j$ dominates and
\be{Ninfsimp}
h_j(N) \sim r_j(N) \sim \frac{\bar{c}_{j-1}(N)}{D(N) N^{j-1}}, \quad j \in \N, 
\qquad D(N) \sim \bar{c}_0(N).
\ee
Also note that, because of \eqref{ak:recurrence-cond} and \eqref{ak:lambda-growth-condition}, 
the following holds:
\begin{equation}
\begin{tabular}{ll}
&\text{For $N\in\Ntwo$ the quantities $h_j(N),r_j(N),D(N)$ are bounded from}\\ 
&\text{above and below by positive finite constants times the right-hand side}\\
&\text{of \eqref{Ninfsimp} uniformly in the index $j$.}
\end{tabular} \label{fdh:Ncomp}
\end{equation}

To compute the sum in \eqref{eq:expected-hazard}, we need to distinguish two cases: 
(1) $\xi = 0^k \in \Omega_N^{(k)}$, the unique site in $\Omega_N^{\mathbb{T}}$ at height 
$k$ above $0 \in \Omega_N$; (2) $\xi \in \Omega_N^{(k)}\backslash\{0^k\}$.

\paragraph{(1)} \underline{$\xi = 0^k$}. Write
\begin{equation}
\label{eq:ak147}
\sum_{\eta,\eta^\prime \in B_{|\xi|}(\xi)} G((0,0),(\eta,\eta^\prime))
= \sum_{0 \leq p,q \leq k}  N[p]N[q]\, G((0,0),(\eta^{(p)},\eta^{(q)})),
\end{equation}
where $\eta^{(p)}$ is any site in $\Omega_N$ such that $d_{\Omega_N}(0,\eta^{(p)}) = p$, and 
\be{Npdef}
N[p] = |B_p(0)\backslash B_{p-1}(0)| = \left\{\begin{array}{ll}
N^p-N^{p-1}, &p>0,\\
1, &p=0.
\end{array}
\right.
\ee
With the help of \eqref{ak:dgw-asympt} we obtain
\be{eq:ak*}
G((0,0),(\eta^{(p)},\eta^{(q)}))
= \sum_{m \geq p} \sum_{n \geq q}  K_{mp}(N) K_{nq}(N)\, N^{-m-n} \frac{1}{h_m(N)+h_n(N)}.
\ee
Inserting \eqref{Kjkdef} and \eqref{Ninfsimp}, we get
\begin{equation}
\mbox{r.h.s. } \eqref{eq:ak*} \sim \frac{\bar{c}_0(N)}{(1+1_{\{p=q\}}) \label{eq:ak**}
\bar{c}_{p \wedge q}(N) N^{p \vee q}}, \qquad N\to \infty,
\end{equation}
where the asymptotics comes from the terms with $m=p+1$ and $n=q+1$.
Combining (\ref{eq:ak147}--\ref{eq:ak**}), we obtain
\be{eq:ak145}
\begin{aligned}
&N^{-k} \sum_{\eta,\eta^\prime \in B_k(0^k)} G((0,0),(\eta,\eta^\prime))\\
&\qquad \sim N^{-k} \left(\sum_{0 \leq p \leq k}  N^{2p} \frac{\bar{c}_0(N)}{2 \bar{c}_p(N) N^p}
+ 2 \sum_{0 \leq p < q \leq k} N^{p+q} \frac{\bar{c}_0(N)}{\bar{c}_p(N) N^{q}}\right)
\sim \frac{\bar{c}_0(N)}{2 \bar{c}_k(N)},
\end{aligned}
\ee
where the asymptotics comes from the term with $p=k$.

\paragraph{(2)} \underline{$\xi \in \Omega_N^{(k)}\backslash\{0^k\}$}. Now $p_t(0,\eta)$ is the 
same for all $\eta \in B_{|\xi|}(\xi)$, and so we have
\begin{equation}
\label{eq:ak101a}
\begin{aligned}
&N^{-k} \sum_{\xi \in \Omega_N^{(k)}\backslash\{0^k\}} \lambda^\xi(\omega)
\sum_{\eta,\eta^\prime \in B_{|\xi|}(\xi)} G((0,0),(\eta,\eta^\prime))\\
&\qquad = N^k \sum_{\xi \in \Omega_N^{(k)}\backslash\{0^k\}}
\lambda^\xi(\omega)\, G((0,0),(\eta^{(k)},\eta^{(k)}))\\
&\qquad \sim N^k \sum_{d\in\N} 
\sum_{ \substack{\xi \in \Omega_N^{(k)} \\ d_{\Omega^{(k)}_N}(0^k,\xi) = d} }
\lambda^\xi(\omega)\, \frac{\bar{c}_0(N)}{2 N^{k+d} \bar{c}_{k+d}(N)}\\
&\qquad = \frac{1}{2} \sum_{d\in\N} \frac{\bar{c}_0(N)}{\bar{c}_{k+d}(N)}
\Bigg(\frac{1}{N^{d}} 
\sum_{ \substack{\xi \in \Omega_N^{(k)} \\ d_{\Omega^{(k)}_N}(0^k,\xi) = d} } 
\lambda^\xi(\omega)\Bigg),
\end{aligned}
\end{equation}
where we use (\ref{eq:ak*}--\ref{eq:ak**}) with $p=q=k$, and $d_{\Omega^{(k)}_N}$ denotes 
the distance within $\Omega_N^{(k)}$.

\medskip
Combining \eqref{lambdaxi}, \eqref{eq:expected-hazard}, \eqref{eq:ak145}--\eqref{eq:ak101a}, 
we arrive at
\be{meanhazfinal}
\bar{E}[H_N(\omega;\infty)] \sim \tfrac12\,\bar{c}_0(N) 
\sum_{k\in\N_0} \frac{1}{\bar{c}_k(N)} \sum_{l=0}^k
\lambda_l \Big\{\tfrac12\,1_{\{l=k\}}\,\Theta(\omega;0,k) + 1_{\{l<k\}}\,\Theta(\omega;k-l,l)\Big\},
\ee
where we abbreviate
\be{Thetadef}
\Theta(\omega;a,b) = \frac{1}{N[a]} 
\sum_{\substack{\xi\in\Omega^{(b)}_N \\ d_{\Omega_N^{(b)}}(0^b,\xi)=a}} 
\rho^\xi(\omega), \qquad a,b \in \N_0.
\ee
Now, by \eqref{Aprop} we have, for some $C<\infty$,
\be{Thetaprop}
\bE\big[\Theta(\omega;a,b)\big] = 1, \quad \bE\big[\Theta(\omega;a,b)\Theta(\omega;a',b')\big] 
\leq C \quad \forall\,a,b,a',b'\in\N_0.
\ee
Because $\{\rho^\xi(\omega)\colon\,\xi \in \Omega_N^{\mathbb{T}}\}$ is stationary, ergodic and 
tail trivial (recall \eqref{fdh:tail}), it follows from a standard second-moment estimate that the 
sum in the right-hand of \eqref{meanhazfinal} is infinite if and only if its expectation w.r.t.\ $\bP$ 
is infinite. Since
\begin{equation}
\bE(\mathrm{r.h.s.}\eqref{meanhazfinal}) = \sum_{k\in\N_0} \frac{1}{\bar{c}_k(N)} \sum_{l=0}^k
\lambda_l \big\{\tfrac12\,1_{\{l=k\}} + 1_{\{l<k\}}\big\},\label{expmeanhazfinal}
\end{equation}
we get the claim in \eqref{dich} for the hierarchical random walk with homogeneous migration 
coefficients $\bar{c}_k(N)$ defined in \eqref{barck} (the factor $\tfrac12$ is harmless for the 
convergence or divergence of the right-hand side of \eqref{expmeanhazfinal}).

\paragraph{Step 2.}
It remains to show that the same dichotomy holds for the coefficients in \eqref{cNomega} rather 
than \eqref{barck}. We start with the observation that the hierarchical random walk in random 
environment is \emph{symmetric} and therefore is {\em reversible} with respect to the 
{\em Haar measure} on $\Omega_N$. We have the representation (see Bovier and den 
Hollander~\cite[Chapter 7]{BdH15})
\be{fdh:jointGreenpottheo}
\begin{aligned}
G^\omega((0,0),(\eta,\eta^\prime)) 
&= \int_0^{\infty} p_t^\omega(0,\eta) p_t^\omega(0,\eta^\prime)\, \dd t\\
&= \frac{P^\omega_{(0,0)}(\tau_{(\eta,\eta^\prime)} < \infty)}
{a^\omega((\eta,\eta^\prime))\,P^\omega_{(\eta,\eta\prime)}(\hat\tau_{(\eta,\eta\prime)} = \infty)}, 
\quad (\eta,\eta^\prime) \in \Omega_N \times \Omega_N,
\end{aligned}
\ee
where $a^\omega((a,b)) = \sum_{(c,d)} a^\omega((a,b),(c,d))$ is the total rate at which the random 
walk jumps out of $(a,b)$, and 
\be{tauabdef}
\begin{aligned}
\tau_{(a,b)} &= \inf\big\{t \geq 0\colon\,Y_t(\omega) = (a,b)\big\},\\
\hat\tau_{(a,b)} &= \inf\big\{t \geq 0\colon\,Y_t(\omega) = (a,b),\,
\exists\,0<s<t\colon\,Y_s(\omega) \neq (a,b) \big\},
\end{aligned}
\ee 
are the first hitting time, respectively, the first return time of $(a,b)$. The point of 
\eqref{fdh:jointGreenpottheo} is that both the numerator and the denominator can be controlled with 
the help of the \emph{Dirichlet Principle}, as follows.

Let
\be{Diridef}
\CE^\omega(f,f) = \sum_{(a,b),(c,d)} \frac{a^\omega((a,b),(c,d))}{a^\omega((a,b))}
\,[f((a,b))-f((c,d))]^2
\ee
be the \emph{Dirichlet form} associated with the two random walks in random environment. 
By classical potential theory, the escape probability in the denominator of \eqref{fdh:jointGreenpottheo} 
is given by the capacity of the pair $(\eta,\eta^\prime)$ and $\infty$,
\be{pot2}
P^\omega_{(\eta,\eta\prime)}(\hat\tau_{(\eta,\eta\prime)} = \infty)
= \capa^\omega((\eta,\eta\prime), \infty)
= \inf_{ {f\colon\,\Omega_N \to [0,1]} \atop {f((\eta,\eta^\prime)) = 1, f(\infty) = 0} } \CE^\omega(f,f), 
\ee
where $f(\infty)=0$ stands for $\lim_{(\eta,\eta^\prime) \to \infty} f((\eta,\eta^\prime)) = 0$ with 
$(\eta,\eta^\prime) \to \infty$ short hand for $d_{\Omega_N}(0,\eta)+d_{\Omega_N}(0,\eta^\prime) 
\to \infty$ (recall \eqref{ag31}). The hitting probability in the numerator of \eqref{fdh:jointGreenpottheo} 
can also be expressed in terms of capacities after we use a renewal argument. Write
\be{pot3}
\begin{aligned}
P^\omega_{(0,0)}(\tau_{(\eta,\eta^\prime)} < \infty) 
&= \frac{P^\omega_{(0,0)}(\tau_{(\eta,\eta^\prime)} < \hat\tau_{(0,0)})}
{1-P^\omega_{(0,0)}(\hat\tau_{(0,0)} < \tau_{(\eta,\eta^\prime)})}\\
&= \frac{P^\omega_{(0,0)}(\tau_{(\eta,\eta^\prime)} < \hat\tau_{(0,0)})}
{P^\omega_{(0,0)}(\tau_{(\eta,\eta^\prime)} < \hat\tau_{(0,0)}) 
+ P^\omega_{(0,0)}(\tau_{(\eta,\eta^\prime)} = \hat\tau_{(0,0)} = \infty)}.
\end{aligned}
\ee 
We have
\be{pot4}
P^\omega_{(0,0)}(\tau_{(\eta,\eta^\prime)} < \hat\tau_{(0,0)}) = \capa^\omega((0,0),(\eta,\eta^\prime))
= \inf_{ {f\colon\,\Omega_N \to [0,1]} \atop {f((\eta,\eta^\prime)) = 1, f((0,0)) = 0} } \CE^\omega(f,f).
\ee
Moreover,
\be{pot5}
P^\omega_{(0,0)}(\tau_{(\eta,\eta^\prime)} = \hat\tau_{(0,0)} = \infty)
= P^\omega_{(0,0)}(\hat\tau_{(0,0)} = \infty) 
- P^\omega_{(0,0)}(\tau_{(\eta,\eta^\prime)} <\infty,  \hat\tau_{(0,0)} = \infty).
\ee
The first term equals $\capa^\omega((0,0), \infty)$, while the second term is bounded from above by 
$P^\omega_{(0,0)}(\tau_{(\eta,\eta^\prime)} <\infty)$, which tends to zero as $(\eta,\eta^\prime) \to 
\infty$ when $G^\omega<\infty$, i.e., when the random walk in random environment is transient.
Below we will show that, under assumption~\eqref{Apropalt}, {\em $G^\omega<\infty$ if and only if $G<\infty$}.

We are now ready to explain why the estimates in Step 1 carry over. 
The transition rates of the random walk in random environment are given by
\be{pot6}
a^\omega((a,b)(c,d)) = \left\{\begin{array}{ll}
a^{\omega,(N)}(a,c), &b=d,\\ 
a^{\omega,(N)}(b,d), &a=c,\\
0, &\text{else.}
\end{array}
\right.
\ee    
where $a^{\omega,(N)}$ is the transition kernel in \eqref{32b}, but with $c_k$ replaced by 
$c_k(\omega)(N,\eta)$ in \eqref{cNomega}: 
\be{pot7}
a^{\omega,(N)}(\eta,\zeta) = \sum_{k \geq d_{\Omega_N}(\eta,\zeta)} 
\frac{c_{k-1} + N^{-1} \lambda^{\mathrm{MC}_k(\eta)}(\omega)}{N^{2k-1}},
\qquad \eta,\zeta \in \Omega_N,\,\eta\neq\zeta, \qquad a^{\omega,(N)}(\eta,\eta) = 0.
\ee 
By \eqref{lambdaxi}, we have $\lambda^{\mathrm{MC}_k(\eta)}(\omega) 
= \lambda_k \rho^{\mathrm{MC}_k(\eta)}(\omega)$.
Assumption \eqref{Apropalt} implies $\delta \leq \lambda^{\mathrm{MC}_k(\eta)}(\omega)/\lambda_k$ 
$\leq\delta^{-1}$ for all $k\in\N_0$, $\eta\in\Omega_N$ and $\bP$-a.e.\ $\omega$, which in turn 
implies
\be{pot8}
\delta \leq \frac{a^\omega((a,b)(c,d))}{a((a,b)(c,d))} \leq \delta^{-1}
\quad \forall\,a,b,c,d \in \Omega_N \text{ for } \bP\text{-a.e. } \omega,
\ee  
where $a$ is the transition kernel in \eqref{32b}, but with $c_k$ replaced by $\bar{c}_k(N)$ in 
\eqref{barck}. Inserting these bounds into the formulas for the capacities in \eqref{pot2} and 
\eqref{pot4}, and recalling \eqref{fdh:jointGreenpottheo}, we see that
\be{pot9}
\exists\,\delta'>0\colon\quad \delta' \leq 
\frac{G^\omega((0,0),(\eta,\eta^\prime))}{G((0,0),(\eta,\eta^\prime))} \leq \delta'^{-1} 
\quad \forall\,\eta,\eta^\prime \in \Omega_N \text{ for } \bP\text{-a.e. } \omega.
\ee
This shows that the Green function for the random walk in random environment is comparable 
to the Green function of the homogeneous random walk. Hence the argument in Step 1 carries over. 

Note that $P^\omega_{(0,0)}(\hat\tau_{(0,0)} = \infty)$ in \eqref{pot5} is comparable to 
$P_{(0,0)}(\hat\tau_{(0,0)} = \infty)$, which is a strictly positive constant when $G<\infty$. 
Consequently, by the observation made below \eqref{pot5}, also $P^\omega_{(0,0)}(\tau_{(\eta,\eta^\prime)} 
= \hat\tau_{(0,0)} = \infty)$ in \eqref{pot5} is comparable to $P_{(0,0)}(\tau_{(\eta,\eta^\prime)} 
= \hat\tau_{(0,0)} = \infty)$ when $G,G^\omega<\infty$.

It remains to show that, under assumption~\eqref{Apropalt}, $G^\omega<\infty$ if and only if 
$G<\infty$. This is easy. Indeed, if $G<\infty$, then $P_{(0,0)}(\hat\tau_{(0,0)} = \infty)>0$, 
hence $P^\omega_{(0,0)}(\hat\tau_{(0,0)} = \infty)>0$, and hence $G^\omega<\infty$ by 
\eqref{fdh:jointGreenpottheo} because $P^\omega_{(0,0)}(\tau_{(\eta,\eta^\prime)} < \infty) \leq 1$.  
Conversely, if $G=\infty$, then $P_{(0,0)}(\hat\tau_{(0,0)} = \infty)=0$, hence $P^\omega_{(0,0)}
(\hat\tau_{(0,0)} = \infty)=0$, and hence $G^\omega=\infty$ by \eqref{fdh:jointGreenpottheo} 
because $P^\omega_{(0,0)}(\tau_{(\eta,\eta^\prime)} < \infty)>0$.
\hfill $\square$
\epr

\subsection{Zero-one law}
\label{ss.zero-one}

To conclude the proof of the dichotomy in Theorem~\ref{T.coexcritNfin}, we use the following 
zero-one law.
 
\begin{lemma}[{\bf Zero-one law}]
\label{zero-one}
For every $N \in \Ntwo$ and $\bP$-a.e.\ $\omega$, $H_N(\omega;\infty) = \infty$
if and only if $\bar{E}^\omega[H_N(\omega;\infty)] = \infty$. \hfill $\square$
\el

\bpr\label{pr.2231}
The proof comes in five steps.
\paragraph{Step 1.} 
For $M,N \in \N$, let $H_N^{(M)}(\omega;\infty)$ denote the truncation of $H_N(\omega;\infty)$ 
obtained by setting $\lambda_k=0$ for $k>M$ (no resampling in blocks of hierarchical size larger 
than $M$). The key to the proof is the following second-moment estimate:
\be{HNM2nd}
\quad \exists\,\,C<\infty\colon\qquad \bar{E}^\omega\big[\big(H_N^{(M)}(\omega;\infty)\big)^2\big] 
\leq C \big(\bar{E}^\omega\big[H_N^{(M)}(\omega;\infty)\big]\big)^2 
\qquad \forall\,M,N\in\N. 
\ee

Before proving \eqref{HNM2nd}, we complete the proof of Theorem~\ref{T.coexcritNfin}. 
By Cauchy-Schwarz, for any non-negative random variable $V$ we have
\be{VCS}
\bar{P}^\omega(V>0) \geq (\bar{E}^\omega[V])^2/\bar{E}^\omega[V^2].
\ee
Picking $V=H_N^{(M)}(\omega;\infty)/\bar{E}^\omega[H_N^{(M)}(\omega;\infty)]$ in \eqref{VCS} 
and using \eqref{HNM2nd}, we obtain
\be{HNM1}
\bar{P}^\omega\Big(H_N^{(M)}(\omega;\infty)/\bar{E}^\omega\big[H_N^{(M)}(\omega;\infty)\big] > 0\Big) 
\geq \frac{1}{C} \qquad \forall\,M,N \in \N.
\ee 
Since $H_N^{(M)}(\omega;\infty) \leq H_N(\omega;\infty)$ and the lower bound in \eqref{HNM1} is 
uniform in $M$ and $N$, it follows that if $\bar{E}^\omega[H_N(\omega;\infty)] = \lim_{M\to\infty} 
\bar{E}^\omega[H_N^{(M)}(\omega;\infty)] = \infty$, then $\bar{P}^\omega(H_N(\omega;\infty)=\infty) 
\geq 1/C$. By \eqref{ak12}, $\{\omega\colon\,H_N(\omega;\infty)=\infty\}$ is an element of the 
sigma-algebra at infinity defined in \eqref{fdh:tail}, which is trivial. The latter event therefore has 
probability either 0 or 1, and since it has positive probability we get the claim.  

\paragraph{Step 2.}
Write out (recall \eqref{ak12})
\be{2ndmomentwrite1}
\begin{aligned}
\bar{E}^\omega\big[\big(H_N^{(M)}(\omega;\infty)\big)^2\big] 
&= \sum_{k,l=0}^M N^{-k-l} \sum_{ {\eta,\eta'\in\Omega_N} \atop {d_{\Omega_N}(\eta,\eta') \leq k} }
\sum_{ {\zeta,\zeta'\in\Omega_N} \atop {d_{\Omega_N}(\zeta,\zeta') \leq l} }
\lambda^{\mathrm{MC}_k(\eta)}(\omega) \lambda^{\mathrm{MC}_l(\zeta)}(\omega)\\
&\qquad \times \bar{E}^\omega\left[\int_0^\infty 
\dd s\,1_{\{Y_s(\omega)=\eta,Y'_s(\omega)=\eta'\}} \int_0^\infty 
\dd u\,1_{\{Y_u(\omega)=\zeta,Y'_u(\omega)=\zeta'\}} \right]\\
& = \sum_{k,l=0}^M N^{-k-l} \sum_{ {\eta,\eta'\in\Omega_N} \atop {d_{\Omega_N}(\eta,\eta') \leq k} }
\sum_{ {\zeta,\zeta'\in\Omega_N} \atop {d_{\Omega_N}(\zeta,\zeta') \leq l} }
\lambda^{\mathrm{MC}_k(\eta)}(\omega) \lambda^{\mathrm{MC}_l(\zeta)}(\omega)\\
&\qquad \times 2\, G^\omega((0,0),(\eta,\eta'))\,G^\omega((\eta,\eta'),(\zeta,\zeta'))\\
&= 2 \sum_{k,l=0}^M N^{-k-l} \sum_{\xi \in \Omega^{(k)}_N} \lambda^{\xi}(\omega)
\sum_{\xi' \in \Omega^{(l)}_N} \lambda^{\xi'}(\omega)\\
&\qquad \times \sum_{\eta,\eta' \in B_{|\xi|}(\xi)} G^\omega((0,0),(\eta,\eta'))
\sum_{\zeta,\zeta' \in B_{|\xi'|}(\xi')} G^\omega((\eta,\eta'),(\zeta,\zeta')).
\end{aligned}
\ee
In what follows, we consider the hierarchical random walk with homogeneous migration coefficients 
$\bar{c}_k$ defined in \eqref{barck}. In Step 4 we incorporate the $\omega$-dependence. 

\medskip\noindent
Use symmetry to replace $\sum_{k,l=0}^M$ by $2 \sum_{k,l=0}^M 1_{\{k<l\}} + \sum_{k,l=0}^M 
1_{\{k=l\}}$. Due to the ultrametricity of the hierarchical distance and the isotropy of the hierarchical 
random walk, we have $G((\eta,\eta'),(\zeta,\zeta'))=G((0,0),(\zeta,\zeta'))$ for all $\eta,\eta' \in 
B_{|\xi|}(\xi)$ in the following three cases (where $\xi<\xi'$ means that $\xi'$ is an ancestor of $\xi$):

\medskip
\begin{tabular}{ll}
&(1) $k<l$, $\xi \nless \xi'$ and $\zeta,\zeta' \in B_{|\xi'|}(\xi')$.\\
&(2) $k<l$ and $\xi < \xi'$ and $\zeta,\zeta' \in B_{|\xi'|}(\xi') \backslash B_{|\xi|}(\xi)$.\\
&(3) $k=l$, $\xi \neq \xi'$ and $\zeta,\zeta' \in B_{|\xi'|}(\xi')$. 
\end{tabular}

\medskip\noindent
Therefore we have
\be{2ndmomentwrite2}
\bar{E}\big[\big(H_N^{(M)}(\omega;\infty)\big)^2\big] 
= 2\bar{E}\big[H_N^{(M)}(\omega;\infty)\big]^2 + R
\ee
with $R$ a correction term given by
\be{Rdef}
\begin{aligned}
R 
&=  4 \sum_{0 \leq k < l \leq M} 
N^{-k-l} \sum_{\xi \in \Omega^{(k)}_N} \lambda^{\xi}(\omega)
\sum_{\xi' \in \Omega^{(l)}_N} \lambda^{\xi'}(\omega) \,1_{\{\xi<\xi'\}}
\sum_{\eta,\eta' \in B_{|\xi|}(\xi)} G((0,0),(\eta,\eta'))\\
&\qquad \qquad \times 
\sum_{ {\zeta,\zeta' \in B_{|\xi'|}(\xi')} \atop {\zeta \in B_{|\xi|}(\xi) 
\text{ and/or } \zeta' \in B_{|\xi|}(\xi)} }
\big[G((\eta,\eta'),(\zeta,\zeta'))-G((0,0),(\zeta,\zeta'))\big]\\
&\quad + 2 \sum_{0 \leq k \leq M} N^{-2k} 
\sum_{\xi \in \Omega^{(k)}_N} [\lambda^{\xi}(\omega)]^2 
\sum_{\eta,\eta' \in B_{|\xi|}(\xi)} G((0,0),(\eta,\eta'))\\
&\qquad\qquad\times \sum_{\zeta,\zeta' \in B_{|\xi|}(\xi)} 
\big[G((\eta,\eta'),(\zeta,\zeta'))-G((0,0),(\zeta,\zeta'))\big].
\end{aligned}
\ee
If $R$ would be absent from \eqref{2ndmomentwrite2}, then we would have proved \eqref{HNM2nd} 
with $C=2$. Thus, it remains to show that $R$ can only raise the constant. We will do this by showing 
that $R \leq O(N^{-2})\,\bar{E}[H_N^{(M)} (\omega;\infty)]^2$ as $N\to\infty$, uniformly in $M$, and 
by appealing to the observation made in \eqref{fdh:Ncomp}.

\paragraph{Step 3.}
By translation invariance, $G((\eta,\eta'),(\zeta,\zeta')) = G((0,0),(\zeta-\eta,\zeta'-\eta'))$. By isotropy, 
$\sum_{\zeta,\zeta' \in B_{|\xi|}(\xi)} G((0,0),(\zeta-\eta,\zeta'-\eta'))= \sum_{\zeta,\zeta' \in B_{|\xi|}(\xi)} 
G((0,0),(\zeta,\zeta'))$ for all $\eta,\eta' \in B_{|\xi|}(\xi)$. Hence, in the first sum in \eqref{Rdef} the term 
with $\zeta,\zeta' \in B_{|\xi|}(\xi)$ vanishes, while the second sum in \eqref{Rdef} vanishes altogether, 
and so $R$ simplifies to
\be{Rsym1}
\begin{aligned}
R &= 8 \sum_{0 \leq k < l \leq M} 
N^{-k-l} \sum_{\xi \in \Omega^{(k)}_N} \lambda^{\xi}(\omega)
\sum_{\xi' \in \Omega^{(l)}_N} \lambda^{\xi'}(\omega) \,1_{\{\xi<\xi'\}}
\sum_{\eta,\eta' \in B_{|\xi|}(\xi)} G((0,0),(\eta,\eta'))\\
&\qquad \qquad \times 
\sum_{ {\zeta \in B_{|\xi|}(\xi)} \atop {\zeta' \in B_{|\xi'|}(\xi') \backslash B_{|\xi|}(\xi)} }
\big[G((0,0),(\zeta-\eta,\zeta'-\eta'))-G((0,0),(\zeta,\zeta'))\big].
\end{aligned}
\ee
By isotropy, $\sum_{\zeta \in B_{|\xi|}(\xi)} G((0,0),(\zeta-\eta,\zeta'-\eta')) = \sum_{\zeta \in B_{|\xi|}(\xi)} 
G((0,0),(\zeta,\zeta'-\zeta))$ for all $\eta,\eta' \in B_{|\xi|}(\xi)$ when $\zeta' \in B_{|\xi'|}(\xi') \backslash 
B_{|\xi|}(\xi)$, and so $R$ simplifies further to
\be{Rsym2}
\begin{aligned}
R &= 8 \sum_{0 \leq k < l \leq M} 
N^{-k-l} \sum_{\xi \in \Omega^{(k)}_N} \lambda^{\xi}(\omega)
\sum_{\xi' \in \Omega^{(l)}_N} \lambda^{\xi'}(\omega) \,1_{\{\xi<\xi'\}}
\sum_{\eta,\eta' \in B_{|\xi|}(\xi)} G((0,0),(\eta,\eta'))\\
&\qquad \qquad \times 
\sum_{\zeta \in B_{|\xi|}(\xi)}
\left[ \sum_{\zeta'' \in B_l(0) \backslash B_k(0)} G((0,0),(\zeta,\zeta''))
-\sum_{\zeta' \in B_{|\xi'|}(\xi') \backslash B_{|\xi|}(\xi)} 
G((0,0),(\zeta,\zeta'))\right].
\end{aligned}
\ee
If $0 \in B_{|\xi'|}(\xi')$, then $B_l(0) = B_{|\xi'|}(\xi')$, in which case the term between brackets equals
\be{bracket}
\sum_{\zeta' \in B_{|\xi|}(\xi)} G((0,0),(\zeta,\zeta'))
- \sum_{\zeta'' \in B_k(0)} G((0,0),(\zeta,\zeta'')).
\ee
If also $0 \in B_{|\xi|}(\xi) \subset B_{|\xi'|}(\xi')$, then also $B_k(0) = B_{|\xi|}(\xi)$, in which case the 
latter difference vanishes. Hence we obtain the bound
\be{Rsym3}
\begin{aligned}
R &\leq 8 \sum_{0 \leq k < l \leq M} 
N^{-k-l} \sum_{\xi \in \Omega^{(k)}_N} \lambda^{\xi}(\omega)
\sum_{\xi' \in \Omega^{(l)}_N} \lambda^{\xi'}(\omega) \,1_{\{\xi<\xi'\}} 
\sum_{\eta,\eta' \in B_{|\xi|}(\xi)} G((0,0),(\eta,\eta'))\\
&\times 
\left[1_{\{0 \nleq \xi, \,0 \leq \xi'\}} 
\sum_{\zeta,\bar\zeta \in B_{|\xi|}(\xi)} G((0,0),(\zeta,\bar\zeta))
+ 1_{\{0 \nleq \xi'\}} 
\sum_{ {\zeta \in B_{|\xi|}(\xi)} \atop {\bar\zeta \in B_l(0)\backslash B_k(0)} } 
G((0,0),(\zeta,\bar\zeta))
\right].
\end{aligned}
\ee 
The sums over $\eta,\eta'$ and $\zeta,\bar\zeta$ can be computed with the help of \eqref{eq:ak*}. 
Recalling \eqref{eq:ak**}--\eqref{eq:ak101a}, we obtain
\be{rest2}
\begin{aligned}
&\sum_{\eta,\eta' \in B_{|\xi|}(\xi)} G((0,0),(\eta,\eta'))
= \sum_{\zeta,\bar\zeta \in B_{|\xi|}(\xi)} G((0,0),(\eta,\eta'))\\
&\qquad \sim N^{2k}\,\frac{\bar{c}_0(N)}{2N^{k+d(\xi)}\bar{c}_{k+d(\xi)}(N)}
= N^k \frac{\bar{c}_0(N)}{2N^{d(\xi)}\bar{c}_{k+d(\xi)}(N)}
\end{aligned}
\ee
with $d(\xi) = d_{\Omega_N^{(k)}}(0^k,\xi)$ and
\be{rest3}
\begin{aligned}
&\sum_{ {\zeta \in B_{|\xi|}(\xi)} \atop {\bar\zeta \in B_l(0) \backslash B_k(0)} } 
G((0,0),(\zeta,\bar\zeta))
\sim \sum_{ {\zeta \in B_{|\xi|}(\xi)} \atop {\bar\zeta \in B_l(0) \backslash B_k(0)} } 
\frac{\bar{c}_0(N)}{N^{l+d'(\xi')}\bar{c}_{d''(\bar\zeta)}(N)}\\
&\qquad = N^{-l+k} \sum_{\bar\zeta \in B_l(0) \backslash B_k(0)} 
\frac{\bar{c}_0(N)}{N^{d'(\xi')}\bar{c}_{d''(\bar\zeta)}(N)}
\sim N^k \frac{\bar{c}_0(N)}{N^{d'(\xi')}\bar{c}_l(N)}
\end{aligned}
\ee
with $d'(\xi')= d_{\Omega_N^{(l)}}(0^l,\xi')$ and $d''(\bar\zeta)=d_{\Omega_N}(0,\bar\zeta)$. 
Here we use that $\xi \neq 0^k$ when $0 \nleq \xi$ and $\xi' \neq 0^l$ when $0 \nleq \xi'$, and 
also that $l+d'(\xi')>d''(\bar\zeta)$ for all $\bar\zeta \in B_l(0)$. Inserting \eqref{rest2}--\eqref{rest3} 
into \eqref{Rsym3}, we get
\be{rest4}
\begin{aligned}
R &\leq 8\,[1+o(1)] \sum_{0 \leq k < l \leq M} N^{k-l} 
\sum_{\xi \in \Omega^{(k)}_N} \lambda^{\xi}(\omega)
\sum_{\xi' \in \Omega^{(l)}_N} \lambda^{\xi'}(\omega) \,1_{\{\xi<\xi',\,0 \nleq \xi'\}}\\
&\times
\left[ 
1_{\{0 \nleq \xi, \,0 \leq \xi'\}} 
\left(\frac{\bar{c}_0(N)}{2N^{d(\xi)}\bar{c}_{k+d(\xi)}(N)} \right)^2
+ 1_{\{0 \nleq \xi'\}}\,
\frac{\bar{c}_0(N)}{2N^{d(\xi)}\bar{c}_{k+d(\xi)}(N)}
\,\frac{\bar{c}_0(N)}{N^{d'(\xi')}\bar{c}_l(N)}
\right].
\end{aligned}
\ee

\paragraph{Step 4.}
If $0 \nleq \xi$, then $d(\xi) \in \N$. Hence the first part of \eqref{rest4} equals $8\,[1+o(1)]$ times
\be{rest5}
\sum_{0 \leq k < l \leq M} N^{k-l} \lambda_k \lambda_l \sum_{d=1}^l 
\left(\frac{\bar{c}_0(N)}{2N^d\bar{c}_{k+d}(N)}\right)^2
\sum_{ {\xi \in \Omega^{(k)}_N} \atop {d(\xi)=d} } \rho^{\xi}(\omega)\rho^{\xi^{l-k}}(\omega),
\ee 
where we recall \eqref{lambdaxi} and write $\xi^{l-k}$ to denote the ancestor of $\xi$ at height $l$. 
Because $\{\rho^\xi(\omega)\colon\,\xi \in \Omega_N^{\mathbb{T}}\}$ is stationary, ergodic and tail 
trivial (recall \eqref{fdh:tail}), the last sum scales as $\sim N^d \E[\rho^{0^k}(\omega)\rho^{0^l}(\omega)]$, 
where the expectation is finite because of \eqref{Aprop}. Hence \eqref{rest5} is
\be{rest6}
\leq C[1+o(1)]\,\tfrac14\bar{c}_0(N)^2
\sum_{0 \leq k < l \leq M} N^{k-l} \lambda_k \lambda_l \sum_{d=1}^l 
\frac{1}{N^d\bar{c}_{k+d}(N)^2}. 
\ee  
The last sum scales as $\sim 1/N\bar{c}_{k+1}(N)^2$, and so \eqref{rest5} is
\be{rest7}
\leq C[1+o(1)]\,\tfrac14\bar{c}_0(N) 
\sum_{k=0}^M \frac{\lambda_k}{\bar{c}_{k+1}(N)^2}
\sum_{l=k+1}^M \lambda_l N^{k-l-1} 
= O(N^{-2})\,\bar{E}[H_N^{(M)}(\omega;\infty)]^2,
\ee 
where the equality follows from \eqref{meanhazfinal} with $k,l$ truncated at $M$. 

\medskip\noindent
If $0 \nleq \xi'$, then $d'(\xi') \in \N$ and $d(\xi)=l-k+d'(\xi')$. Hence the second part of 
\eqref{rest4} equals $8\,[1+o(1)]$ times
\be{rest8}
\sum_{0 \leq k < l \leq M} N^{k-l} \lambda_k \lambda_l \sum_{d'\in\N}
\frac{\bar{c}_0(N)}{2N^{l-k+d'}\bar{c}_{l+d'}(N)}\,\frac{\bar{c}_0(N)}{N^{d'}\bar{c}_l(N)}
\sum_{ {\xi \in \Omega^{(k)}_N} \atop {d(\xi)=l-k+d'} } \rho^{\xi}(\omega)\rho^{\xi^{l-k}}(\omega).
\ee 
The last sum is $\leq C[1+o(1)] N^{l-k+d'}$. Hence \eqref{rest8} is 
\be{rest9}
\leq C[1+o(1)]\,\tfrac12\bar{c}_0(N)^2
\sum_{0 \leq k < l \leq M} N^{k-l} \lambda_k \lambda_l \sum_{d'\in\N}
\frac{1}{N^{d'}\bar{c}_{l+d'}(N)\bar{c}_l(N)}.  
\ee
The last sum scales as $\sim 1/N\bar{c}_{l+1}(N)\bar{c}_l(N)$, and so \eqref{rest8} is
\be{rest10}
\leq C[1+o(1)]\,\tfrac12\bar{c}_0(N)^2
\sum_{k=0}^M \lambda_k \sum_{l=k+1}^M 
\frac{\lambda_l}{\bar{c}_{l+1}(N)\bar{c}_l(N)}\,N^{k-l-1} 
= O(N^{-2})\,\bar{E}[H_N^{(M)}(\omega;\infty)]^2.  
\ee

\paragraph{Step 5.}
We can again use \eqref{pot9} to show that the proof carries over to the random walk in random 
environment.
\epr

\medskip\noindent
Lemmas~\ref{meanhazard}--\ref{zero-one} combine to yield Theorem~\ref{T.coexcritNfin} 
(recall the discussion at the beginning of this section).

\section{Multi-scale analysis}
\label{s.mkvrand}

In this section we prove Theorem~\ref{mainth}. We first consider a {\em mean-field system}, i.e., 
the geographic space is $G=\{1,\ldots, N\}$ with $N \to \infty$. In Section~\ref{ss.mkvlim} we 
look at this system on time scale $t$ (on which the single components evolve) and on time 
scale $Nt$ (on which the block average evolves). In Section~\ref{ss.hierarmfl} we use the results 
to analyze the system on $\Omega_N$ as $N \to\infty$. Our general strategy runs parallel to that 
in \cite{GHKK14} for the homogeneous model. We only point out which {\em new issues} arise. 
Thus, this section is \emph{not autonomous}, 
the principal steps of the arguments are given but not all formulas are repeated, and for an 
understanding of the fine details the reader must check the relevant passages in \cite{GHKK14}.
 
\subsection{The mean-field finite-system scheme}
\label{ss.mkvlim} 

As geographic space and transition kernel we take
\be{ag45}
\Omega = \{1,\dots,N\}, \qquad a(i,j) = \frac{1}{N}, \quad i,j \in \Omega.
\ee
As migration rate we take $c_0$, and as resampling measures
\be{ag45alt}
\Lambda^i = \lambda_0\chi^i,\quad i\in \Omega, 
\ee
with total masses $\rho^i = \chi^i((0,1])$. We assume that $(\chi^i)_{i \in \N}$ is stationary and 
ergodic such that $\varrho^i$ has mean $1$. We also allow a component with Fleming-Viot 
resampling at rate $ d_0 $. The corresponding stochastic system is denoted by $(Z^{(N)}(t))_{t \geq 0}$ 
with $Z^{(N)}(t) = (Z_1^{(N)}(t),\ldots,Z^{(N)}_N(t))$. 

We consider time scales $t$ and $Nt$ for the components, respectively, the block average:
\be{ag46}
\begin{aligned}
&(Z^{(N)}(t))_{t \geq 0},\\
&(\bar Z^{(N)}(t))_{t \geq 0} \text{ with } 
\bar Z^{(N)}(t) = \frac{1}{N} \suml^N_{i=1} Z^{(N)}_i (Nt).
\end{aligned}
\ee

\begin{theorem}[{\bf [Mean-field finite-system scheme}]
\label{T.McKVl}
Suppose that the initial state is i.i.d.\ with mean measure $\theta \in \CP(E)$. Then
\be{ag47}
\lim_{N\to\infty} \CL \left[(Z^{(N)}(t))_{t \geq 0}\right] 
= \bigotimes_{i \in \N} \CL \left[(Z^{c_0, d_0, \Lambda^i}_\theta (t))_{t \geq 0}\right]
\ee
and 
\be{ag48}
\lim_{N\to\infty} \CL \left[(\bar Z^{(N)}(t))_{t \geq 0}\right] 
= \CL \left[(Z^{0,d_1,0}_\theta (t))_{t \geq 0}\right],
\ee
where $(Z_\theta^{c,d,\Lambda}(t))_{t\geq 0}$ is the McKean-Vlasov process defined in 
Section~{\rm \ref{sss.mfl}}. 
\hfill $\square$
\et

\begin{proof}\label{pr.2519}
We follow \cite[Section~6]{GHKK14}. The proof of \eqref{ag47} carries over in a straightforward 
way. In the proof of \eqref{ag48} a new issue arises: the increasing process of the limit process 
incorporates an additional averaging over the random environment controlling the resampling for 
the single components. This is handled as follows. 

Calculate the generator for a polynomial of $\bar Z^{(N)}(t)$, namely, a function $F$ of the form
\be{ag49}
F(z) = \langle f,z^{\otimes n}\rangle, \quad f \in C_b(E^n,\R),
\ee
applied to a $z \in E$ of the form $z = \frac{1}{N} \sum_{i=1}^N z_i$. This expression can be 
expanded in terms of sums of products of monomials of single components. The action of the 
generator was calculated and analysed in \cite[Section~6]{GHKK14}. We can argue in the 
same way with the following changes. In the action of the generator, integrals are taken with 
respect to the random sequence of resampling measures $(\Lambda^i)_{i \in \Omega}$ rather 
than a fixed resampling measure $\Lambda$. This entails that for the block average we get a 
sum of terms where the random sequence $(\rho^i)_{i\in\Omega}$ appears as weights. This in 
turn requires us to change the definition of the set of configurations on which the generator 
converges in the limit as $ N\to\infty$ (see \cite[Eq.~(6.41)--(6.42)]{GHKK14}) as follows.

Let $\B^\ast$ be the set of $\ux = (x_i)_{i\in\Omega}\in \CP(E)^\N$ with
\begin{align}
\label{ag60}
\lim_{N\to\infty} \CL\left[ \frac{1}{N} \suml^N_{i=1} \delta_{(\chi^i(\omega),x_i)} \right]
= \Gamma \in \CP\big(\CM_f([0,1]) \times \CP(E)\big), 
\end{align}
where
\begin{align}
\label{ag61}
\Gamma(\,\cdot\,,\CP(E)) = \CL[\chi^1],\quad \Gamma (\dd x_1 \mid \chi^1) 
= \nu^{c_0,0,\chi^1}(\dd x_1), \quad x_1 \in \CP(E).
\end{align} 
In order to calculate the sum of the resampling operators as in \cite[Eq.~(6.46)]{GHKK14}, we have 
to account for the presence of $\chi^i$, $i \in \Omega$, and invoke the law of large numbers for the 
expression in the variance formula, namely, $2c_0/(2c_0+\lambda_0\rho^i +2d_0)$, $i\in\Omega$. 
We write the latter as 
\be{ag61alt} 
\frac{c_0}{c_0+\mu_0\rho^i+d_0}, \qquad i \in \Omega.
\ee 
The expressions appearing in the generator, which are averages of local functions of the configuration 
and their shifts to any of the $N$ locations, result in the same expression as the one we obtain by 
using \eqref{ag61alt} averaged over $i\in\Omega$. In the limit as $N\to\infty$ this leads to the recursion 
formula in \eqref{diffusion-constants} for $k=0$. With these changes, the argument runs as in the case 
of the homogeneous environment.
\end{proof}

\subsection{The hierarchical mean-field limit}
\label{ss.hierarmfl} 

In this section we prove the results claimed in Section~\ref{sss.blav}. The strategy of the proof is to 
approximate our system with \emph{infinitely} many hierarchies of components and time scales by 
systems with \emph{finitely} many hierarchies of components and time scales, uniformly in $N$. The 
latter are analyzed by using the multiscale analysis of the mean-field system. In Section~\ref{sss.2l} 
we consider 2-level systems with $N^2$ components, in Section~\ref{sss.kl} $k$-level systems with 
$N^k$ components, and in Section~\ref{sss.3l} we pass to the limit $k\to\infty$ of infinitely many 
hierarchies. Along the way we make frequent reference to Dawson, Greven and Vaillancourt~\cite{DGV95} 
and the work on the homogeneous version of the model in \cite{GHKK14}. 

\subsubsection{The $2$-level system on 3 time scales}
\label{sss.2l}

The geographic space is $G_{N,2}=\{0,1,\ldots,N-1\}^2 = G_{N,1}^2$. We pick $d_0>0$,
$c_0,c_1,\mu_0,\mu_1>0$ and $c_k,\mu_k=0$ for $k \geq 2$.  We choose the random environment 
that is obtained by restricting the random environment of Section~\ref{s.model} to the subtree 
corresponding to the 2-block around 0. We show that, on time scales $t$ and $Nt$, we obtain the 
same limiting objects as described in Section~\ref{ss.mkvlim}, but with additional volatility and 
block resampling. 

For the 1-block averages we use the notation
\be{add21} 
Y^{(N)}_\eta(t) = N^{-1} \sum_{\sigma \in G_{N,1}} X_{(\sigma,\eta)}^{(N)}(t),
\qquad \eta \in G_{N,1},
\ee
and for the 2-block average (= total average) 
\be{add20}
Z^{(N)}(t) = N^{-2} \sum_{(\sigma,\eta) \in G_{N,2}} X_{(\sigma,\eta)}^{(N)}(t)
= N^{-1} \sum_{\eta \in G_{N,1}} Y^{(N)}_\eta(t).
\ee

\begin{proposition}[{\bf [Two-level rescaling}]
\label{P.2lresc}
Under the above assumptions,
\be{ak15new}
\begin{aligned}
\lim_{N\to\infty} \CL\left[\left(X^{(N)}_{(\sigma,\eta)}(t)\right)_{t \geq 0}\right] 
&= \CL\left[\left(Z^{c_0, d_0, \Lambda^{\mathrm{MC}_1
(\cdot,\eta)}(\omega)}_\theta (t)\right)_{t \geq 0}\right]
\qquad \forall\, (\sigma,\eta) \in G_{N,2},\\
\lim_{N\to\infty} \CL\left[\left(Y^{(N)}_\eta (Nt)\right)_{t \geq 0}\right] 
&= \CL\left[\left(Z^{c_1, d_1, \Lambda^{\mathrm{MC}_2
(\cdot,\cdot)}(\omega)}_\theta (t)\right)_{t \geq 0}\right]
\qquad \forall\, \eta \in G_{N,1},\\ 
\lim_{N\to\infty} \CL\left[\left(Z^{(N)} (N^2 t)\right)_{t \geq 0}\right]
&=\CL\left[\left(Z^{0,d_2,0}_\theta (t)\right)_{t \geq 0}\right],
\end{aligned}
\ee
with
\be{ak16new}
d_1 = \E_{\CL_\rho} \left[\frac{c_0 (\mu_0 \rho(\omega) + d_0)}
{c_0 + (\mu_0 \rho(\omega) + d_0)}\right],
\qquad 
d_2 = \E_{\CL_\rho} \left[\frac{c_1 (\mu_1 \rho(\omega) + d_1)}
{c_1 + (\mu_1 \rho(\omega) + d_1)}\right].
\ee 
\hfill $\square$
\end{proposition}

To prove the above results in the homogeneous environment, we used {\em uniform estimates} 
for higher-order perturbations of generators. These no longer hold in the random environment, 
due to the unboundedness of the random resampling rates $\rho^{(\cdot,\eta)}(\omega)$. (There 
is no problem under assumption \eqref{Apropalt}, and the proof carries over from \cite{GHKK14}.) 

To handle this problem we first consider the system where the coefficients $ \lambda^{\mathrm{MC}_k
(\cdot,\eta)} (\omega)$, $k=1,2$, are truncated at level $M<\infty$. For this system we show, with 
the help of a coupling argument, that on time scale $N^k t$, $k=1,2$, and averaged over the 
random environment and the dynamics, the effect of the truncation goes to zero as $M\to\infty$. 
The same holds for the limiting objects, so that we get the claim by using the existence of the 
expectation in combination with the stationarity of $\omega$.

To get tightness of the approximating sequence of processes, as in \cite[Eq.~(7.52), p.~117]{GHKK14}, 
we use the fact that the laws conditioned on the environment $ \omega $ of the averages in 
\eqref{add21}--\eqref{add20} are tight. To prove the latter, we use the criterion of Joffe and Metivier 
in the form as given in Dawson~\cite[p.~55]{D93}, observing that $\chi^{\mathrm{MC}_k(\cdot,\eta)}
(\omega)$, $\eta \in G_{N,1}$, $k=1,2$, are integrable uniformly in $N$. To check the criterion, we 
observe that we can code the information on the random environment into the initial condition of 
the process. With this observation, the proof works as for the homogeneous environment.

\subsubsection{The $k$-level system on $k+1$ time scales}
\label{sss.kl}

The reasoning addresses the same points raised above and runs otherwise exactly as in 
\cite[Section 7.2]{GHKK14}.

\subsubsection{The infinite-level system on infinitely many time scales}
\label{sss.3l}

The problem is again the extension of the uniform perturbation arguments, which have 
to be adapted to guarantee that cutting off higher hierarchical levels leads to an approximation 
by finite systems, for which we can apply the reasoning in the previous section, on the relevant 
time scales. To get the necessary arguments and estimates we refer the reader to the material 
in \cite[Sections 8.1--8.2]{GHKK14}.

The argument used for the homogeneous environment to obtain uniforms bounds does not apply 
because the perturbation of the migration and the resampling coming from the hierarchical levels 
$\geq k+1$ is unbounded. However, the perturbation terms can be stochastically bounded by a 
random variable that has a {\em finite expectation} over the random environment. Again, it suffices to 
show with the help of a coupling argument that the stochastic dynamics with $k$ hierarchical 
levels approximates the infinite stochastic dynamics on time scales $tN^l$ with $0 \leq l \leq k$. 
Apart from that the argument is the same.

\subsection{Dichotomy in the hierarchical mean-field limit}
\label{ss.HMFLdicho}

In this section we prove Theorem~\ref{T.dicho}. First, we argue that the entrance law exists, a fact 
that was established in Dawson, Greven and Vaillancourt~\cite{DGV95}[Section 6(a), Proposition 6.2] 
for the Fleming-Viot model, based on a variance estimate and the convergence of the sum in the 
coexistence criterion. The argument from that paper carries over despite the $\omega$-dependence 
of the transition kernels of the interaction chain (read this of from \eqref{e2691} and \eqref{e2749alt} below).

Next, we argue that the dichotomy holds. Here, we again follow the strategy for the homogeneous 
environment by calculating the variance of $\langle M^{(j)}_{\eta,0},f \rangle$ for every 
$\eta\in\Omega_\infty$ and $f \in C_b(E,\R)$ and showing that as $j\to\infty$ this variance
converges to zero, respectively, remains positive, depending on whether the sum in \eqref{inftydich} 
is infinite or finite. 

The variance formula reads
\begin{equation}\label{e2690}
\var_{\nu_\theta^{c,d,\Lambda}} (\langle\,\cdot\,,f \rangle) 
= \frac{2c}{2c+\lambda\rho(\omega)+2d} \var_\theta(f).
\end{equation}
Consequently, by iteration, 
\begin{equation}\label{e2691}
\var \langle M_{\eta,0}^{(j)}, f\rangle 
= \left[\prod^{j}_{k=0} \frac{2c_k}{2c_k + \lambda_k\rho_k(\omega) + 2d_k}\right] 
\var_\theta(f),
\end{equation}
where $\ud=(d_k)_{k \in \N_0} $ is determined by the recursion relation in \eqref{diffusion-constants}.
Taking logarithms, we see that the product tends to a positive limit as $j\to\infty$ if and only if
\begin{equation}
\label{e2749}
\sum_{k\in\N_0} \frac{1}{c_k} (\mu_k\rho_k(\omega) + d_k) < \infty.
\end{equation}
By assumptions \eqref{Aprop}--\eqref{fdh:tail}, the sum converges $\omega$-a.s.\ if and only if
\begin{equation}
\label{e2749alt}
\sum_{k\in\N_0} \frac{1}{c_k} (\mu_k + d_k) < \infty.
\end{equation}
Indeed, the variance of the sum in \eqref{e2749} equals the variance of the $\rho$-field times 
$\sum_{k\in\N_0} (\frac{\mu_k}{c_k})^2$, and the latter is bounded from above by the square 
of the average of the sum. As shown in \cite[Theorem 3.7(c)]{GHKK14}, the criterion in \eqref{e2749alt} 
is the same as the criterion in \eqref{inftydich}.

\section{The orbit of the renormalization transformations}
\label{s.completeproof}

In Section~\ref{ss.volcomp} we show the ordering in Theorem~\ref{T.order}. 
In Sections~\ref{ss.scalpol}--\ref{ss.scalexp}, we derive the scaling behaviour 
in Theorems~\ref{T.scaleFVpol}--\ref{T.scaleFVexp}.

\subsection{Random environment lowers the volatility}
\label{ss.volcomp}

\begin{proof}[Proof of Theorem~\textup{\ref{T.order}}]
Recall the notation introduced in Section~\ref{ss.random}. Fix $\uc$ and $\ul$.
Let $\ud$ be the solution of the recursion relation in \eqref{diffusion-constants}. 
Let $\ud^0,\ud^1$ be the solutions when $\CL_\rho$ is replaced by $\delta_0,\delta_1$ 
(recall that $\rho$ has mean 1 under $\CL_\rho$).  As initial values take $d^0_0 \leq 
d_0 \leq d^1_0$. We use induction on $k$ to show that $d^0_k < d_k < d^1_k$ for 
all $k\in\N$.

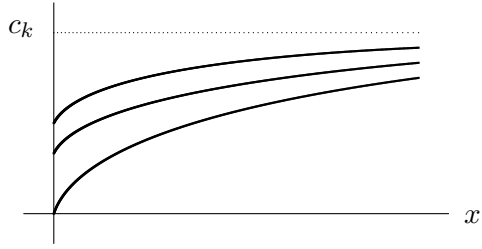
\begin{figure}[htbp]
\vspace{-.5cm}
\begin{center}
\setlength{\unitlength}{0.4cm}
\begin{picture}(14,10)(0,-1)
\put(-1,0){\line(1,0){14}}
\put(0,-1){\line(0,1){8}}
{\thicklines
\qbezier(0,0)(1,3)(12,4.5)
\qbezier(0,2)(1,4)(12,5)
\qbezier(0,3)(1,5)(12,5.5)
}
\qbezier[80](0,6)(5,6)(12,6)
\put(-1.5,6){$c_k$}
\put(13.5,-.3){$x$}
\end{picture}
\caption{\small Qualitative pictures of $x \mapsto f^0_k(x)$ (bottom), $x \mapsto f_k(x)$ (middle) 
and $x \mapsto f^1_k(x)$ (top). All three functions are strictly increasing and strictly concave 
on $[0,\infty)$, and tend to $c_k$ as $x \to \infty$.}
\label{fig-Mobius}
\end{center}
\end{figure}

Define (see Fig.~\ref{fig-Mobius})
\be{fdh:f_k}
f^0_k(x) = \frac{c_kx}{c_k+x},
\qquad
f_k(x) = \bE_{\CL_\rho}\left[\frac{c_k(\mu_k\rho+x)}{c_k+(\mu_k\rho+x)}\right],
\qquad f^1_k(x) = \frac{c_k(\mu_k+x)}{c_k+(\mu_k+x)}.
\ee
Because $a \mapsto c_k(\mu_ka+x)/[c_k+(\mu_ka+x)]$ is strictly increasing and strictly concave 
on $[0,\infty)$ for all $x \in [0,\infty)$, it follows that $f^0_k(x)<f_k(x) < f^1_k(x)$ for all 
$x \in [0,\infty)$. Hence, if $d^0_k \leq d_k \leq d^1_k$, then $d^0_{k+1} = f^0_k(d^0_k)< f_k(d^0_k) 
\leq f_k(d_k) = d_{k+1}$ and $d_{k+1} = f_k(d_k) < f^1_k(d_k) \leq f^1_k(d^*_k) = d^1_{k+1}$.      
\end{proof}

The same argument proves the claim made in Section~\ref{ss.effectre} that $M<M^*$ for the fixed 
points of \eqref{MK} (random environment) and its analogue with $\CL_\rho$ replaced by 
$\delta_1$ (average environment).

\subsection{Scaling of the volatility: polynomial coefficients}
\label{ss.scalpol}

\begin{proof}[Proof of Theorem~\textup{\ref{T.scaleFVpol}}]\label{pr.2742}
We look at each of the four parameter regimes separately.
Recall \eqref{fdh:regcond}--\eqref{Kdef}.

\medskip\noindent
(a) Let $K_k=\mu_k/c_{k-1}$, $R_k=c_k/c_{k-1}$ and $\mho_k=d_k/c_{k-1}$. 
Rewrite \eqref{diffusion-constants} as
\be{fdh:recrew}
\mho_{k+1} = g_k(\mho_k) \quad \text{ with } \quad 
g_k(x) = \bE_{\CL_\rho}\left[\frac{(K_k\rho+x)}{R_k+(K_k\rho+x)}\right].
\ee
Since $g_k$ is non-decreasing on $[0,\infty)$, we have the sandwich 
\be{fdh:sandwich}
g_k(0) \leq  \mho_{k+1} \leq g_k(\infty) = 1.
\ee
We are in the regime where $\lim_{k\to\infty} K_k = K = \infty$ and $\lim_{k\to\infty} R_k = R = 1$. 
Hence $\lim_{k\to\infty} g_k(0) = 1$, and so \eqref{fdh:sandwich} yields $\lim_{k\to\infty} d_k/c_k 
= \lim_{k\to\infty} \mho_k/R_k = 1/R = 1$.

\medskip\noindent
(b) Again use \eqref{fdh:recrew}. We are in the regime where $\lim_{k\to\infty} K_k = K \in (0,\infty)$ 
and $\lim_{k\to\infty} R_k = R = 1$. Hence, we see that $g_k$ converges point-wise to $g$ given 
by
\be{fdh:flim}
g(x) = \bE_{\CL_\rho}\left[\frac{(K\rho+x)}{R+(K\rho+x)}\right].
\ee
Both $g$ and $g_k$ are strictly increasing and strictly concave on $[0,\infty)$, with $g([0,\infty]) 
\subseteq [0,1]$ and $g_k([0,\infty]) \subseteq [0,1]$, with unique attracting fixed points $M \in (0,1)$ 
and $M_k \in (0,1)$, and with $M$ the solution of \eqref{MK}. To show that $\lim_{k\to\infty} \mho_k 
= M$, we need two facts.

\begin{lemma}
\label{sk}
Let $s_k = \sup_{x \in [0,1]} |g_k(x) - g(x)|$. Then $\lim_{k\to\infty} s_k = 0$. 
\hfill $\square$
\end{lemma}

\begin{proof}\label{pr.2774}
Estimate
\be{sk1}
\begin{aligned}
&\frac{(K_k\rho+x)}{R_k+(K_k\rho+x)}-\frac{(K\rho+x)}{R+(K\rho+x)}
= \frac{R}{R+(K\rho+x)} - \frac{R_k}{R_k+(K_k\rho+x)}\\
&= \frac{(RK_k-R_kK)\rho+(R-R_k)x}{[R+(K\rho+x)][R_k+(K_k\rho+x)]}
\leq \frac{(RK_k-R_kK)}{KR_k} + \frac{R-R_k}{RR_k}\,x.
\end{aligned}
\ee
This gives
\be{sk2}
s_k \leq \frac{R}{K}\left|\frac{K_k}{R_k}-\frac{K}{R}\right| + \left|\frac{1}{R_k} - \frac{1}{R}\right|.
\ee
Let $k\to\infty$ to get the claim.
\end{proof}

\begin{lemma}
\label{contr}
Function $g$ is a strict contraction around $M$, i.e., there exists a $\beta \in (0,1)$ such that 
$\sup_{x \in [0,\infty)} (g(x)-M)/(x-M) = \beta$. \hfill $\square$
\end{lemma}

\begin{proof}\label{pr.2796}
Consider the linear function $L(x) = g(0) + [1-\frac{g(0)}{M}]x$, $x \in [0,\infty)$, which satisfies 
$L(0)=g(0)$ and $L(M) = M = g(M)$ (see Fig.~\ref{fig-comparison}). Note that $g \geq L$ on $[0,M]$ 
while $g \leq L$ on $[M,\infty)$. Hence, we have
\be{contr1}
0 \leq \frac{g(x)-M}{x-M} \leq \frac{L(x)-M}{x-M} = 1 - \frac{g(0)}{M}. 
\ee
Since $g(0)>0$, we get the claim with $\beta=1-\frac{g(0)}{M}$.
\end{proof}

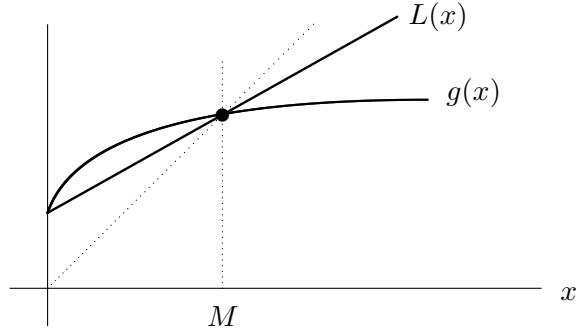
\begin{figure}[htbp]
\begin{center}
\setlength{\unitlength}{0.5cm}
\begin{picture}(14,10)(0,-1.5)
\put(-1,0){\line(1,0){14}}
\put(0,-1){\line(0,1){8}}
{\thicklines
\qbezier(0,2)(1,5)(10,5)
\qbezier(0,2)(4.6,4.6)(9.2,7.2)
}
\qbezier[60](0,0)(3,3)(7,7)
\qbezier[40](4.6,0)(4.6,3)(4.6,6)
\put(10.5,5){$g(x)$}
\put(9.5,7){$L(x)$}
\put(13.5,-.3){$x$}
\put(4.2,-1){$M$}
\put(4.6,4.6){\circle*{0.35}}
\end{picture}
\caption{\small Comparison of $g$ and $L$.}
\label{fig-comparison}
\end{center}
\end{figure}

\noindent
We can now complete the proof as follows. Let  $\Delta_k = |\mho_k-M|$. Then
\be{compl1}
\begin{aligned}
\Delta_{k+1} &= |\mho_{k+1}-M| \leq |\mho_{k+1}-g(\mho_k)| + |g(\mho_k)-M|\\
&= |g_k(\mho_k)-g(\mho_k)| + |g(\mho_k)-M| \leq s_k + \beta\Delta_k.   
\end{aligned}
\ee
Iteration yields
\be{compl2}
\Delta_{k+1} \leq \sum_{l=0}^k \beta^l s_{k-l} + \beta^{k+1}\Delta_0.
\ee
It follows from Lemma~\ref{sk}--\ref{contr} that $\lim_{k\to\infty} \Delta_k = 0$.
Hence $\lim_{k\to\infty} d_k/c_k = \lim_{k\to\infty} \mho_k/R_k = M/R = M$. 

\medskip\noindent
(c--d) Like in Case~(a), the scaling turns out to be the \emph{same} as for the average 
environment. The proof is based on a \emph{comparison} between the recursions for the 
random environment and the average environment (last two items in \eqref{fdh:f_k}). The 
key idea is the following lemma, which can be viewed as a \emph{stability property}.

\begin{lemma}
\label{lem:scal1}
Let $d_0 = d^1_0$. Then, the solution of the recursion $d_{k+1}=f_k(d_k)$, $k\in\N_0$, is 
the same as the solution of the recursion $d^1_{k+1}=f^1_k(d^1_k)$, $k\in\N_0$, when in 
the latter recursion the coefficient  $\mu_k$ is replaced by $\mu_kr_k$ with
\begin{equation}
\label{rkdef}
r_k = \frac{N_k}{D_k}, \qquad N_k = \bE_{\CL_\rho}\left[\frac{c_k\rho}{c_k(1+K_k\rho)+d_k}\right],
\qquad D_k = \bE_{\CL_\rho}\left[\frac{c_k}{c_k(1+K_k\rho)+d_k}\right].
\end{equation} 
\hfill $\square$
\end{lemma}
 
\bpr\label{pr.2860}
Check that
\begin{equation}\label{e2894}
\frac{c_k(\mu_kr_k+d_k)}{c_k+(\mu_kr_k+d_k)} 
= \bE_{\CL_\rho}\left[\frac{c_k(\mu_k\rho+d_k)}{c_k + (\mu_k\rho+d_k)}\right]
= d_{k+1},
\end{equation}
and use induction on $k$.
\epr 

\noindent
Since $\rho \mapsto c_k/[c_k(1+K_k\rho)+d_k]$ is non-increasing, we have $N_k \leq D_k
\bE_{\CL_\rho}[\rho] = D_k$, and so $r_k \leq 1$. The following result shows that $r_k$ tends 
to 1 as $k \to \infty$ in Cases (c) and (d). 

\begin{lemma}
\label{lem:scal2}
If $\lim_{k\to\infty} K_k = K = 0$, then $\lim_{k\to\infty} r_k=1$. 
\hfill $\square$
\end{lemma}

\bpr\label{pr.2879}
For any $C \in (0,\infty)$, we may estimate
\begin{equation}\label{e2913}
N_k \geq \frac{c_k}{c_k(1+K_kC)+d_k}\,\bE_{\CL_\rho}\left[\rho\,1_{\{\rho \leq C\}}\right],
\qquad D_k \leq \frac{c_k}{c_k+d_k}.
\end{equation}
Since $\lim_{k\to\infty} K_k=0$, we have $\lim_{k\to\infty} (c_k+d_k)/[c_k(1+K_kC)+d_k]
=1$, and hence 
\begin{equation}\label{e2919}
\liminf_{k\to\infty} \frac{N_k}{D_k} \geq   \bE_{\CL_\rho}\left[\rho\,1_{\{\rho \leq C\}}\right].
\end{equation}
Now let $C\to\infty$ and use that $\lim_{C\to\infty} \bE_{\CL_\rho}[\rho\,1_{\{\rho \leq C\}}] 
= \bE_{\CL_\rho}[\rho]=1$ by monotone convergence. 
\epr 
 
\medskip\noindent
Lemma~\ref{lem:scal1} implies that the scaling of $d_k$ is the same as the scaling of $d^1_k$ 
after $\mu_k$ is replaced by $\mu_kr_k$. But the latter scaling was derived in \cite{GHKK14}, 
and a glance at the results for Cases~(c) and (d) obtained there shows that the scaling is 
unaffected by the extra factor $r_k$ because of Lemma~\ref{lem:scal2}.
\epr

A technical remark is in order, for which we refer the reader to \cite[Section~11.3]{GHKK14}. 
We have assumed that $k \mapsto \mu_k$ is regularly varying at infinity (recall \eqref{fdh:regcond}). 
Because $\lim_{k\to\infty} r_k = 1$, also $k \mapsto r_k\mu_k$ is regularly varying at infinity. 
Therefore, $(r_k\mu_k)_{k\in\N_0}$ can be approximated from above and from below by sequences 
that have the same scaling behaviour but are \emph{smoothly varying}, i.e., for all $n\in\N$ their 
$n$-th order discrete differences are regularly varying as well. This approximation is harmless 
because the maps $\underline{c} \mapsto \underline{d}$ and $\underline{\mu} \mapsto \underline{d}$ 
are component-wise non-decreasing (a fact that is immediate from \eqref{diffusion-constants}), and 
so the approximating sequences provide a \emph{sandwich} for the scaling. Now, if the tail exponent 
of $r_k\mu_k$ is non-integer, i.e., $b \notin \N$ in \eqref{fdh:regcond}, then for all $n\in\N$ the $n$-th 
order discrete differences are \emph{asymptotically monotone}. This observation is important because 
it implies that certain sequences arising in \cite[Section~11.3]{GHKK14} have \emph{summable variation}, 
a property that is crucial for the proof of the scaling. If the tail exponent is integer, i.e., $b \in \N$ in 
\eqref{fdh:regcond}, then the asymptotic monotonicity still holds for all $n \leq b$, which turns out to 
be enough for the argument.

The extra regularity conditions on $L_c,L_\mu$ in \eqref{fdh:regcond}, which are stated in 
\cite[Eqs.\ (1.79)--(1.81)]{GHKK14}, need no modification: $(r_k\mu_k)_{k\in\N_0}$ has the same 
slowly varying function $L_\mu$ as  $(\mu_k)_{k\in\N_0}$.

\subsection{Scaling of the volatility: exponential coefficients}
\label{ss.scalexp}

\bpr[Proof of Theorem~\textup{\ref{T.scaleFVexp}}]\label{pr.2914}
We look at each of the five parameter regimes (= universality classes) separately.
Recall (\ref{expregvar}--\ref{barK}).

\paragraph{(A)} 
Use \eqref{fdh:recrew}. We are in the regime where $\lim_{k\to\infty} K_k=K=\infty$ and 
$\lim_{k\to\infty} R_k = c$. The same argument as in the proof of Case~(a) yields 
$\lim_{k\to\infty} d_k/c_k = \lim_{k\to\infty} \mho_k/R_k = 1/c$.

\paragraph{(B)} 
Let $\bar{K}_k = \bar{\mu}_k/\bar{c}_{k-1}$ and $\bar{R}_k = \bar{c}_k/\bar{c}_{k-1}$. Then 
$K_k = c \bar{K}_k$ and $R_k = c \bar{R}_k$ by \eqref{expregvar}, and so \eqref{fdh:recrew} 
becomes
\be{fdh:recrewbar}
\mho_{k+1} = \bar{g}_k(\mho_k) \quad \text{ with } \quad 
\bar{g}_k(x) = \bE_{\CL_\rho}\left[\frac{(c \bar{K}_k\rho+x)}{c \bar{R}_k+(c \bar{K}_k\rho+x)}\right].
\ee
We are in the regime where $\lim_{k\to\infty} \bar{K}_k = \bar{K} \in (0,\infty)$ and $\lim_{k\to\infty} 
\bar{R}_k = \bar{R} = 1$. The same argument as in Case~(b) therefore yields $\lim_{k\to\infty} 
d_k/c_k = \lim_{k\to\infty} \mho_k/R_k = \bar{M}/c\bar{R} = \bar{M}/c$ with $\bar{M}$ the unique 
attracting fixed point of
\be{fdh:barflim}
\bar{g}(x) = \bE_{\CL_\rho}\left[\frac{(c\bar{K}\rho+x)}{c+(c\bar{K}\rho+x)}\right],
\ee
which is the analogue of \eqref{fdh:flim}.

\paragraph{(C1)} 
This case is the same as Case~(B), but with $\bar{K}=0$. The analogue of \eqref{fdh:barflim} reads 
$\bar{g}(x) = x/(c+x)$. Since $\bar{g}$ has $\bar{M}=1-c \in (0,1)$ as unique attracting fixed point, 
we can copy the proof of Case~(b) to get $\lim_{k\to\infty} d_k/c_k = \lim_{k\to\infty} \mho_k/R_k 
= (1-c)/c\bar{R} = (1-c)/c$. \emph{Note}: In the proof of Case~(b) we used that $g(0)>0$, which fails 
here. However, even when $d_0=0$, the iterates $d_k$, $k\in\N$, are bounded away from $0$ 
because the attracting fixed points of $f_k$, $k\in\N$, are bounded away from $0$. Hence we may 
restrict the entire argument to $[\epsilon,1]$ for some $\epsilon>0$ instead of $[0,1]$, and use 
that $g(\epsilon)>0$ (recall Fig.~\ref{fig-comparison}).

\paragraph{(C2)} 
This case is like Case~(c). Since $\bar{K}=0$, we can copy the proof of Case~(c) and show that 
the same scaling holds as in the average environment.

\paragraph{(C3)} 
This case is like Case~(d). Since $\bar{K}=0$, we can copy the proof of Case~(d) and show that 
the same scaling holds as in the average environment.  
\epr

\section{Identification of the universality classes of cluster formation}
\label{s.clustering}

In this section we prove Theorem~\ref{T.cluform}. In Section~\ref{s.size} we deal with cases 
(a), (A) and (b), (B), (C1), in Section~\ref{s.order} with cases (c), (C2) and (d), (C3). 
The strategy of proof is the same as for the homogeneous environment, except at a few points 
where the random environment comes into play seriously. We focus on the necessary 
modifications. Like Section~\ref{s.mkvrand}, this section is \emph{not completely autonomous}, and for an 
understanding of the fine details the reader must check the relevant passages in \cite{GHKK14}.

Before we begin we recall why we may choose the starting configuration to be identically equal 
to $\theta$, the mean of the starting configuration. The initial state and the environment of our 
Cannings process are such (recall Theorem~\ref{mainth}) that the scaling limit in 
\eqref{macroscopic-behaviour} yields on average $\theta$ on level $j+1$.

\subsection{Random cluster size}
\label{s.size}

\begin{proof}[\bf Proof of cases \textup{(b), (B), (C1), (C3)[first subcase]}]
In Step 1 we give the proof for an i.i.d.\ random environment. In Step 2 we extend the proof to 
a stationary and ergodic random environment.

\paragraph{Step 1.} 
We consider the set $\CM_f([0,1]) \times \CP(E)$, describing the environment and the state of a block. 
If the random environment is i.i.d., then the sequence
\be{ag50}
\Big(\chi^{(\eta,j+1-\alpha)}(\omega), M^{(j)}_{-(j+1-\alpha)}\Big)_{\alpha \in \N_0}
\ee
is a time-inhomogeneous Markov chain. Let $(K^{\ast,(j)}_\alpha)_{\alpha\in\N_0}$ be its sequence of 
transition kernels. (We suppress the index $\eta$ from $M^{(j)}_{\eta,-(j+1-\alpha)}$ because its law is 
the same for all $\eta\in\Omega_N$.) It suffices to prove three properties: 
\begin{itemize}
\item[(1)] 
The sequence of transition kernels $(K^{\ast,(j)}_\alpha)_{\alpha\in\N_0}$ converges as $j \to \infty$ 
to the sequence $(K^{\ast,\infty}_\alpha)_{\alpha\in\N_0}$ of transition kernels given by
\be{ag51}
K^{\ast,\infty}_\alpha((\chi, \theta), \cdot) 
= \CL \left[\chi^\alpha \otimes \nu_\theta^{1,\wt M, 2 \wt K \chi^{\alpha}(\omega)}\right] (\cdot),
\ee
where $\wt M,\wt K$ are defined in \eqref{wtMwtK} and $(\chi_\alpha(\omega))_{\alpha\in\N_0}$ in 
\eqref{regi3}.
\item[(2)] The map 
\be{ag52}
((0,\infty) \times (0,\infty) \times \CP([0,1])) \ni 
(c,d,\Lambda) \mapsto \nu^{c,d, \Lambda}_\theta \in  \CP(\CP(E))
\ee
is continuous.
\item[(3)] The map 
\be{ag53}
\CP(E) \ni \theta \mapsto \nu^{c, d, \Lambda}_\theta \in \CP(\CP(E))
\ee
is continuous.
\end{itemize}
Items (1) and (3) imply the convergence of the process in \eqref{ag50}, while item (2) is needed in 
the proof item (1). 

\begin{proof}[Proof of \textup{(1)--(3)}]\label{pr.1-3}
Here a key is the {\em duality relation} for the McKean-Vlasov limit process. This duality arises as 
a special case of our duality relation by choosing a suitable geographic space. This coalescent is 
obtained by taking as space $\{0,\ast\}$, where the rates for all transitions in $ \ast $ are zero (cemetery) 
state and jumps occur from $0$ to $\ast$ at rate $c$. Kingman coalescence occurs at rate $d$ and 
the $\Lambda$-coalescence is given via $\Lambda$ (all as long as we are in $0$). For a detailed discussion,
see \cite[Section 4]{GHKK14}.

With the help of duality we can identify the equilibrium measure $\nu_\theta^{c,d,K\chi}$ by using a 
measure-determining sequence of test functions. The parameters $c,d,\chi $ enter via the rate of 
jump to the cemetery state (parameter $c$), the rate of pairwise coalescence (parameter $d$), 
and the rate of coalescence (parameter $\chi$). In the latter, the ratio $\chi/\chi((0,1])$ determines 
the probability for partition elements to coalesce in groups ($\Lambda$-coalescence). In this  
equilibrium representation, the coalescent has run for {\em infinite} time.

\medskip\noindent
(1) With $\CL$ acting on $ \chi_{j+1-\alpha}(\omega)$, we have
\begin{equation}\label{e3052}
K^{\ast,(j)}_\alpha((\chi,\theta),\cdot) = \CL \left[ (\chi^{j+1-\alpha}) \otimes \nu_\theta^{1,d_{j+1-\alpha}/c_{j+1-\alpha}, 2K_{j+1-\alpha} \chi^{j+1-\alpha}(\omega)}\right] (\cdot).
\end{equation}
From Theorems~\eqref{T.scaleFVpol} and \eqref{T.scaleFVexp}, we know that $d_{j+1-\alpha}
/c_{j+1-\alpha}$ and $K_{j+1-\alpha}$ converge to $\widetilde{M}$ and $\widetilde{K}$ as 
$j\to\infty$. The point is to show for every $\alpha\in\N_0$ the equilibrium measure in the 
right-hand side converges as $j\to\infty$. By the stationarity of the random environment, the 
law of $\chi_{j+1-\alpha}(\omega)$ is independent of $j$. Hence (2) and (3) yield the claim.

\medskip\noindent
(2) The continuity in \eqref{ag52} can be deduced from the dual representation in the McKean-Vlasov 
limit dynamic, in particular, from the fact that the coalescent has run for infinite time, and depends 
continuously on the migration rate $c$ and the Kingman coalescence rate $d$, respectively, the 
rates for the $ \Lambda$-coalescence. The coalescent has a monotone 
decreasing number of partition elements off the cemetery where all rates are zero and reaches 
the cemetery state after a finite time.
This means we have a Markov chain hitting a trap in finite time and therefore depends continuously on the finitely many involved jump-rates.

\medskip\noindent
(3) The continuity in \eqref{ag53} is deduced from the dual representation. We have to show that 
the dual expectation depends continuously on $\theta$, which goes as follows. First note that the 
monomials $ \{\,\langle\cdot\,,f\rangle^\ell\colon\,f \in C_b (E,\R),\,\ell \in\N\}$ are measure-determining 
on $(E,\CB)$. The dual expectation is a finite sum over terms arising from partition elements that 
are coalescing before jumping to the cemetery state. If $\ell$ partition elements remain, then the 
$\theta$-dependence is via $\langle \theta,f \rangle^\ell$, which is a continuous function of $\theta$.
\end{proof}

\paragraph{Step 2.} 
To deal with a stationary and ergodic random environment, we {\em condition} on the sequence
$(\chi_\alpha)_{\alpha\in\N_0}$. This leads to a sequence of Markov chains in random environment, 
indexed by $j$, for which the result in (1) holds, as explained above. After that we argue that 
(1)--(3) again imply the claim, because of the stationarity and the fact that we need only consider finite 
$\alpha$.

Next, we consider the finite-dimensional laws of the Markov chain in random environment conditional 
on $(\chi_\alpha)_{\alpha\in\N_0}$ and we verify the appropriate versions of (1)--(3). To this end, 
we extend the duality to a {\em space-time duality} and obtain an expression for the {\em mixed 
space-time moments} in terms of triples of parameters
\begin{equation}\label{e3076}
{(c_k, d_k, \chi_k)}_{k = j+1, j, \ldots , j+1-L},
\end{equation}
with $L$ being the order of the marginal distribution we consider. 

In the {\em space-time dual}, we work with frozen partition elements which are activated (then 
once and forever) at a present time. Namely we add partition elements marked by a label in 
$[0,\infty]$, which indicates from which time on the mechanisms of the coalescent are activated. 
Before this time, the partition element neither moves nor coalesces. This allows us to characterize 
the finite-dimensional marginals of the forward process. Suppose that we want to study the 
finite-dimensional distributions associated with times $0 \leq t_1 < t_2 < t_3 \ldots < t_n<t$. Then 
we take individuals marked with $0,t-t_n,t-t_{n-1}, \ldots, t-t_1$, consider the test functions in the 
duality relations for the time horizon $t_1,t_2, \ldots, t_n,t$, and form the product. The duality 
relation holds again. Compare with Greven, Sun and Winter~\cite[Corollary 1.20]{GSW}.

In this setting, (1)--(3) turn into claims about the expectation of the duality expression under the 
law of the space-time coalescent, after which the argument proceeds as above.
\end{proof}

\begin{proof}[\bf Proof of cases \textup{(a), (A)}]\label{pr.ofcasesaA}
The limiting transition kernel of the rescaled interaction chain for a given environment degenerates 
to a transition kernel concentrated on the traps. We have
\begin{equation}\label{e3086}
K = \infty, \qquad \lim_{k\to\infty} d_k/c_k = 1.
\end{equation}
We must therefore show that
\begin{equation}\label{e3090}
\lim_{K\to\infty} \CL\left[\nu_\theta^{1,1,2 K\chi}\right] 
= \CL\left[\int_E \theta (\dd u)\delta_{\delta_u}\right].
\end{equation}
Taking the dual representation, we see that as $K\to\infty$ the rate of the $\Lambda$-coalescence 
tends to infinity, implying that the coalescent converges before it jumps, and coalesces into a single 
partition element. The duality relation says that the original McKean-Vlasov process is in a 
mono-type equilibrium, where the type is chosen at random according to $\theta$. The claim now 
follows because for $K(\theta,\cdot) = \int_E \theta (\dd u) \delta_{\delta_u}(\cdot)$ the state 
$\delta_u$ is a {\em trap}, so that the limiting Markov chain is constant for every $\alpha \neq 0$, 
the constant being chosen according to $\theta$ for every realization of the random environment.
\end{proof}

\subsection{Random cluster order}
\label{s.order}

\begin{proof}[\bf Proof of cases \textup{(c), (C2)} and \textup{(d), (C3)[second subcase]}]
In cases (c) and (d), averaging takes place via a law of large numbers and the situation is similar 
to the homogeneous environment, for which the results in Theorem~\ref{T.cluform} are of the same 
type, and it is only the formula for $d$ that changes.

The claim is that the interaction chain, which is a space-time rescaled Markov chain and a 
measure-valued square-integrable martingale, converges to a limit that is a measure-valued 
diffusion and a square-integrable martingale. In \cite[Section 6(b)]{DGV95}, it was pointed out 
how, for the case of the Fleming-Viot process, this convergence reduces to the study of the 
process of conditioned variances along the path, which in turn reduces to showing the following 
asymptotic relations for these objects. Pick $\alpha_1,\alpha_2 \in I$ with $\alpha_2<\alpha_1$, 
and suppose that $\lim_{j\to\infty} k_\alpha(j)/j=\beta(\alpha)$ with $0 \leq \beta(\alpha_1) 
< \beta(\alpha_2) \leq 1$. If the scaled Markov chain is such that
\begin{equation}
\label{Mjscalclustalt}
\lim_{j\to\infty} \var\left(\left\langle f,M^{(j)}_{-k_{\alpha_1}(j)}\right\rangle ~\Big|~
M^{(j)}_{-k_{\alpha_2}(j)}=\theta\right) 
= \frac{\beta(\alpha_2) - \beta(\alpha_1)}{\beta(\alpha_2)}\,\var_\theta(f), \qquad f \in C_b(E,\R),
\end{equation}
with $\beta(\alpha)=1-\alpha$, then by applying the transformation $\beta(\alpha) = \eee^{-s}$ 
the right-hand side turns into the expression $(1-e^{-(s_1-s_2)})\,\var_\theta(f)$. Since
this scales like $(s_1-s_2)\,\var_\theta(f)$ for $s_1 \downarrow s_2$, we see that the standard
Fleming-Viot process $Y(s)_{s \geq 0}$ appears as the scaling limit. Since $s=\log(1/(1-\alpha))$, 
we get the time-scaled Fleming-Viot process $Y(\log(1/(1-\alpha))_{\alpha \in [0,1)}$ (see 
\cite[Section 6]{DGV95}). 

With suitable time transformations, we can also handle the other forms of scaling $j \to k_\alpha(j)$ 
in Definition~\ref{def:clusuniv}. Namely, we have to identify the function $F(\alpha_1,\alpha_2)$
appearing in front of $\var_\theta(f)$ and find the transformation $\alpha = L(s)$ such that 
\begin{equation}
\label{delF}
\Delta F(s_2) = \lim_{s_1 \downarrow s_2} \frac{F(L(s_1),L(s_2))}{s_1-s_2} \equiv 1,
\end{equation}
so that again the standard Fleming-Viot process $(Y(s))_{s \geq 0}$ appears as the scaling limit. 
Since $s=L^{-1}(\alpha)$, we get the time-scaled Fleming-Viot process $(Y(L^{-1}(\alpha))_{\alpha \in I}$.

It was pointed out in \cite[Section 9.3]{GHKK14} how \eqref{Mjscalclustalt} is established for the 
homogeneous hierarchical Cannings process by using the scaling analysis of the coefficients 
$\ud = (d_k)_{k \in \N_0}$. In our case, we need to work with a random sequence 
$(\mu_k\rho_k(\omega))_{k \in \N_0} $ instead of $(\mu_k)_{k \in \N_0}$, where $\rho_k(\omega)$ 
arises from the term $\Lambda=\Lambda^{(\eta,k)}((0,1])(\omega)$ in the following variance formula
\begin{align}\label{a3103}
\int_E \mu_\theta^{c,d,\Lambda}(\dd x)\big(\langle f,x\rangle^2-\langle f,\theta\rangle^2\big)
= \frac{2c}{2c + \lambda\rho_k(\omega) + 2d}\,\var_\theta(\langle f,x\rangle), \qquad f \in C_b(E,\R).
\end{align} 
We thus have to see whether the product (with $\rho_k(\omega) = \rho^{\mathrm{MC}_k(0)}(\omega)$)
\begin{equation}\label{e3108}
\prod_{k=j_1}^{j_2} \frac{2c_k}{2c_k+\lambda_k\rho_k(\omega)+2d_k},
\end{equation}
appearing in the expression for the variance in \eqref{Mjscalclustalt}, does indeed exhibit averaging 
based on the tail triviality of the random sequence $(\rho_k(\omega))_{k \in \N_0}$ (see 
\cite[Eq.~(8.14)]{GHKK14}). 

To that end, we abbreviate
\begin{equation}
\label{e3116}
m_k(\omega)=\frac{\mu_k \rho_k(\omega)+d_k}{c_k},
\ee
consider the relation
\begin{equation}
\label{e2858}
\var\left( \left\langle M^{(j_2)}_{j_1},f\right\rangle \mid M^{(j_2)}_{j_2+1} = \theta\right) 
= \left[\sum_{k=j_1}^{j_2} \frac{d_{k+1}}{c_k} \prod^{j_2}_{l=k+1}
\frac{1}{1+m_l(\omega)}\right]\, \var(\langle \theta,f \rangle)
\end{equation}
and analyse its behaviour as $j\to\infty$ for appropriate choices of $j_1=j_1(j)$ and $j_2=j_2(j)$. 
We must show that, for $\mathbb{P}$ almost all $\omega$, \eqref{e2858} behave asymptotically like the right-hand 
side of \eqref{Mjscalclustalt}, and we must identify the associated $F$, $\Delta F$ and $L$. 

In order to decide how the product scales as $j_2-j_1\to\infty$, we take logarithms to turn this 
into the question whether the sum
\begin{equation}
\label{e3128}
\sum_{k=j_1}^{j_2} \frac{\mu_k\rho_k(\omega)+d_k}{c_k} 
= \sum^{j_2}_{k=j_1} m_k(\omega)
\end{equation}
has a certain scaling behaviour, and we link this to the scaling behaviour of $\mu_k/c_k$ and 
$d_k/c_k$ for $k\to\infty$ (which we know from Theorems~\ref{T.scaleFVpol} and \ref{T.scaleFVexp}) 
to derive the relevant asymptotics. We have to show that this asymptotics does \emph{not} 
depend on $\omega$ and is equal to that with $\rho_k(\omega)$ replaced by its mean $1$. To
achieve the latter, we use the stationarity of $(\rho_k(\omega))_{k \in \N_0}$, plus the fact that it 
has bounded and decaying covariances (recall \eqref{Aprop}--\eqref{fdh:tail}). The key is the 
following lemma.

\begin{lemma}
\label{lem:wlln}
Define $S(j_1,j_2)(\omega) = \sum_{k=j_1}^{j_2} m_k(\omega)$. Then,
\begin{equation}
\lim_{j_2-j_1\to\infty} \frac{S(j_1,j_2)(\omega)}{\mathbb{E}[S(j_1,j_2)(\omega)]} = 1 
\quad \text{ in $\mathbb{P}$-probability}.
\end{equation} 
\end{lemma}

\begin{proof}
Define
\begin{equation}
\chi_k(j_1,j_2) = \frac{\mu_k/c_k}{\sum_{k=j_1}^{j_2} (\mu_k+d_k)/c_k},
\qquad j_1 \leq k \leq j_2.
\end{equation}
Then
\begin{equation}
\frac{S(j_1,j_2)(\omega)}{\mathbb{E}[S(j_1,j_2)(\omega)]} - 1 =  \sum_{k=j_1}^{j_2} 
\chi_k(j_1,j_2) [\rho_k(\omega)-1].
\end{equation}
With the help of Chebyshev's inequality we see that it suffices to show that
\begin{equation}
\lim_{j\to\infty} \sum_{k=j_1}^{j_2} \sum_{l=j_1}^{j_2} \chi_k(j_1,j_2) \chi_l(j_1,j_2)\, 
\mathbb{C}\mathrm{ov}[\rho_k(\omega),\rho_l(\omega)] = 0.
\end{equation}
We have $\mathbb{C}\mathrm{ov}[\rho_k(\omega),\rho_l(\omega)] = C_{|k-l|}$ with 
$\lim_{m\to\infty} C_m = 0$. Since, by our assumptions on $(c_k)_{k\in\N_0}$ and 
$(\mu_k)_{k\in\N_0}$, we have
\begin{equation}
\lim_{j_2-j_1\to\infty} \max_{j_1 \leq k \leq j_2} \chi_k(j_1,j_2) = 0, \qquad
\sum_{j_1 \leq k \leq j_2} \chi_k(j_1,j_2) \leq 1,
\end{equation} 
the claim follows. 
\end{proof}

\begin{remark}
\label{rem:homclus}
{\rm The role of Lemma~\ref{lem:wlln} is to show that the same clustering behaviour occurs 
in the random environment as in the homogeneous environment. We are only able to prove 
convergence in $\mathbb{P}$-probability and not $\mathbb{P}$-a.s. In the prefactor in the 
right-hand side of \eqref{e2858} weighted averages over $j_1,j_2$-dependent sliding windows 
of the random environment appear, which would need to be shown to converge $\mathbb{P}$-a.s. 
It is unclear how to do this, even for an i.i.d.\ random environment.}
\end{remark}

Lemma~\ref{lem:wlln} implies that the term between square brackets in \eqref{e2858} scales like
\begin{equation}
\label{FdHscal1}
\Delta(j_1,j_2) = \sum_{k=j_1}^{j_2} \frac{d_{k+1}}{c_k}
\exp\left[-\sum_{l=k+1}^{j_2} \frac{\mu_l+d_l}{c_l}\right]
\qquad \text{ in $\mathbb{P}$-probability as $j_2-j_1\to\infty$},
\end{equation}
where we use that $\lim_{l\to\infty} (\mu_l+d_l)/c_l = 0$ in all cases of interest. In the remainder of 
the proof, we pick $j_1=k_{\alpha_1}(j)$ and $j_2=k_{\alpha_2}(j)$ with $\alpha_2<\alpha_1$, with 
$k_\alpha(j)$ as in Definition~\ref{def:clusuniv}, and compute the limit of \eqref{FdHscal1} as 
$j\to\infty$. We omit writing $\lfloor\cdot\rfloor$ at places where labels are obviously integer. We 
determine $k_\alpha$ and identify $F$, $L$  (recall the discussion leading up to \eqref{delF}) for the 
different cases, in the order (c), (C2), (d), (C3). Recall that $K_k= \frac{\mu_k}{c_k}$ and $\bar{K}_k
= \frac{\bar{\mu}_k}{\bar{c}_k}$.

\medskip\noindent
{\bf Case (c).} 
Pick $k_\alpha(j)= j+1-\alpha h(j)$ with $h(j)=1/\sqrt{K_j}$, and insert $d_k \sim \sqrt{c_k\mu_k} 
= c_k \sqrt{K_k}$ and $d_{k+1} \sim d_k$, to obtain that \eqref{FdHscal1} scales like
\begin{equation}
\label{FdHscal2}
\Delta(j) = \sum_{k= j+1-\alpha_1/\sqrt{K_j}}^{j+1-\alpha_2/\sqrt{K_j}} 
\sqrt{K_k}\,\,\exp\left[-\sum_{l=k+1}^{j+1-\alpha_2/\sqrt{K_j}} \big(K_l + \sqrt{K_l}\big) \right].
\end{equation}
Putting $x = (j+1-k)\sqrt{K_j}$, and using that $\lim_{k\to\infty} K_k = 0$, $\lim_{k\to\infty} k^2K_k 
= \infty$ and $K_k \sim K_l \sim K_j$ uniformly in $k,l$ in both sums, we get
\begin{equation}
\lim_{j\to\infty} \Delta(j) = \int_{\alpha_2}^{\alpha_1} \dd x\,\,\exp[-(x-\alpha_2)]
= 1 - \exp[-(\alpha_1-\alpha_2)]. 
\end{equation}
Pick $\alpha = L(s) = s$. Then $\Delta F \equiv 1$. Since $s=L^{-1}(\alpha) = \alpha$, this proves 
the claim.
  
\medskip\noindent
{\bf Case (C2)[subcase $\lim_{k\to\infty} k \bar{K}_k = \infty$].} 
Pick $k_\alpha(j)=j+1-\alpha h(j)$ with $h(j)=1/\bar{K}_j$, and insert $d_k \sim \mu_k/(\mu-1)
= \bar{K}_k c_k/(\mu-1)$ and $d_{k+1} \sim \mu d_k$, to obtain that \eqref{FdHscal1} scales like
\begin{equation}
\label{FdHscal3}
\Delta(j) = \frac{\mu}{\mu-1} \sum_{k=j+1-\alpha_1/\bar{K}_j}^{j+1-\alpha_2/\bar{K}_j}
\bar{K}_k\,\,\exp\left[- \frac{\mu}{\mu-1} \sum_{l=k+1}^{j+1-\alpha_2/\bar{K}_j} \bar{K}_l \right].
\end{equation}
Putting $x = (j+1-k)\bar{K}_j$, and using that $\lim_{k\to\infty} \bar{K}_k = 0$, $\lim_{k\to\infty} 
k\bar{K}_k = \infty$ and $\bar{K}_k \sim \bar{K}_l \sim \bar{K}_j$ uniformly in $k,l$ in both sums, 
we get
\begin{equation}
\label{eq:FdHscal3.5}
\lim_{j\to\infty} \Delta(j) = \frac{\mu}{\mu-1} \int_{\alpha_2}^{\alpha_1} \dd x\,\,
\exp\left[- \frac{\mu}{\mu-1}(x-\alpha_2)\right] 
= 1 - \exp\left[- \frac{\mu}{\mu-1}(\alpha_1-\alpha_2)\right]. 
\end{equation}
Pick $\alpha = L(s) = \frac{\mu-1}{\mu}s$. Then $\Delta F \equiv 1$. Since $s=L^{-1}(\alpha) 
= \frac{\mu}{\mu-1}\alpha$, this proves the claim.

\medskip\noindent
{\bf Case (d).} 
Pick $k_\alpha(j)= (1-\alpha)(j+1)$, and insert $d_k \sim M/\sigma_k$, $\sigma_k c_k \sim k/(1-a)$
and $d_{k+1} \sim d_k$, to obtain that \eqref{FdHscal1} scales like
\begin{equation}
\label{FdHscal4}
\Delta(j) = M(1-a) \sum_{k= (1-\alpha_1)(j+1)}^{(1-\alpha_2)(j+1)} 
\frac{1}{k}\,\,\exp\left[-\sum_{l=k+1}^{(1-\alpha_2)(j+1)} \left(K_l + \frac{M(1-a)}{l}\right) \right].
\end{equation}
Putting $x=(j+1-k)/(j+1)$, and using that $\lim_{k\to\infty} k^2K_k = 0$, we get
\begin{equation}
\begin{aligned}
\lim_{j\to\infty} \Delta(j) &= M(1-a) \int_{\alpha_2}^{\alpha_1}
\frac{\dd x}{1-x}\,\,\exp\left[-M(1-a) \int_{\alpha_2}^{x} \frac{\dd y}{1-y} \right]\\
&= M(1-a)\, (1-\alpha_2)^{-M(1-a)} 
\int_{\alpha_2}^{\alpha_1} \dd x\,(1-x)^{-1+M(1-a)}
= 1 - \left(\frac{1-\alpha_1}{1-\alpha_2}\right)^{M(1-a)}.
\end{aligned}
\end{equation}
Pick $\alpha = L(s) = 1-e^{-s/R}$ with $R=M(1-a)$. Then $\Delta F \equiv 1$. Since $s= L^{-1}(\alpha) 
= \log(1/(1-\alpha)^R)$ we get the claim. 

\medskip\noindent
{\bf Case (C2)[subcase $\lim_{k\to\infty} k \bar{K}_k = \bar{N}$].} 
This is the same as case (d) with $M(1-a)$ replaced by $\bar{N}\frac{\mu}{\mu-1}$.

\medskip\noindent
{\bf Case (C3)[second subcase].} 
This is the same as case (d) with $M$ replaced by 1.
\end{proof}

\addcontentsline{toc}{section}{References}
\bigskip

\bibliography{ghkRE}
\bibliographystyle{alpha}

\end{document}